\documentclass[preprint,aihp]{imsart}

\RequirePackage{amsthm,amsmath,amsfonts,amssymb}
\RequirePackage[numbers,sort&compress]{natbib}
\RequirePackage[colorlinks,citecolor=blue,urlcolor=blue]{hyperref}
\RequirePackage{graphicx}

\startlocaldefs

\usepackage{amsthm,amsmath,amsfonts,amssymb,latexsym,amsfonts,amscd,mathrsfs,enumerate,graphicx,subfigure}
\usepackage{tikz}
\usepackage{enumitem}
\usepackage{bbm}

\usetikzlibrary{shapes.geometric, arrows}
\usetikzlibrary{decorations,decorations.markings}
\usetikzlibrary{arrows,shapes,chains}
\numberwithin{equation}{section}

\startlocaldefs
\theoremstyle{plain}
\newtheorem{theorem}{Theorem}[section]
\newtheorem{corollary}[theorem]{Corollary}
\newtheorem{lemma}[theorem]{Lemma}
\newtheorem{proposition}[theorem]{Proposition}

\theoremstyle{definition}
\newtheorem{definition}[theorem]{Definition}
\newtheorem{remark}[theorem]{Remark}

\newtheorem*{remark*}{Remark}

\def\e{{\mathbb E}}
\def\p{{\mathbb P}}
\def\q{{\mathbb Q}}

\def\E{{\mathbb E}}

\def\N{{\mathbb N}}

\def\d{\, \mathrm{d}}

\def\eps{{\varepsilon}}

\def\T{\mathscr T}

\def\cX{\mathcal X}

\def\kp{\mathfrak p}

\newcommand{\abs}[1]{\left\lvert #1 \right\rvert}
\allowdisplaybreaks

\endlocaldefs

\begin{document}

\begin{frontmatter}

\title{From Cannings model to Brownian motion conditioned on local time profile}
\runtitle{Cannings model and conditioned Brownian motion}

\begin{aug}
\author[A]{\inits{}\fnms{Xiaodan}~\snm{Li}\ead[label=e1]{lixiaodan@mail.shufe.edu.cn}},
\author[B]{\inits{}\fnms{Chengshi}~\snm{Wang}\ead[label=e2]{cswang17@fudan.edu.cn}}
\and
\author[C]{\inits{}\fnms{Yushu}~\snm{Zheng}\ead[label=e3]{yszheng666@gmail.com}}
\address[A]{School of Mathematics, Shanghai University of Finance and Economics, Shanghai, China \printead[presep={,\ }]{e1}}

\address[B]{School of Mathematical Sciences, Fudan University, Shanghai, China \printead[presep={,\ }]{e2}}

\address[C]{National Center for Mathematics and Interdisciplinary Sciences, Academy of Mathematics and Systems Science, Chinese Academy of Sciences, Beijing,
   China \printead[presep={,\ }]{e3}}
\end{aug}

\begin{abstract}
We study the scaling limits of genealogical trees arising from Cannings models. Under suitable moment conditions, we show that the rescaled contour and height functions converge to a time change of Brownian motion conditioned on a given local time profile. This conditioned Brownian motion is a self-interacting diffusion constructed independently by Warren--Yor (1998) and Aldous (1998). A key ingredient in our proof is a sequential version of the coming-down-from-infinity property.
\end{abstract}

\begin{abstract}[language=french]
Nous étudions les limites d’échelle des arbres généalogiques issus de modèles de Cannings. Sous des conditions de moments appropriées, nous montrons que les fonctions de contour et de hauteur, après renormalisation, convergent vers un changement de temps d’un mouvement brownien conditionné par un profil de temps local donné. Ce mouvement brownien conditionné est une diffusion auto-interagissante, construite indépendamment par Warren--Yor (1998) et Aldous (1998). Un ingrédient clé de notre preuve est une version séquentielle de la propriété de descente depuis l’infini.
\end{abstract}



\end{frontmatter}






\setcounter{tocdepth}{2}
\tableofcontents

\section{Introduction}\label{Introduction}
Cannings models, first introduced in \cite{cannings74,cannings1975}, form a class of discrete-time population models in which each generation consists of a fixed number of individuals, and reproduction is governed by an exchangeable offspring distribution. This framework includes many classical models such as the Wright--Fisher model and (discrete-time) Moran model, and has played a central role in mathematical population genetics.

In this paper, we study the genealogical trees arising from Cannings models, referred to as Cannings trees. While the scaling limits of various genealogical trees---most notably Galton--Watson trees---have been extensively studied (e.g., \cite{1993Aldous,marckert2003depth,jean2005random}), the scaling behavior of Cannings trees remains largely unexplored. This motivates our present work.

To obtain a richer class of limiting results, we extend the classical Cannings framework to a class of inhomogeneous Cannings models, in which the population size may vary across generations. We assume that the rescaled generation-size profiles converge to a continuous function. Our main result (Theorem~\ref{main}) shows that, under suitable moment assumptions, the rescaled height and contour functions of Cannings trees converge to a time change of Brownian motion conditioned on local time profile. 

\subsection{Setup and main results}\label{sec:setup}
\subsubsection*{Inhomogeneous Cannings model}
An inhomogeneous Cannings model describes the evolution of a finite population over discrete generations $s\in \mathbb{N}_+=\{1,2,\ldots\}$, with population sizes across generations given by a sequence $q=\big(q(s)\big)_{s\in \mathbb{N}_+}$ taking values in $\mathbb{N}=\{0,1,\ldots\}$. We assume that $q$ satisfies the following \textit{extinction closure} condition:
    \[
    q(s) = 0 \quad \text{for all }s\ge  h(q),\quad\text{where }h(q):=\inf\big\{i\in \mathbb{N}_+:q(i)=0\big\}\in \mathbb{N}\cup\{\infty\}.
    \]
We refer to such a sequence $q$ as a \textit{discrete profile}.

The offspring distribution in generation $s$ is described by a random vector
\[\nu^s=\big(\nu_1^s,\ldots,\nu_{q(s)}^s\big),\]
where $\nu_i^s\in \mathbb{N}$ denotes the number of offspring produced by the $i$-th individual in generation $s$.  The sequence \( (\nu^s)_{s\in \mathbb{N}_+} \) is assumed to satisfy the following conditions:
\begin{itemize}
    \item \textit{Mass conservation}: 
    \[
    \sum_{i=1}^{q(s)} \nu_i^s = q(s+1);
    \]
    \item \textit{Exchangeability}: for all permutations \( \pi \in S_{q(s)} \),
    \[
    (\nu_1^s, \ldots, \nu_{q(s)}^s) \overset{d}{=} (\nu_{\pi(1)}^s, \ldots, \nu_{\pi(q(s))}^s);
    \]
    \item \textit{Generation-wise independence}: the vectors \( \nu^s \) are independent across \( s \).
\end{itemize}

When \( q(s) \equiv N \) is constant and $(\nu^s)_{s\in\mathbb{N}}$ are i.i.d., the model reduces to the classical Cannings model. Important examples of the inhomogeneous case include the inhomogeneous Wright--Fisher model and the Galton--Watson model conditioned on generation-size profile.

The genealogical structure induced by an inhomogeneous Cannings model can be naturally represented as a random rooted ordered tree, which we call a \textit{Cannings tree}. A formal definition is given in Section~\ref{sec:Cannings_tree}.

\subsubsection*{Brownian motion conditioned on local time profile}
Let $W$ be a reflected Brownian motion on $\mathbb{R}_+=[0,\infty)$, and denote by $L_t(x)$ its local time at level $x$ up to time $t$. Define the inverse local time at the origin by
\[
\tau_u := \inf\{t \geq 0 : L_t(0) > u\}.
\] \pagebreak

\noindent A function $\ell: \mathbb{R}_+ \to \mathbb{R}_+$ is called a \textit{continuous profile} if it satisfies: 
\begin{itemize}
    \item $\ell(x)=0$ for all $x\ge h(\ell):=\inf\{y> 0:\ell(y)=0\}$;
    \item $\ell$ is continuous on $[0, h(\ell))$ and admits a finite left limit at $h(\ell)$ if $h(\ell)<\infty$. 
\end{itemize}
(Note that $\ell$ may be discontinuous at $h(\ell)$.)

Given a continuous profile \( \ell \), we define \( W^\ell \) as the process \( W \) up to time \( \tau_{\ell(0)} \), conditioned on the local time profile satisfying
\[
L_{\tau_{\ell(0)}}(x) = \ell(x) \quad \text{for all } x \ge 0.
\]
This conditioned process was constructed explicitly by Warren--Yor~\cite{warren98} and characterized via continuum random trees by Aldous~\cite{1998BROWNIAN}. In Section~\ref{sec:limiting}, we recall these constructions and introduce a canonical version of \(W^\ell\). Throughout the paper, we always work with that canonical version.

\subsubsection*{Convergence assumptions and results}
We now consider a sequence of inhomogeneous Cannings models indexed by $n$, with generation-size profiles $q_n=(q_n(s))_{s\in\mathbb{N}_+}$ and offspring vectors $(\nu^{n,s})_{s\in\mathbb{N}_+}$. For $n,s\in\mathbb{N}_+$, set
\[
\sigma_n(s)^2 := \mathbf{1}\{s < h(q_n)\}\,\mathrm{Var}(\nu_1^{n,s}).
\]
The assumptions for our main result are as follows:
\begin{enumerate}[label=(H\arabic*), ref=H\arabic*]
    \item\label{H1} There exist continuous profiles $\ell$ and $\sigma$ with $h(\ell)=h(\sigma)<\infty$ such that
    \begin{enumerate}[label=(\roman*)]
        \item\label{assump1} As $n\to\infty$,
        \begin{align*}
            &\frac{h(q_n)}{n}\to h(\ell),\qquad 
    \frac{q_n(\lfloor n\,\cdot\rfloor)}{n}\to \ell \ \text{uniformly on compact subsets of }(0,h(\ell)),\\
    &\text{and}\quad \sigma_n(\lfloor n\,\cdot\rfloor)\to \sigma \ \text{uniformly on compact subsets of }(0,h(\sigma)).
        \end{align*}
    \item\label{assump2} $\lim\limits_{\delta\rightarrow 0}\sup\limits_{n\in\mathbb{N}_+}\dfrac{1}{n^2}\sum\limits_{s\in\mathbb{N}_+\setminus [\delta n, h(q_n)-\delta n]}q_n(s)=0$;
        \item\label{assump3} Both\quad $\lim\limits_{s\rightarrow 0+}\dfrac{\ell(s)}{\sigma(s)^2}$\quad  and\quad   $\lim\limits_{s\rightarrow h(\ell)-}\dfrac{\ell(s)}{\sigma(s)^2}$\quad exist.
    \end{enumerate}

    \item\label{H2} 
    For every $\delta>0$, there exists $\varepsilon>0$ such that
    \begin{align}
        &\liminf_{n\to \infty}\inf_{\delta n\le s\le h(q_n)-\delta n} \e \bigg[(\nu^{n,s}_1)_2\mathbf{1}\bigg\{\nu_1^{n,s}\le  \frac{n}{(\log n)^{1+\varepsilon}} \bigg\}\bigg] > 0, \tag{H2a}\label{eq:liminf}\\
        &\lim_{n \rightarrow \infty} \ \sup_{\delta n \le s \le h(q_n)-\delta n} \frac{\mathbb{E}\big[(\nu_1^{n,s})^3\big]}{n}=0,\tag{H2b}\label{eq:3rd},
    \end{align}
    where $(m)_k := m(m-1)\cdots(m-k+1)$ denotes the falling factorial.
\end{enumerate}
\begin{remark*}
\begin{itemize}
    \item Condition~\ref{assump2} in \eqref{H1} ensures that the contribution to the rescaled population size from the very early and very late generations---near the beginning and near extinction---is negligible. Combined with \ref{assump1}, this yields the convergence of the rescaled total population size:
\begin{align}\label{eq:total_popuplation}
    \frac{1}{n^2}\sum_{s\in\mathbb{N}_+} q_n(s)
    \;\longrightarrow\;
    \int_{0}^{\infty} \ell(s)\, \mathrm{d}s.
\end{align}
    \item Condition~\ref{assump3} in \eqref{H1} ensures that the function
\[
\ell^\sigma(s):=\mathbf{1}\{s<h(\ell)\}\,\frac{4\ell(s)}{\sigma(s)^2}
\]
defines a continuous profile. As will be seen in Theorem~\ref{main}, the limiting object involves Brownian motion conditioned on the local time profile $\ell^\sigma$. Condition~\ref{assump3} ensures that this conditioning profile has well-defined endpoint values. In many applications, $\sigma_n$ is taken to be constant (and hence so is $\sigma$), in which case \ref{assump3} is automatically satisfied.
\item To clarify the meaning of condition~\eqref{eq:liminf}, note first that
uniformly for $\delta n \le s \le h(q_n)-\delta n$,
\[
    \E(\nu^{n,s}_1)
    = \frac{q_n(s+1)}{q_n(s)}
    = 1+o_n(1).
\]
Thus \eqref{eq:liminf} simply requires that the variance
$\text{Var}(\nu^{n,s}_1)$ retain a non‐negligible contribution from the
regime $\nu^{n,s}_1 \le n/(\log n)^{1+\varepsilon}$.
We also record two simple sufficient conditions for~\eqref{eq:liminf} under the assumption \eqref{H1}:
\begin{enumerate}[label=(a\arabic*)]
    \item\label{item:sufficient} There exists $\varepsilon>0$ such that, for every $\delta>0$,
    uniformly for $\delta n\le s\le h(q_n)-\delta n$,
    \[
        \E\big[(\nu^{n,s}_1)^3\big]
        = o\!\left( \frac{n}{(\log n)^{1+\varepsilon}} \right).\footnote{The sufficiency of \ref{item:sufficient} follows from the two estimates: for any $\delta,\eta>0$,
\begin{align*}
    \liminf_{n\to \infty}\inf_{\delta n\le s\le h(q_n)-\delta n} \e \bigg[(\nu^{n,s}_1)_2\bigg] > 0\quad\text{under \eqref{H1}},\quad
    \text{and}\quad \e \bigg[(\nu^{n,s}_1)_2\mathbf{1}\bigg\{\nu_1^{n,s}>  \frac{n}{(\log n)^{1+\eta}} \bigg\}\bigg]\le \frac{(\log n)^{1+\eta}}{n}\e \big[(\nu^{n,s}_1)^3 \big].
\end{align*}}
    \]

    \item For every $\delta>0$, the family
    $\{\nu^{n,s}_1: n\ge1,\ \delta n\le s\le h(q_n)-\delta n\}$
    is uniformly $L^2$-integrable; that is,
    \[
        \lim_{K\to\infty}
        \sup_{n\ge1}\sup_{\delta n\le s\le h(q_n)-\delta n}
        \E\!\left[(\nu^{n,s}_1)^2\mathbf{1}\{\nu^{n,s}_1>K\}\right] =0.
    \]
\end{enumerate}
See Remark~\ref{rmk:optimal} below for further discussion of conditions \eqref{eq:liminf} and \eqref{eq:3rd}.
\end{itemize}
\end{remark*}

Let \( \mathscr{T}^n \) denote the Cannings tree associated with the model indexed by \( n \), and let \( C_\cdot(\mathscr{T}^n) \) and \( H_\cdot(\mathscr{T}^n) \) be its contour and height functions, respectively (see Section~\ref{sec:geometric} for precise definitions). Our main result is as follows.

\begin{theorem}\label{main}
Assume \eqref{H1}--\eqref{H2}. Define $\alpha(t) :=\frac{1}{4}\int_0^t\sigma(W^{\ell^\sigma}_s)^2\d s$. Let $\alpha^{-1}(t)$ be the inverse of $\alpha(t)$.\footnote{We will verify in Lemma~\ref{lem:alpha} that $\alpha$ is a.s. strictly increasing, so $\alpha^{-1}(t)$ is well-defined.} Then
\begin{align}\label{eq:main_convergence}
    \Big(n^{-1}C_{\lfloor 2n^2 t_1\rfloor}(\mathscr{T}^{n}), n^{-1}H_{\lfloor n^2 t_2\rfloor}(\mathscr{T}^{n})\Big)_{0\le t_1,t_2<\infty} \overset{d}{\longrightarrow} \Big(W_{\alpha^{-1}(t_1)}^{\ell^\sigma}, W^{\ell^\sigma}_{\alpha^{-1}(t_2)}\Big)_{0\le t_1,t_2<\infty}
\end{align}
in the product space \(\big(D[0, \infty)\big)^2\) equipped with the product Skorokhod topology
(where the functions on both sides are extended by \(0\) outside their original domains).
\end{theorem}

It is a standard fact that convergence of contour functions implies convergence of the associated trees in the Gromov-Hausdorff topology (see  \cite[Lemma 2.4]{jean2005random}). Applying this, we obtain the following corollary.

\begin{corollary}
    Assume \eqref{H1}--\eqref{H2}. Let \( \widetilde{\mathscr{T}}^n \) be the geometric tree obtained from \( \mathscr{T}^n \) by assigning length \(1/n\) to each edge. Let \( \mathcal{T}^{\ell,\sigma} \) denote the compact real tree encoded by the time-changed conditioned Brownian motion \( W^{\ell^\sigma}_{\alpha^{-1}(t)} \). Then
    \[\widetilde{\mathscr{T}}^n \overset{d}{\longrightarrow} \mathcal{T}^{\ell,\sigma} \quad \text{under the Gromov-Hausdorff topology.}
    \]
    (See Section~\ref{sec:foundation} below for precise definitions of geometric trees and tree encodings.)
\end{corollary}

\begin{remark}\label{rmk:optimal}
By Aldous' theorem \cite[Theorem 20]{1993Aldous}, the convergence of the contour functions is equivalent to the combination of tightness and finite-dimensional convergence of the associated tree sequence. In \eqref{H2}, condition \eqref{eq:3rd} is in fact equivalent to the finite-dimensional convergence (see Remark \ref{rmk:equivalent});
while condition~\eqref{eq:liminf} is closely related to the tightness. In Appendix~\ref{sec:example}, we present a counterexample in which the offspring distribution slightly violates \eqref{eq:liminf}, leading to a failure of tightness. This indicates a form of sharpness of condition~\eqref{eq:liminf}.

We also remark that the counterexample in Appendix~\ref{sec:example} shows that the condition \eqref{eq:3rd}, although sufficient for finite-dimensional convergence, is not enough to guarantee tightness.
\end{remark}

\begin{remark}
Our convergence result \eqref{eq:main_convergence} extends to the case where the height profile
satisfies $h(q_n)=h(\ell)=\infty$.  
In this case, the height and contour processes are
defined by following the first infinite branch until it escapes to~$\infty$.
The associated rescaled processes then converge weakly under the topology of
uniform convergence on compact subsets of $\mathbb{R}_+$.

The key point is a \emph{local dependence property}:  
for both discrete and continuous conditioned (height or contour)
functions, the law of the path up to the first hitting time of a height $s$
depends only on the generation profile below height~$s$, and coincides with the corresponding law in a finite–height model with the same profile on $[0,s]$. This property is immediate in the discrete case from generation–wise independence, and in the continuous setting, it follows from excursion theory.

This property allows one to define the conditioned process $W^{\ell}$ for profiles
$\ell$ with $h(\ell)=\infty$, and together with the finite-height convergence \eqref{eq:main_convergence}
and a truncation argument yields the convergence in the infinite-height
setting as well.
\end{remark}

\begin{remark}
The convergence results can also be extended to other population models, such as the continuous-time Moran model and the diploid Cannings model (as defined in \cite{mohle2003coalescent}), using a very similar approach.
\end{remark}

\subsection{Related works}\label{sec:related_works}
\subsubsection*{Coalescent theory}
A major focus in the study of Cannings models lies in coalescent theory, which describes the genealogies of finite samples via limiting processes such as Kingman’s coalescent \cite{kingman1982genealogy}. As explained in Remark \ref{rmk:equivalent}, convergence to Kingman’s coalescent is in fact equivalent to the finite-dimensional convergence of the associated tree sequence.

In \cite[Theorem 4]{mohle2000total}, M\"ohle gave the necessary and sufficient condition for convergence to Kingman's coalescent in the classical Cannings model. 
Condition~\eqref{eq:3rd} serves as the analogue of this criterion in the inhomogeneous Cannings setting. Since it is equivalent to the finite-dimensional convergence, it constitutes a necessary condition of the convergence result~\eqref{eq:main_convergence}, as discussed in Remark~\ref{rmk:optimal}.

\subsubsection*{Encoding tree of conditioned Brownian motion}

It is classical that every continuous bridge function naturally encodes a compact real tree (see Section~\ref{sec:encode} for details). For the conditioned Brownian motion \( W^\ell \), Aldous \cite{1998BROWNIAN} derived the finite-dimensional distributions of the corresponding encoding tree. Subsequently, Warren \cite{warren1999result} provided two alternative, direct proofs of this result. This distribution plays a key role in identifying the limiting process in our setting.

\subsubsection*{Processes conditioned on local time profile}

The conditioned Brownian motion \( W^\ell \) is a self-interacting diffusion constructed by Warren and Yor \cite{warren98}, and independently by Aldous \cite{1998BROWNIAN}. Detailed constructions will be presented in Section~\ref{sec:limiting}. Subsequently, Warren \cite{warren05} further established a close connection between this conditioning problem and a stochastic flow of the Bass-Burdzy type.

More recently, Lupu, Sabot, and Tarr\`es \cite{lupu2021inverting}, as well as A\"id\'ekon, Hu, and Shi \cite{aidekon2023stochastic}, have further developed this framework by analyzing perturbed reflecting Brownian motions and Brownian loop soups conditioned on local time profiles\footnote{Specifically, Lupu, Sabot, and Tarr\`es \cite{lupu2021inverting} studied the Gaussian free field conditioned on local time profile, which is equivalent---via Dynkin's isomorphism---to a Brownian loop soup of intensity \(\frac12\), or equivalently, a perturbed reflecting Brownian motion with parameter \(2\), conditioned on local time profile. A\"id\'ekon, Hu, and Shi \cite{aidekon2023stochastic} later generalized this result to arbitrary loop soup intensities (or equivalently, arbitrary parameters for perturbed reflecting Brownian motion).}. A natural extension of our setting is to incorporate immigration into the inhomogeneous Cannings model. In that case, the rescaled contour and height processes will converge to the (time-changed) perturbed reflecting Brownian motion conditioned on its local time profile. 

On the discrete side, the conditional laws of Markov chains and Markovian loop soups given their local time profiles have been studied in~\cite{2016Inverting,lupu2019inverting,li2024inverting}.

\subsection{Proof outline}

Our approach to proving Theorem~\ref{main} builds on the framework developed by Aldous in \cite[Theorem 20]{1993Aldous}, where he showed that for a sequence of random rooted ordered trees, convergence of the contour process is equivalent to the combination of tightness and finite-dimensional convergence of the trees (see Theorem~\ref{thm:convergence} for a modified version of this result adapted to our setting).
In our case, finite-dimensional convergence is obtained via its equivalence with convergence to Kingman’s coalescent (see Proposition \ref{sconvergence}). The main technical challenge lies in proving tightness. 
Once tightness is established, the convergence of the contour process follows, and the joint convergence with the height process is then deduced using a deterministic relation between the two processes from \cite{marckert2003depth} (see Section~\ref{sec:joint}).

We now outline the key ideas used in the proof of tightness. Roughly speaking, tightness of a sequence of trees means that for any \( \varepsilon, \delta > 0 \), there exists a positive integer \( k \) such that for every tree in the sequence, with probability at least \( 1 - \varepsilon \), a uniformly chosen set of \( k \) vertices forms a \( \delta \)-net of the tree—that is, every vertex lies within distance \( \delta \) of at least one of the selected vertices. In the context of the (unrescaled) Cannings trees \( \mathscr{T}^n \), this corresponds to a \( \lfloor \delta n \rfloor \)-net formed by \( k \) uniformly chosen vertices. 

To this end, we analyze the ancestral structure of the tree from top to bottom. We partition the height of the tree into intervals of length \( \lfloor \delta n \rfloor \). For each such interval \(\big[ s\lfloor \delta n \rfloor, (s+1)\lfloor \delta n \rfloor\big) \), we consider the ancestors at height \( s\lfloor \delta n \rfloor \) of all vertices located at height \( (s+1)\lfloor \delta n \rfloor \). These ancestral vertices are referred to as the \textit{\(\delta\)-coalescent vertices} at the \( s \)-th height interval; see Definition~\ref{def:coalescent} for the precise definition.

The union of all \(\delta\)-coalescent vertices forms a \( 2\lfloor \delta n \rfloor \)-net of the tree. Tightness then follows by showing that, with high probability, these vertices lie within distance \( \lfloor \delta n \rfloor \) of the \( k \) uniformly chosen vertices. This relies on two key facts, which hold with high probability:
\begin{enumerate}
    \item The total number of \(\delta\)-coalescent vertices is bounded by a universal constant (Proposition~\ref{prop:coalescent});
    \item Each \(\delta\)-coalescent vertex has offspring of size \( \Theta(n^2) \) within distance \( \lfloor \delta n \rfloor \) (Proposition~\ref{prop:size}).
\end{enumerate}

We highlight that the first fact relies on a sequential version of the coming-down-from-infinity (CDFI) property, established in Theorem~\ref{thm:CDFI}. To this end, we introduce a \textit{coalescent process}, a Markov process that tracks the number of ancestral lineages backward from height $(s+1)\lfloor \delta n \rfloor$. In particular, at height $s\lfloor \delta n\rfloor$ this count is exactly the number of $\delta$-coalescent vertices at the $s$-th height interval. Our assumption~\eqref{eq:liminf} is used in the proof to control the merge rates of these lineages. Heuristically, under \eqref{eq:liminf} the process contracts at a uniform multiplicative rate on mesoscopic windows: on average, over each window of length
\(
\frac{n}{(\log n)^{1+\varepsilon}},
\)
the lineage count shrinks by a factor $(\log n)^{-c}$ (see Proposition \ref{prop:rely1}). So over a time horizon $\delta n$, the cumulative shrinkage is
\[
(\log n)^{-c\delta (\log n)^{1+\varepsilon}}
=\exp\!\big(-c\delta (\log n)^{1+\varepsilon}\log\log n\big)\ll 1/n.
\]
This results in a rapid decrease in the number of lineages—from $\Theta(n)$ down to a constant—within $\delta n$ steps. By contrast, in the counterexample of Appendix~\ref{sec:example} (where \eqref{eq:liminf} fails), a contraction occurs only every
\(
\frac{n}{(\log n)^{1-\varepsilon}}
\)
steps; over the same horizon $\delta n$ the cumulative shrinkage is
\[
(\log n)^{-c'\delta(\log n)^{1-\varepsilon}}
=\exp\!\big(-c'\delta(\log n)^{1-\varepsilon}\log\log n\big)\gg 1/n,
\]
which is insufficient to reduce the lineage count from order $\Theta(n)$ to $O(1)$. This leads to non-convergence of the rescaled contour functions in this example (see Proposition \ref{thm:example}).

For the second fact, we analyze the ancestral structure of all vertices at height \( (s+1)\lfloor \delta n \rfloor \). We show that, with high probability, there exists a stretch of height—around \( (s + \tfrac{1}{2})\lfloor \delta n \rfloor \) and of length \( \Theta(n) \)—during which the number of their ancestors remains constant and is bounded by a constant independent of $n$.
This stability implies that a large proportion of other vertices at these heights are likely to coalesce onto these few persistent ancestral lines within this time window. These vertices already account for \( \Theta(n^2) \) descendants of the \textit{\(\delta\)-coalescent vertices} at the \( s \)-th height interval. See Sections \ref{sec:pf_size}--\ref{sec:pf_size2} for details.

\subsection{Organization of the paper}

The remainder of the paper is organized as follows. 
In Section~\ref{sec:preliminaries}, we present several important preliminaries, including the formal definition of Cannings trees, Aldous' theorem on the convergence of contour functions, and the construction of the conditioned Brownian motion \( W^\ell \). 
Section~\ref{sec:tightness} introduces the coalescent process and derives some of its basic properties, including the CDFI property. 
Section~\ref{sec:tight} is devoted to proving the tightness of the rescaled trees. 
Finally, in Section~\ref{sec:proof}, we complete the proof of the main theorem by establishing the finite-dimensional convergence and the joint convergence of the contour and height functions.
Appendix~\ref{sec:example} provides a counterexample demonstrating the near-sharpness of condition \eqref{eq:liminf}.

\section{Preliminaries}\label{sec:preliminaries}
\subsection{Cannings tree}\label{sec:Cannings_tree}

Let
\[
\mathcal{U} := \bigcup_{n \geq 0} (\mathbb{N}_+)^n,
\]
where by convention $(\mathbb{N}_+)^0 := \{\emptyset\}$. Every rooted ordered tree can be canonically represented as a prefix-stable finite subset of $\mathcal{U}$. See, for example, \cite[Section 1.1]{jean2005random} for details.

Using this formalism, we now give a precise definition of the \textit{Cannings tree} associated with an inhomogeneous Cannings model. Every $u\in\mathcal U$  can be represented as $u = (u_1, \ldots, u_m) \in (\N_+)^m$ for some $m\in\N$, with the convention that $m=0$ corresponds to $u=\emptyset$; we write $|u| = m$ for its length. For $u\in\mathcal{U}$ and $j \in \mathbb{N}_+$, we denote by $uj$ the concatenated sequence $(u_1, \ldots, u_m, j)$. For $u, v \in \mathcal{U}$, we write \(u \prec v\) if \(u\) precedes \(v\) in the lexicographical order on $\mathcal{U}$.

\begin{definition}[Cannings tree]\label{def:CanningsTree}
Given a population profile $q = (q(s))_{s \in \mathbb{N}_+}$ and offspring vectors $\nu^s = (\nu_1^s, \ldots, \nu_{q(s)}^s)$, we define the associated Cannings tree $\mathscr{T}$ as a subset of $\mathcal{U}$ constructed recursively as follows:
\begin{itemize}
    \item $\emptyset\in \mathscr{T}$;
    \item Suppose all elements \( u \in \mathscr{T} \) with \( |u| = s \) have been defined. Label them in lexicographic order as \( u^1 \prec u^2 \prec \cdots  \prec u^{q(s)} \). Then the elements \( u \in \mathscr{T} \) with \( |u| = s + 1 \) are defined as
    \[
    \left\{ u^i j : 1 \le i \le q(s),\, 1 \le j \le \nu^s_i \right\},
    \]
    where we set \( q(0) := 1 \) and \( \nu^0_1 := q(1) \).
\end{itemize}
The resulting set $\mathscr{T}$ is a prefix-stable subset of $\mathcal{U}$ representing a random rooted ordered tree.
\end{definition}
Intuitively, this construction starts by introducing an artificial ancestor, the root node $\emptyset$, which serves as the common parent of all individuals in generation $1$. Each vertex at height $s$ (i.e., with label length $s$) represents an individual in generation $s$, and the corresponding entry $\nu^s_i$ in the offspring vector determines the number of children of the $i$-th individual (in lexicographic order) in generation $s$. These children are then represented by appending $1,2,\ldots,\nu^s_i$ to the parent's representation.
The tree structure constructed in this way encodes the entire genealogy of the population.

\subsection{Foundations of tree encoding and convergence}\label{sec:foundation}

This section introduces the framework for encoding trees by height and contour functions, and then provides necessary and sufficient conditions for the convergence of contour functions.
\subsubsection{Geometric trees}\label{sec:geometric}

In this section, we extend the notion of rooted ordered trees by assigning a positive length to each edge, and thereby introduce the concept of a \emph{rooted ordered geometric tree}. This enriched structure is better suited for studying convergence, as it allows for the rescaling of edge lengths.

\medskip

A rooted ordered geometric tree is defined as a pair \((T, w)\), where:
\begin{itemize}
    \item \(T\) is a rooted ordered tree with vertex set \(V(T)\) and edge set \(E(T)\);
    \item \( w: E(T) \to (0,\infty)\) is a function that assigns a positive length \(w(e)\) to each edge \(e \in E(T)\). 
\end{itemize}
The function \(w\) naturally induces a metric on \(V(T)\) defined by
\[
d(v,v') := \sum_{e \in [\![v,v']\!]} w(e),\quad v,v' \in V(T),
\]
where \([\![v,v']\!]\) denotes the unique path connecting \(v\) and \(v'\) in \(T\). In particular, with $\rho$ denoting the root of $T$, \[\abs{v}:=d(\rho,v)\] represents the height (or depth) of \(v\) measured in terms of edge lengths.

\subsubsection*{Height function}
The \emph{height function} \(H(T)\) of a rooted ordered tree \((T, w)\) records the height of each vertex according to the lexicographic order on \(T\). Specifically, we can rewrite
\[
V(T) = \{v_0, v_1, \dots, v_{k-1}\} \quad \text{with} \quad \rho=v_0 \prec v_1 \prec \cdots \prec v_{k-1}.
\]
The height function of \((T, w)\) is defined by
\[
H_n(T) := \abs{v_n}, \quad 0 \le n < k.
\]

\subsubsection*{Contour function}
The \emph{contour function} \(C(T)\) encodes the tree via a depth-first traversal that respects the order on \(T\). Imagine a particle starting at the root \(\rho\) and exploring \(T\) ``from left to right'' according to this order. The particle takes one time unit to traverse each edge, and since each edge is crossed twice (once forward and once backward), the total number of steps in the traversal is \(2(|V(T)|-1)\), where $|V(T)|$ is the number of vertices in $V(T)$. For each integer \(n \in \{0, 1, \dots, 2(|V(T)|-1)\}\), let \(f(n)\) denote the vertex visited at step \(n\). Then the contour function is defined by
\[
C_n(T) := \abs{f(n)}, \quad 0 \le n \le 2(|V(T)|-1).
\]
\begin{remark}
    A rooted ordered tree can be viewed as a special case of a rooted ordered metric tree in which all edge lengths are equal to one (i.e., \(w(e) \equiv 1\) for all \(e \in E(T)\)). In this case, the classical definitions of height and contour functions, typically defined by the number of edges,  are recovered.
\end{remark}

\subsubsection{Encoding compact real trees by bridge functions}\label{sec:encode}

Following \cite[Section 2.1]{jean2005random}, compact real trees are defined as compact metric spaces satisfying the standard tree properties. A canonical way to construct such trees, which simultaneously endows them with a root and an ordered structure, is through encoding functions drawn from the space  
\[
C_{\rm bridge} := \{g \in C[0,1] : g(0) = g(1) = 0, \ g(x) \geq 0 \text{ for } x \in [0,1]\}.
\]

\noindent Specifically, given \( g \in C_{\rm bridge} \), we define a pseudo-metric
\[
d_g(s,t) := g(s) + g(t) - 2\inf_{r \in [s \wedge t, s \vee t]} g(r), \quad \text{for } s,t \geq 0,
\]
which induces an equivalence relation:
\[
s \sim t \quad \text{if and only if} \quad d_g(s,t) = 0.
\]
The quotient space \(\mathcal{T}_g := [0,1]/\sim\), equipped with the inherited metric \(d_g\), is a compact real tree by \cite[Theorem~2.2]{jean2005random}.
Let \(p_g : [0,1]\to \mathcal{T}_g\) be the canonical projection.
We regard \(\mathcal{T}_g\) as rooted at \(p_g(0)\) and equip it with an order by declaring \(s_1 \prec_g s_2\) if and only if
\[
p_g^{-1}(s_1) < p_g^{-1}(s_2),
\qquad
p_g^{-1}(s) := \min\{t\in[0,1] : p_g(t)=s\}.
\]

Intuitively, this construction can be understood by visualizing the graph of \( g \) and identifying points that are connected through horizontal compression. This process ``flattens'' the graph into a tree structure.

\begin{remark}
Le Gall's encoding of compact real trees in \cite{jean2005random} uses continuous functions with compact support. Our encoding through \( C_{\rm bridge} \) functions is essentially equivalent because, in Le Gall's setting, the function \( g(x) \) and its rescaled version \( g(cx) \) encode the same real tree. Thus, without loss of generality, we can always assume \( g \in C_{\rm bridge} \).    
\end{remark}

We now consider a random rooted ordered real tree \((\mathcal{T}_g,d_g,\preceq_g)\) encoded by a random function \(g \in C_{\rm bridge}\).  
Its finite-dimensional distributions are captured by the uniform \(k\)-point subtrees defined below.

For any \(t_1,\ldots, t_k\in [0,1]\), let \(\mathcal{T}_g(t_1,\ldots, t_k)\) be the subtree of \(\mathcal{T}_g\) generated by the points
\(\{p_g(t_1),\ldots, p_g(t_k)\}\):
\[
\mathcal{T}_g(t_1,\ldots, t_k) 
= \bigcup_{i=1}^k [\![\rho,\, p_g(t_i)]\!],
\]
where \( [\![x,y]\!] \subset \mathcal{T}_g \) denotes the unique geodesic path between \(x,y\in\mathcal{T}_g\).
This subtree is naturally viewed as a rooted ordered geometric tree:  
its vertex set consists of the root \(\rho\), the sampled points \(p_g(t_1),\dots,p_g(t_k)\), and all branch points along their geodesics;  
edges are the geodesic segments connecting consecutive vertices, endowed with their metric lengths;  
and the order is inherited from \(\preceq_g\).

\begin{definition}[Uniform \(k\)-point subtree induced by \(g\)]\label{def:k-subtree}  
Let \(k\ge 1\), and let \(U_1,\dots,U_k\) be i.i.d.\ \(\mathrm{Uniform}[0,1]\), independent of \(g\).  
The uniform \(k\)-point subtree of \(\mathcal{T}_g\) is
\[
\mathcal{T}_{g,k}:=\mathcal{T}_g(U_1,\ldots,U_k).
\]
\end{definition}

\begin{remark}\label{treeiso}
Suppose \(g, \widetilde{g} \in C_{\rm bridge}\) differ only by a strictly increasing time change; that is, there exists a strictly increasing homeomorphism
\(\phi:[0,1]\to[0,1]\) such that
\(g(t)=\widetilde{g}\big(\phi(t)\big)\) for all \(t\in[0,1].\)
Then the map
\[
\Psi:\mathcal{T}_g \longrightarrow \mathcal{T}_{\widetilde{g}},\qquad
\Psi\big(p_g(t)\big) := p_{\widetilde{g}}\big(\phi(t)\big),
\]
defines an isomorphism between the rooted ordered real trees \(\mathcal{T}_g\) and \(\mathcal{T}_{\widetilde{g}}\).
In particular, for any $t_1,\ldots, t_k\in [0,1]$, the subtrees
\(\mathcal{T}_g(t_1,\ldots, t_k)\quad\text{and}\quad \mathcal{T}_{\widetilde{g}}\big(\phi({t_1}),\ldots, \phi({t_k})\big)\)
are isomorphic.
\end{remark}

\subsubsection{Convergence of contour functions}  
In this part, we present a fundamental criterion, originally due to Aldous, which characterizes the weak convergence of contour functions in terms of tightness and finite-dimensional convergence of the tree sequence.

We introduce a metric on the space of rooted ordered geometric trees with \(k\) distinguished vertices as follows.  
For each \(k \in \mathbb{N}_+\) and \(i=1,2\), let \(T^i\) be rooted ordered trees and 
\(\mathbf{V}^i\) be a multiset of \(k\) vertices in \(T^i\).
Let \((V^i_{(1)},\dots,V^i_{(k)})\) denote the non-decreasing rearrangement of \(\mathbf{V}^i\) with respect to the tree order \(\prec\).
We then define
\begin{align}\label{fidis}
\begin{aligned}
    d\big((T^1,\mathbf{V}^1), (T^2,\mathbf{V}^2)\big)
    :=\sum_{i=1}^k\Bigg(\Big|\, d_{T^1}\big(b^1_{(i-1)}, V^1_{(i)}\big) 
    - d_{T^2}\big(b^2_{(i-1)}, V^2_{(i)}\big)\,\Big|+\Big|\, d_{T^1}\big(V^1_{(i)}, b^1_{(i)}\big) 
    - d_{T^2}\big(V^2_{(i)}, b^2_{(i)}\big)\,\Big|\Bigg),    
\end{aligned}
\end{align}
where \(b^i_{(0)} = b^i_{(k)}\) is the root of \(T^i\), and for \(1 \le j \le k-1\), 
\(b^i_{(j)}\) denotes the branch point of \(V^i_{(j)}\) and \(V^i_{(j+1)}\) in \(T^i\).
It is straightforward to check that \(d\) induces a metric on rooted ordered geometric trees with \(k\) distinguished vertices, modulo rooted order-preserving isometries.
\begin{theorem}\label{thm:convergence}
Let \((\mathcal{T}, d) = (\mathcal{T}_g, d_g)\) be a random rooted ordered real tree encoded by a random \(g \in C_{\rm bridge}\).  
Consider a sequence \((\mathscr{T}^n : n \ge 1)\) of random ordered discrete trees with \(a_n := |V(\mathscr{T}^n)| \to \infty\), and suppose that the maximal edge length of \(\mathscr{T}^n\) converges to \(0\) in distribution.  
Then the following statements are equivalent:
\begin{enumerate}[label=(\roman*), ref=\roman*]
    \item\label{prop:convergence1} The contour functions
    \[
        C_{\lfloor2(a_n-1)\,\cdot\,\rfloor}(\mathscr{T}^n)\ \overset{d}{\longrightarrow}\  g\quad\text{in \(D[0,1]\).}
    \] 
    \item  For each \(1 \le k \le a_n\), let \((V^n_1, \dots, V^n_{a_n})\) be a uniform random ordering of the vertices of \(\mathscr{T}^n\), and denote by \(\mathscr{T}^n_k\) the rooted ordered subtree spanned by the root \(\rho\) and \(\{V^n_1,\dots,V^n_k\}\). Then it holds that:
    \begin{itemize}
        \item \textit{Tightness}: 
        for any \(\delta > 0\),
        \[
            \lim_{k \to \infty} \limsup_{n \to \infty} \mathbb{P}(\Delta(n,k) > \delta) = 0,
        \]
        where
        \begin{align}\label{def:Delta}
            \Delta(n,k) := \max_{v \in V(\mathscr{T}^n)} \min_{u \in V(\mathscr{T}^n_k)} d(v,u).
        \end{align}
        
        \item \textit{Finite-dimensional convergence}: for every fixed \(k \in \mathbb{N}_+\),
        \begin{equation}\label{finitedistribution}
            \big(\mathscr{T}^n_k, \{V^n_1,\dots,V^n_k\}\big) \overset{d}{\longrightarrow} \big(\mathcal{T}_{g,k}, \{p_g(U_1),\dots,p_g(U_k)\}\big)
        \end{equation}
        under the distance \eqref{fidis}, where \(\{V^n_1,\dots,V^n_k\}\) and \(\{p_g(U_1),\dots,p_g(U_k)\}\) are considered as multisets.
    \end{itemize}
\end{enumerate}
\end{theorem}

\begin{remark}
Aldous originally proved this equivalence in \cite[Theorem~20]{1993Aldous} for the case where the limiting tree is a compact binary real tree encoded by an excursion. In that setting, his notion of finite-dimensional distributions differs from \eqref{finitedistribution}.  
Nevertheless, Theorem~\ref{thm:convergence} is already implicit in Aldous’ proof. Indeed, the first step of his proof shows that his definition is equivalent to requiring that, for every \(k \ge 1\),
\begin{align}\label{eq:aldous}
\begin{aligned}         &\big(C_{U^n_{(1)}}(\mathscr{T}^n),
    \inf_{U^n_{(1)}<t<U^n_{(2)}} C_{t}(\mathscr{T}^n),
   C_{U^n_{(2)}}(\mathscr{T}^n), \inf_{U^n_{(2)}<t<U^n_{(3)}} C_{t}(\mathscr{T}^n),\ldots,
   C_{U^n_{(k)}}(\mathscr{T}^n)\big)\\
    &\overset{d}{\longrightarrow}
    \big(g(U_{(1)}), \inf_{U_{(1)}<t<U_{(2)}} g(t), g(U_{(2)}),
    \inf_{U_{(2)}<t<U_{(3)}} g(t), \ldots, g(U_{(k)})\big),    
\end{aligned}
\end{align}
where \((U^n_{1},\ldots,U^n_{k})\) is a random vector sampled uniformly without replacement from \(\{0,\ldots,2(a_n-1)\}\), with \(U^n_{(i)}\) denoting its order statistics, and \(U_{(1)},\ldots,U_{(k)}\) are the order statistics of \(k\) i.i.d. uniform \((0,1)\) random variables.  
By \cite[Lemma~12]{1993Aldous}, \eqref{eq:aldous} is readily seen to be equivalent to \eqref{finitedistribution}. This is the only step in his argument that relies on the assumption of a binary limiting tree; the remainder of the proof proceeds without invoking it.
In our setting, as the root of the limiting tree is not necessarily binary, the formulation in \eqref{finitedistribution} is better suited.
\end{remark}

\subsection{Conditioned Brownian motion and its subtree distribution}\label{sec:limiting}

Recall the notion of a continuous profile introduced in Section~\ref{Introduction}.
In this section, we first construct a canonical version of \(W^\ell\) based on Aldous' construction \cite{1998BROWNIAN}. 
We then study the laws of the uniform \(k\)-point subtrees induced by \(W^\ell\) and by the time-changed  process \(W^{\ell^\sigma}_{\alpha^{-1}(t)}\). 
These results will be used in the proof of finite-dimensional convergence in Section~\ref{sec:fd}.

\subsubsection{Canonical version of \(W^\ell\)}

Recall that \(L_t(x)\) denotes the local time of \(W\) at level \(x\) up to time \(t\), and that \(\tau_u\) represents the inverse local time at the origin. For any non-negative measurable function $f$ on $\mathbb{R}_+$, we write
\[I(f):=\int_0^\infty f(x)\d x.\]

In \cite[Section~4]{1998BROWNIAN}, Aldous derived the law of the reflected Brownian bridge on the time interval \([0,1]\) conditioned on its local time profile \(L_1=\ell\), where \(\ell\) may lie in a broad class of functions that includes all continuous profiles.  
In our notation, this is exactly the law of \(W^\ell\) for continuous profiles $\ell$ satisfying \(h(\ell)<\infty\) and \( I(\ell) = 1\).

We now explain how this conditional law is described. 
Let \(\mathcal{T}^\ell\) denote the rooted ordered real tree encoded by \(W^\ell\), and let \(\mathcal{T}^\ell_k\) denote its uniform \(k\)-point subtree. 
Aldous~\cite{1998BROWNIAN} and, subsequently, Warren~\cite{warren1999result} determined the law of \(\mathcal{T}^\ell_k\), which will be described in detail later in Section~\ref{sec:time_change}.  
Since the family \(\{\mathcal{T}^\ell_k\}_{k \ge 1}\) is consistent in \(k\), the correspondence theorem of Aldous~\cite[Theorem~15]{1993Aldous}\footnote{A direct extension of the correspondence theorem from excursion functions to bridge functions applies here.} yields a version of \(W^\ell\).  
Moreover, this version of the conditional law is continuous in $\ell$ in a suitable sense as shown in \cite[Construction 1]{1998BROWNIAN}.

Next, we extend Aldous’ version of \(W^\ell\) to all continuous profiles \(\ell\) with \(h(\ell)<\infty\) (i.e., without the restriction \(I(\ell)=1\)).  
This follows from the scaling invariance of reflected Brownian motion and its local times.

\begin{lemma}\label{lem:scale}
For any \(c,u>0\),
\[
\big(W_t,\ L_{\tau_u}(x): 0\le t\le \tau_u,\ x\ge 0\big)
\ \overset{d}{=}\ 
\big(c^{-1} W_{c^2 t},\ c^{-1} L_{\tau_{cu}}(cx): 0\le t\le c^{-2}\tau_{cu},\ x\ge 0\big).
\]
\end{lemma}

\begin{definition}[Canonical version of $W^\ell$]\label{def:W-ell}
Let $\ell$ be a continuous profile with $h(\ell)<\infty$.
Set
\[
\tilde{\ell}(s):= I(\ell)^{-\frac12}\,\ell\big(I(\ell)^{\frac12}s\big),\qquad s\ge 0,
\]
so that \(I(\tilde{\ell})=1\).  
Let \(W^{\tilde{\ell}}\) denote the version of the conditioned reflected Brownian motion associated with the normalized profile \(\tilde{\ell}\) constructed above.  
We then define
\[
W^\ell_t:= I(\ell)^{\frac12}\, W^{\tilde{\ell}}_{I(\ell)^{-1}t}\,,\qquad 0\le t\le I(\ell).
\]
By Lemma~\ref{lem:scale}, this defines a version of conditional law, which will serve as our \emph{canonical version} of \(W^\ell\) throughout the paper.
\end{definition}

Many properties that hold in the case \(I(\ell)=1\) (cf.~\cite[Construction~1]{1998BROWNIAN}) extend to this general setting, including continuity in \(\ell\) and the following:

\begin{proposition}\label{prop:aldous1}
\begin{enumerate}[label=(\roman*), ref=\roman*]
    \item\label{item:aldous1} For any continuous profile \(\ell\) with \(h(\ell)<\infty\), the local time profile of \(W^\ell\) is \(\ell\); that is, a.s. for every non-negative measurable \(f\) on \(\mathbb{R}_+\),
\begin{align}\label{eq:aldous1}
    \int_{0}^{I(\ell)} f\big(W^\ell_t\big)\,\mathrm{d}t \ = \ \int_{0}^{\infty} \ell(s)\,f(s)\,\mathrm{d}s.
\end{align}
In particular, the total time duration of $W^\ell$ is $I(\ell)$.
\item\label{item:aldous2} The process $W^\ell$ is a bridge from $0$ to $0$ on the time interval $[0,I(\ell)]$.
\end{enumerate}
\end{proposition}

We now extend the associated tree notation to this setting.

\begin{definition}[Real tree encoded by \(W^\ell\)]\label{def:W^ell}
Let \(\ell\) be a continuous profile with \(h(\ell)<\infty\).  
Since \(W^\ell_{I(\ell)t} \in C_{\mathrm{bridge}}\) by Proposition \ref{prop:aldous1} \eqref{item:aldous2}, we define \(\mathcal{T}^\ell\) to be the rooted ordered real tree encoded by \(W^\ell_{I(\ell)t}\).  
We denote by \(\mathcal{T}^\ell_k\) its uniform \(k\)-point subtree.
\end{definition}

\begin{remark}[Warren--Yor's construction]
In \cite{warren98}, Warren and Yor provided an explicit construction of \(W^\ell\) for continuous profiles~\(\ell\) satisfying \(\ell(0) > 0\) and \(h(\ell) < \infty\).  
Roughly speaking, they introduced a process called \emph{Brownian burglar}, which is obtained through a suitable time-space transformation of a Brownian path.  
A version of \(W^\ell\) can then be constructed from the Brownian burglar by applying another time-space transformation.  

In particular, this version of \(W^\ell\) is continuous in~\(\ell\) in the following sense (cf.~\cite[Proposition~5.5]{li2024inverting}):  
if \((\ell^n)\) and $\ell$ are continuous profiles with $
h(\ell_n)\rightarrow h(\ell)
$ and $\ell_n\rightarrow \ell$ uniformly on compact subsets of $(0,h(\ell))$,
then  
\[
    W^{\ell^n} \overset{d}{\longrightarrow}\, W^{\ell} 
    \qquad \text{in } C[0,\infty),
\]
where each process is extended by~\(0\) outside its original domain.
In particular, this continuity allows the definition of \(W^\ell\) to be extended to continuous profiles with \(\ell(0)=0\) by approximation.

By comparing the continuity of the canonical version \(W^\ell\) with respect to \(\ell\), one can readily deduce that the two versions coincide for all continuous profiles~\(\ell\) with \(h(\ell) < \infty\).
Hence, the Warren--Yor's construction gives an explicit realization of \(W^\ell\).
\end{remark}

\subsubsection{Distribution of $\mathcal{T}^\ell_k$}\label{sec:time_change}

In this part, we describe the distribution of the uniform \(k\)-point subtree \(\mathcal{T}^\ell_k\). 
We recall Aldous' construction and then reformulate it in terms of a piecewise Kingman's representation. The latter formulation will play a key role in Section~\ref{sec:fd}. In particular, it makes explicit that the finite-dimensional convergence of the tree sequence is equivalent to the convergence of the associated coalescent dynamics to Kingman's coalescent.

\subsubsection*{Aldous' construction via a coalescent process}
In \cite{1998BROWNIAN}, the uniform \(k\)-point subtree \(\mathcal{T}^\ell_k\) is described via a coalescent process of \(k\) particles. Fix a continuous profile \(\ell\) with \(h(\ell)<\infty\) and \(I(\ell) = 1\).  
Think of ``time'' \(t\) running downward from \(\infty\) to \(0\).  
Each of $k$ particles is born at independent random times with probability density $\ell(x)\d x$.  
As time decreases, clusters of particles merge according to the following rule: in the interval \([t,t-\mathrm{d}t]\), each pair of existing clusters coalesces into one cluster with probability \(\frac{4}{\ell(t)}\,\mathrm{d}t\).  
At time \(t=0\), all remaining clusters merge into a single cluster.

The coalescent dynamics described above determine a rooted unordered real tree with \(k\) leaves.  
The \(k\) birth times specify the heights of the leaves, and each merge event at height \(t>0\) creates a branching point at height \(t\).  
The root is located at height \(0\).

To obtain an ordered tree, we assign a left-right ordering at each branching point.  
Whenever two clusters merge at height \(t>0\), one is designated as the left subtree and the other as the right subtree, with the choice made uniformly at random.  
At time \(t=0\), when the final merge occurs, a uniformly random ordering is again assigned.  
This procedure yields a rooted ordered real tree. Aldous \cite{1998BROWNIAN} showed that the law of this tree coincides with the law of the uniform \(k\)-point subtree \(
\mathcal{T}^\ell_k\).

For a general continuous profile \(\ell\) with \(h(\ell)<\infty\), one easily checks that the law of \(\mathcal{T}^\ell_k\) (recall Definition \ref{def:W^ell}) is given by the same coalescent construction, except that the birth density \(\ell(x)\,\mathrm{d}x\) is replaced by its normalized form \(\frac{\ell(x)}{I(\ell)}\,\mathrm{d}x\).

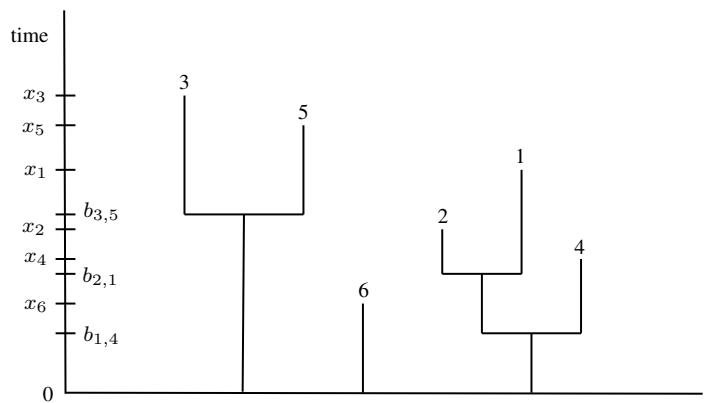
\begin{figure}
    \centering
\tikzset{every picture/.style={line width=0.75pt}} 

\begin{tikzpicture}[x=0.75pt,y=0.75pt,yscale=-0.75,xscale=1]

\draw    (140,300) -- (461.33,300.67) ;
\draw    (139.33,42.67) -- (140,300.67) ;
\draw    (200,100) -- (200,180) ;
\draw    (260,120) -- (260,180) ;
\draw    (200,180) -- (260,180) ;
\draw    (370,150) -- (370,220) ;
\draw    (330,220) -- (370,220) ;
\draw    (400,210) -- (400,260) ;
\draw    (230,180) -- (229.33,299.33) ;
\draw    (375,260) -- (375,300) ;
\draw    (290,240) -- (290,300) ;
\draw    (135,100) -- (145,100) ;
\draw    (135,120) -- (145,120) ;
\draw    (135,180) -- (145,180) ;
\draw    (135,190) -- (145,190) ;
\draw    (135,150) -- (145,150) ;
\draw    (135,210) -- (145,210) ;
\draw    (135,260) -- (145,260) ;
\draw    (135,220) -- (145,220) ;
\draw    (135,240) -- (145,240) ;
\draw    (330,190) -- (330,220) ;
\draw    (350,220) -- (350,260) ;
\draw    (350,260) -- (400,260) ;

\draw (195.8,83.93) node [anchor=north west][inner sep=0.75pt]   [align=left] {{\footnotesize 3}};
\draw (255.8,103.93) node [anchor=north west][inner sep=0.75pt]   [align=left] {{\footnotesize 5}};
\draw (326.2,174.2) node [anchor=north west][inner sep=0.75pt]   [align=left] {{\footnotesize 2}};
\draw (365.27,133.53) node [anchor=north west][inner sep=0.75pt]   [align=left] {{\footnotesize 1}};
\draw (395.27,195) node [anchor=north west][inner sep=0.75pt]   [align=left] {{\footnotesize 4}};
\draw (286.33,224.23) node [anchor=north west][inner sep=0.75pt]   [align=left] {{\footnotesize 6}};
\draw (127,293.67) node [anchor=north west][inner sep=0.75pt]   [align=left] {{\footnotesize 0}};
\draw (117,94.15) node [anchor=north west][inner sep=0.75pt]  [font=\footnotesize]  {$x_{3}$};
\draw (116.5,116.15) node [anchor=north west][inner sep=0.75pt]  [font=\footnotesize]  {$x_{5}$};
\draw (117.5,235.65) node [anchor=north west][inner sep=0.75pt]  [font=\footnotesize]  {$x_{6}$};
\draw (116.5,183.65) node [anchor=north west][inner sep=0.75pt]  [font=\footnotesize]  {$x_{2}$};
\draw (117.5,203.65) node [anchor=north west][inner sep=0.75pt]  [font=\footnotesize]  {$x_{4}$};
\draw (117.5,145.65) node [anchor=north west][inner sep=0.75pt]  [font=\footnotesize]  {$x_{1}$};
\draw (147.5,169.65) node [anchor=north west][inner sep=0.75pt]  [font=\footnotesize]  {$b_{3,5}$};
\draw (147,213.4) node [anchor=north west][inner sep=0.75pt]  [font=\footnotesize]  {$b_{2,1}$};
\draw (147.5,253.15) node [anchor=north west][inner sep=0.75pt]  [font=\footnotesize]  {$b_{1,4}$};
\draw (111,52) node [anchor=north west][inner sep=0.75pt]  [font=\footnotesize] [align=left] {time};

\end{tikzpicture}
  \caption{An illustration of the coalescent construction of the uniform $6$-point subtree.  
The vertical axis represents time decreasing from top to bottom.  
Each leaf $i=1,\ldots, 6$ is placed at height $x_i$ equal to its birth time, and each $b_{i,j}$ denotes the branching point at which the ancestral clusters of $i$ and $j$ merge.  
The final merge occurs at height $0$, which corresponds to the root.}
    \label{coalescent}
\end{figure}

\subsubsection*{Piecewise Kingman representation}

The Aldous' coalescent construction suggests a useful way to understand the tree \(\mathcal{T}^\ell_k\) through the evolution of the number of ancestral lineages across heights. 
For \(s \ge 0\), let \(Y^{(k)}_s\) denote the number of points in \(\mathcal{T}^\ell_k\) at height \(s\). 
Then \(Y^{(k)}_s\) can be naturally linked to a piecewise Kingman's coalescent 
(here and below, ``Kingman's coalescent'' always refers to its block-counting process), defined as follows.
\begin{definition}[Piecewise Kingman's coalescent]
    For real numbers $x_1 > x_2 > \cdots > x_k >0$, the piecewise Kingman's coalescent $\big(Z^{(k)}_s=Z^{(k)}_s(x_1,\ldots, x_k):s > 0\big)$ with rate $\frac{4}{\ell(s)}$ is defined as:
\begin{itemize}
\item For \(s > x_1\), we set \(Z^{(k)}_s = 0\).
\item Recursively, for each \(i = 1,\ldots,k\), conditionally on \(\big(Z^{(k)}_s : s > x_i\big)\), the process \(Z^{(k)}_s\) on the interval \(s \in (x_{i+1}, x_i]\) (with $x_{k+1}:=0$) evolves as a Kingman's coalescent started with \(\lim_{s \downarrow x_i} Z^{(k)}_s + 1\) lineages, where each pair of lineages coalesces at rate \(\frac{4}{\ell(s)}\).
Here the height parameter \(s\) is interpreted as backward time, so the coalescent runs as \(s\) decreases.
\end{itemize}
\end{definition}
\begin{proposition}\label{prop:piecewise_Kingman}
    For any $k\ge 1$ and real numbers $x_1 > x_2 > \cdots > x_k>0$, conditionally on the leaf heights being $x_1,\ldots, x_k$, the process \(Y^{(k)}_s\) evolves as the piecewise Kingman's coalescent $\big(Z^{(k)}_s(x_1,\ldots, x_k):s > 0\big)$.
\end{proposition}

We mention that the additional term \(+1\) at height \(x_{i+1}\) in the definition of \(Z^{(k)}\) corresponds, in the process \(Y^{(k)}\), to the introduction of a new leaf at its birth height \(x_{i+1}\).

This proposition gives rise to the following equivalent construction of \(\mathcal{T}^\ell_k\) for continuous profile $\ell$ with $h(\ell)<\infty$:

\begin{enumerate}
\item Sample $k$ heights independently with density \(\frac{\ell(x)}{I(\ell)}\,\mathrm{d}x\), and reorder them so that \(x_1 > \cdots > x_k\).
\item Sample a piecewise Kingman's coalescent process \(\big(Z^{(k)}_s=Z^{(k)}_s(x_1,\ldots,x_k): s \ge 0\big)\).
\item Conditional on \((Z^{(k)}_s)_{s > 0}\), sample uniformly from the set of rooted ordered real trees whose lineage counts at each height agree with \(Z^{(k)}\).
\end{enumerate}

\begin{proposition}
    The tree obtained from this construction has the same distribution as \(\mathcal{T}^\ell_k\).
\end{proposition}

\subsection{Distribution of uniform \(k\)-point subtrees induced by \(W^{\ell^\sigma}_{\alpha^{-1}(t)}\)}\label{sec:time-changed}
We now turn to the time-changed process \(W^{\ell^\sigma}_{\alpha^{-1}(t)}\). As a first step, we verify that \(\alpha\) is strictly increasing, so that its inverse \(\alpha^{-1}\) is well-defined. 
We then show that the total time duration of \(W^{\ell^\sigma}_{\alpha^{-1}(t)}\) is equal to \(I(\ell)\), which implies that the rescaled process
\(W^{\ell^\sigma}_{\alpha^{-1}(I(\ell) t)} \in C_{\mathrm{bridge}}\). 
After that, we describe the law of the uniform \(k\)-point subtrees induced by this process.

\begin{lemma}\label{lem:alpha}
    For any continuous profiles $\ell$ and $\sigma$ satisfying \eqref{H1}, the function
    \[
        \alpha(t)=\frac{1}{4}\int_0^t \sigma\big(W^{\ell^\sigma}_s\big)^2\,\mathrm{d}s
    \]
    is a.s.\ strictly increasing.
\end{lemma}

\begin{proof}
    By Proposition~\ref{prop:aldous1}, the total time duration of \(W^{\ell^\sigma}\) equals \(I(\ell^\sigma)\).
    Using \eqref{eq:aldous1}, we have a.s.
    \[
        \mathrm{Leb}\{s\in[0,I(\ell^\sigma)] : W^{\ell^\sigma}_s \notin (0,h(\ell))\} = 0.
    \]
    Since $\sigma>0$ on $(0,h(\ell))$, it follows that $\sigma(W^{\ell^\sigma}_s)^2>0$ for a.e.\ $s$, hence $\alpha$ is a.s.\ strictly increasing.
\end{proof}

\begin{proposition}
    For any continuous profile $\ell$ with $h(\ell)<\infty$,
    The total time duration of \(W^{\ell^\sigma}_{\alpha^{-1}(t)}\) is $I(\ell)$. Consequently, \(W^{\ell^\sigma}_{\alpha^{-1}(I(\ell)t)} \in C_{\mathrm{bridge}}\).
\end{proposition}
\begin{proof}
    The total time duration of the time-changed process \(W^{\ell^\sigma}_{\alpha^{-1}(t)}\) is
\[\alpha\big(I(\ell^\sigma)\big)=\int_0^{I(\ell^\sigma)} \sigma(W^{\ell^\sigma}_s)^2\d s\overset{\eqref{eq:aldous1}}{=}\int_0^\infty \ell^\sigma(s) \sigma(s)^2 \d s=I(\ell).\]
The fact that \(W^{\ell^\sigma}_{\alpha^{-1}(I(\ell)t)} \in C_{\mathrm{bridge}}\) then follows from Proposition~\ref{prop:aldous1} \eqref{item:aldous2}.
\end{proof}

Let $\mathcal{T}^{\ell,\sigma}$ denote the rooted ordered real tree encoded by 
$W^{\ell^\sigma}_{\alpha^{-1}(I(\ell)t)}$, and let $\mathcal{T}^{\ell,\sigma}_k$ be its uniform 
$k$-point subtree. We now describe the distribution of $\mathcal{T}^{\ell,\sigma}_k$.

\begin{theorem}\label{thm:tree_construction}
For any $k\in\mathbb{N}_+$, the tree $\mathcal{T}^{\ell,\sigma}_k$ admits the following construction:
\begin{enumerate}
\item Sample $k$ heights independently with density \(\frac{\ell(x)}{I(\ell)}\,\mathrm{d}x\), and reorder them so that \(x_1 > \cdots > x_k\).
\item Sample a piecewise Kingman coalescent 
\(Z^{(k)}=\big(Z^{(k)}_s(x_1,\ldots,x_k): s>0\big)\) in which each pair of lineages coalesces at rate $\frac{4}{\ell^\sigma(s)}$.
\item Conditionally on \(Z^{(k)}\), sample uniformly from the set of rooted ordered real trees whose lineage counts at each height agree with \(Z^{(k)}\).
\end{enumerate}
\end{theorem}

The proof proceeds by relating $\mathcal{T}^{\ell,\sigma}_k$ to $\mathcal{T}^{\ell^\sigma}_k$, the uniform $k$-point
subtree induced by $W^{\ell^\sigma}_{I(\ell^\sigma)t}$.
Note that $W^{\ell^\sigma}_{\alpha^{-1}(I(\ell)t)}$ differs from $W^{\ell^\sigma}_{I(\ell^\sigma)t}$ only by a strictly increasing time change:
\[
W^{\ell^\sigma}_{\alpha^{-1}(I(\ell)t)}
=
W^{\ell^\sigma}_{I(\ell^\sigma)\phi(t)},
\qquad 
\phi(t):=\frac{\alpha^{-1}(I(\ell)t)}{I(\ell^\sigma)}.
\]
Thus by Remark~\ref{treeiso}, for any $t_1,\ldots,t_k\in(0,1)$,
\begin{align}\label{eq:iso}
\mathcal{T}^{\ell,\sigma}(t_1,\ldots,t_k)
\text{ is isomorphic to }
\mathcal{T}^{\ell^\sigma}\big(\phi({t_1}),\ldots,\phi({t_k})\big).
\end{align}

\begin{proposition}
For any $k\in\mathbb{N}_+$, the laws of $\mathcal{T}^{\ell,\sigma}_k$ and $\mathcal{T}^{\ell^\sigma}_k$ 
are mutually absolutely continuous. More precisely, for any test function $f$,
\begin{align}\label{eq:RN}
\mathbb{E}\big[f(\mathcal{T}^{\ell,\sigma}_k)\big]
=
\left(\frac{I(\ell^\sigma)}{4 I(\ell)}\right)^k
\mathbb{E}\!\big[f(\mathcal{T}^{\ell^\sigma}_k)\,\sigma(\mathcal{T}^{\ell^\sigma}_k)^2\big],
\end{align}
where
\[
\sigma(\mathcal{T}^{\ell^\sigma}_k)
=
\prod_{j=1}^k \sigma(h_j),
\qquad
\text{$h_1,\ldots,h_k$ denote the leaf heights of $\mathcal{T}^{\ell^\sigma}_k$.}
\]
\end{proposition}

\begin{proof}
For any test function $f$,
\begin{align}\label{eq:change_of_variables}
\begin{aligned}
\mathbb{E}\big[f(\mathcal{T}^{\ell,\sigma}_k)\big]
&=\mathbb{E}\big[f(\mathcal{T}^{\ell,\sigma}(U_1,\ldots,U_k))\big] \\
&\overset{\eqref{eq:iso}}{=}\mathbb{E}\big[f\big(\mathcal{T}^{\ell^\sigma}(\phi({U_1}),\ldots,\phi({U_k})) \big)\big] \\
&=\int_{[0,1]^k}\mathbb{E}\big[f\big(\mathcal{T}^{\ell^\sigma}(\phi({s_1}),\ldots,\phi({s_k})) \big)\big]\,\mathrm{d}s_1\cdots\mathrm{d}s_k \\
&=\int_{[0,1]^k}\mathbb{E}\big[f(\mathcal{T}^{\ell^\sigma}(s_1,\ldots,s_k))\big]\,\mathrm{d}\phi^{-1}(s_1)\cdots\mathrm{d}\phi^{-1}(s_k).
\end{aligned}
\end{align}
By the definition of $\phi$,
\[
\phi^{-1}(s)=\frac{\alpha(I(\ell^\sigma)s)}{I(\ell)},
\qquad
\mathrm{d}\phi^{-1}(s)
=\frac{I(\ell^\sigma)}{4 I(\ell)}\,\sigma(W^{\ell^\sigma}_{I(\ell^\sigma)s})^2\,\mathrm{d}s,
\]
substituting into \eqref{eq:change_of_variables} yields
\begin{align*} \mathbb{E}\big[f(\mathcal{T}^{\ell,\sigma}_k)\big]&=\left(\frac{I(\ell^\sigma)}{4 I(\ell)}\right)^k \int_0^1\cdots \int_0^1\mathbb{E}\big[f(\mathcal{T}^{\ell^\sigma}(s_1,\ldots,s_k))\big]\prod_{j=1}^k \sigma(W^{\ell^\sigma}_{I(\ell^\sigma)s_j})^2 \d s_1\cdots \d s_k\\ 
&=\left(\frac{I(\ell^\sigma)}{4 I(\ell)}\right)^k\mathbb{E}\big[f(\mathcal{T}^{\ell^\sigma}_k)\sigma(\mathcal{T}^{\ell^\sigma}_k)^2\big]. \end{align*} This completes the proof. 
\end{proof}

\begin{proof}[Proof of Theorem \ref{thm:tree_construction}]
Let $\mathring{\mathcal{T}}^{\ell,\sigma}_k$ denote the random rooted ordered tree defined by the three-step procedure in the statement of the theorem. We assume that $\mathcal{T}^{\ell^\sigma}$ is realized via the piecewise Kingman representation in Proposition~\ref{prop:piecewise_Kingman}.

For any $x_1>\cdots>x_k>0$, the conditional laws of 
$\mathring{\mathcal{T}}^{\ell,\sigma}_k$ and $\mathcal{T}^{\ell^\sigma}_k$ given that the leaf heights equal $(x_1,\ldots,x_k)$ coincide.  
Thus, the only difference between their unconditional distributions comes from the respective sampling densities of the leaf heights.  
Comparing these densities yields, for any test function $f$,
\[
\mathbb{E}\big[f(\mathring{\mathcal{T}}^{\ell,\sigma}_k)\big]
=
\left(\frac{I(\ell^\sigma)}{4 I(\ell)}\right)^k\,
\mathbb{E}\!\big[f(\mathcal{T}^{\ell^\sigma}_k)\,\sigma(\mathcal{T}^{\ell^\sigma}_k)^2\big].
\]
Together with \eqref{eq:RN}, this implies 
$\mathring{\mathcal{T}}^{\ell,\sigma}_k\overset{d}{=}\mathcal{T}^{\ell,\sigma}_k$, thus completing the proof.
\end{proof}

\section{Coalescent process}\label{sec:tightness}
In this section, we first derive preparatory moment and probability estimates for elementary coalescence events (Section~\ref{sec:coal-prob}); these bounds will be used repeatedly in Sections~\ref{sec:tightness}--\ref{sec:proof}.
We then introduce the coalescent process arising in Cannings trees, show that it forms a Markov chain, and prove a coming-down-from-infinity (CDFI) property that serves as a central ingredient for the tightness of \(\widetilde{\mathscr{T}}^n\).

We begin with some notation. 
Denote by \(x_1^{n,s},\dots,x_{q_n(s)}^{n,s}\) the vertices at generation \(s\) of \(\mathscr{T}^n\), listed according to the lexicographic order.  
Let 
\[(u^{n,s}_1,\ldots,u^{n,s}_{q_n(s)})=(x^{n,s}_{\pi^{n,s}(1)},\ldots,x^{n,s}_{\pi^{n,s}(q_n(s))})\] 
be obtained from
$(x^{n,s}_1,\ldots,x^{n,s}_{q_n(s)})$ by a uniformly random permutation
$\pi^{n,s}$, where the permutations $\{\pi^{n,s}\}_s$ are
independent across generations and independent of $\mathscr{T}^n$. 
For any vertex \(x\), we write \(\xi(x)\) for the number of children of \(x\) and \(\mathfrak{p}(x)\) for the parent of \(x\). In particular, we have \(\xi(x_i^{n,s})=\nu_i^{n,s}\). 

For notational convenience, from this point onward we write \(\mathbb{P}_n\) and \(\mathbb{E}_n\) for the probability and expectation under the model indexed by \(n\), respectively, and we omit the superscript \(n\) in \(\nu^{n,s}\), \(x_i^{n,s}\), and \(u_i^{n,s}\) whenever it is clear from context.

\subsection{Coalescent probability}\label{sec:coal-prob}
In this subsection, we introduce coalescent events in Cannings trees and express their probabilities in terms of moments of the offspring counts \(\nu\). We then estimate these moments, which in turn yield quantitative estimates for the coalescent probabilities. 

\medskip

For $n \ge 1$ and $1 \le s < h(q_n)$, define  
\begin{align}\label{def:D}
    M^s=M^{n,s} := \Big\{ f : A \to \{1, \ldots, q_n(s)\}\,:\, A \subset \{1, \ldots, q_n(s+1)\} \Big\},
\end{align}
the collection of all mappings from a subset of $\{1, \ldots, q_n(s+1)\}$ to $\{1, \ldots, q_n(s)\}$.  

Each $f \in M^{s}$ naturally encodes a coalescent event:
\begin{align}\label{def:coal_event}
    \text{Coal}^{s}(f)=\text{Coal}^{n,s}(f) := \Big\{ \forall\, j \in \mathrm{Dom}(f),\ \mathfrak{p}(u^{n,s+1}_j) = x^{n,s}_{f(j)} \Big\}.
\end{align}
Suppose the range of $f$ is $\{j_1<j_2<\cdots<j_m\} \subset \{1,\ldots,q_n(s)\}$.  
We then define
\[
    N_f := \big(\#f^{-1}(j_1),\ldots,\#f^{-1}(j_m)\big),
\]
which records the multiplicities of coalescence in the event $\text{Coal}^{s}(f)$.

\subsubsection{Linking coalescent probability to offspring moments}

\begin{proposition}\label{prop:coal-prob}
For $n \ge 1$, $1 \le s < h(q_n)$, and $f \in M^{s}$ with $N_f=(N_1,\ldots,N_m)$ and $\#\mathrm{Dom}(f) :=\sum_{i=1}^m N_i$, we have
\begin{align}\label{eq:conclusion1}
    \p_n\big(\text{Coal}^{s}(f)\big)
    = \frac{1}{\big(q_n(s+1)\big)_{\#\mathrm{Dom}(f)}}\, \mathbb{E}_n\!\Bigg[\prod_{i=1}^{m}\big(\nu^s_{j_i}\big)_{N_i}\Bigg],
\end{align}
where $(l)_k := l(l-1)\cdots(l-k+1)$ denotes the falling factorial.
\end{proposition}

\begin{proof}
Denote by $\{j_1<\cdots<j_m\}$ the range of $f$.  
Conditionally on $(\nu^s_{j_i}: 1\le i\le m)$, since $(u^{s+1}_1,\dots,u^{s+1}_{q_n(s+1)})$ is a uniform random permutation of the $(s+1)$-st generation, the probability that the $\#\mathrm{Dom}(f)$ uniformly selected positions realize the multiplicities $N_i$ is
\[
\frac{\prod_{i=1}^m (\nu^s_{j_i})_{N_i}}{(q_n(s+1))_{\#\mathrm{Dom}(f)}}.
\]
Taking expectation over $(\nu^s_{j_i})$ yields \eqref{eq:conclusion1}.
\end{proof}
We now record several direct consequences of Proposition~\ref{prop:coal-prob}, which will be used in the sequel.

\begin{corollary}\label{puni}
For $n\ge 1$, $1\le s< h(q_n)$, $1\le l\le m\le q_n(s+1)$, it holds that
    \begin{gather}
    \p_n\!\Big(\{\mathfrak{p}(u_{1}^{s+1})=\cdots=\mathfrak{p}(u_{l}^{s+1})\}\cap\{\mathfrak{p}(u_{l}^{s+1}),\dots,\mathfrak{p}(u_{m}^{s+1})~\text{are distinct}\}\Big) 
    =\,\frac{(q_n(s))_{m-l+1}}{(q_n(s+1))_m}\,\e_n\!\Big[(\nu_1^s)_l\prod_{i=2}^{m-l+1}\nu_i^s\Big],
    \label{item:puni1} \\
    \p_n\!\Big(\mathfrak{p}(u_{1}^{s+1})=\mathfrak{p}(u_{2}^{s+1})\neq\mathfrak{p}(u_{3}^{s+1})=\mathfrak{p}(u_{4}^{s+1})\Big)=\frac{(q_n(s))_{2}}{(q_n(s+1))_4}\,\e_n\!\Big[(\nu_1^s)_2(\nu_2^s)_2\Big]. \label{item:puni2}
    \end{gather}
\end{corollary}

For later use in the proof of the main theorem, when applying \eqref{item:puni1}, we will mainly need the cases $l=1$ or $2$. Following the same argument as in the proof of Proposition~\ref{prop:coal-prob}, Corollary~\ref{puni} admits the following refinement.  

\begin{corollary}\label{coleprob}
For $n,K\ge 1$, $1\le s< h(q_n)$ and $1\le l\le q_n(s+1)$, it holds that:
  \begin{align*}
  &\p_n\!\Big(\mathfrak{p}(u_{1}^{s+1})=\cdots=\mathfrak{p}(u_{l}^{s+1}),\;
  \xi(\mathfrak{p}(u_{1}^{s+1}))\le K\Big)= \frac{q_n(s)}{(q_n(s+1))_l}\,
  \e_n\!\Big[(\nu_1^s)_l\,\mathbf{1}\{\nu_1^s\le K\}\Big]; \\
  &\p_n\!\Big(\mathfrak{p}(u_{1}^{s+1})=\mathfrak{p}(u_{2}^{s+1})\neq
  \mathfrak{p}(u_{3}^{s+1})=\mathfrak{p}(u_{4}^{s+1}),\, \xi(\mathfrak{p}(u_{1}^{s+1}))\le K,\;
  \xi(\mathfrak{p}(u_{3}^{s+1}))\le K\Big) \\
  =\,& \frac{(q_n(s))_{2}}{(q_n(s+1))_4}\,
  \e_n\!\Big[(\nu_1^s)_2(\nu_2^s)_2\,
  \mathbf{1}\{\nu_1^s,\,\nu_2^s\le K\}\Big].
  \end{align*}
\end{corollary}

\subsubsection{Estimates for offspring moments}
In this part, we provide estimates for the offspring moments appearing in Corollary~\ref{puni}. We start with the $(2,2)$-type moment appearing in \eqref{item:puni2}. Recall the conditions \eqref{H1} and \eqref{eq:3rd} from Section \ref{sec:setup}.
\begin{lemma}\label{22limit}
For any $n\ge 1$ and $1 \le s < h(q_n)$,
\begin{align}\label{eq:(2,2) comparison}
    \e_n[(\nu^s_1)^2(\nu^s_2)^2] \le \frac{q_n(s+1)}{q_n(s)-1}\e_n[(\nu^s_1)^3]
\end{align}
In particular, under conditions \eqref{H1} and \eqref{eq:3rd}, for any $\delta>0$,
\begin{align}\label{eq:(2,2)}
    \lim_{n\to\infty} \sup_{\delta n \le s \le h(q_n)-\delta n} \frac{\mathbb{E}_n[(\nu_1^{s})^2 (\nu_2^{s})^2]}{n} = 0.
\end{align}
\end{lemma}

\begin{proof}
We first note the elementary inequality $2(\nu^s_1)^2(\nu^s_2)^2\le \nu^s_1(\nu^s_2)^3+(\nu^s_1)^3\nu^s_2$. By exchangeability of the offspring vector $\nu^s$, this implies $\e_n[(\nu^s_1)^2(\nu^s_2)^2]\le \e_n[(\nu^s_1)^3\nu^s_2]$. Moreover,
\begin{align*}
     \e_n[(\nu^s_1)^3\nu^s_2] = \frac{1}{q_n(s)-1} \e_n\Big[(\nu^s_1)^3\sum_{j=2}^{q_n(s)}\nu^s_j\Big] = \frac{q_n(s+1)}{q_n(s)-1}\e_n[(\nu^s_1)^3] - \frac{1}{q_n(s)-1}\e_n[(\nu^s_1)^4] \le \frac{q_n(s+1)}{q_n(s)-1}\e_n[(\nu^s_1)^3],
\end{align*}
which proves \eqref{eq:(2,2) comparison}. 

Under condition~\eqref{H1}, the continuity of $\ell$ implies that for every
$\delta>0$,  
\[
    \sup_{\delta n \le s \le h(q_n)-\delta n}
    \frac{q_n(s+1)}{q_n(s)-1}
    \;\longrightarrow\; 1
    \qquad \text{as } n\to\infty.
\]
Combining this with \eqref{eq:(2,2) comparison} and \eqref{eq:3rd}
yields \eqref{eq:(2,2)}.
\end{proof}

Next, we consider the moments in \eqref{item:puni1} with $l=1,2$.

\begin{lemma}\label{estmom}
Assume \eqref{H1} and \eqref{eq:3rd}. For any $\delta>0$ and $m\ge 1$, uniformly for $\delta n \le s \le h(q_n)-\delta n$, we have
    \begin{align}\label{item1_1}
        \mathbb{E}_n\Big[\prod_{i=1}^m \nu_i^s\Big] = 1 + o_n(1),
    \end{align}
    and more precisely,
    \begin{align}\label{item1_2}
        \mathbb{E}_n\Big[\prod_{i=1}^m \nu_i^s\Big] 
        =\frac{(q_n(s+1))^m}{(q_n(s))_m}-\dfrac{\binom{m}{2}\big(1+\sigma_n^2(s)\big)}{q_n(s)}\big(1+o_n(1)\big); 
    \end{align}
    Furthermore, 
    \begin{align}\label{item2}
        \mathbb{E}_n\Big[(\nu_1^s)^2 \prod_{i=2}^m \nu_i^s\Big] = 1 + \sigma_n(s)^2 + o_n(1).
    \end{align}
\end{lemma}

\begin{proof}
Set
\begin{align*} a_m &:= \mathbb{E}_n\Big[\prod_{i=1}^m \nu_i^s\Big], \quad 1\le m\le q_n(s),\\
b_1 &:= \mathbb{E}_n[(\nu_1^s)^2], \quad b_m := \mathbb{E}_n\Big[(\nu_1^s)^2 \prod_{i=2}^m \nu_i^s\Big],\quad 2\le m\le q_n(s),\\
c_1 &:= \mathbb{E}_n[(\nu_1^s)^3],\quad c_m := \mathbb{E}_n\left[(\nu_1^s)^3 \prod_{i=2}^m \nu_i^s\right],\quad 2\le m\le q_n(s),\\
d_2 &:= \mathbb{E}_n[(\nu_1^s)^2 (\nu_2^s)^2],\quad d_m := \mathbb{E}_n\left[(\nu_1^s)^2 (\nu_2^s)^2 \prod_{i=3}^m \nu_i^s\right],\quad 3\le m\le q_n(s). 
\end{align*}
By conditioning on $\nu_1^s,\dots,\nu_{m-1}^s$, we obtain by induction that for any $m\ge 1$
\begin{align}\label{ind1}
\begin{aligned}
a_m &= \mathbb{E}_n\Bigg[\frac{q_n(s+1)-\sum_{i=1}^{m-1} \nu_i^s}{q_n(s)-(m-1)} \prod_{i=1}^{m-1} \nu_i^s\Bigg]
= \frac{q_n(s+1)}{q_n(s)-(m-1)} a_{m-1} - \frac{m-1}{q_n(s)-(m-1)} b_{m-1}\\
&=\cdots= \frac{(q_n(s+1))^m}{(q_n(s))_m} - \sum_{i=1}^{m-1} \frac{(q_n(s+1))^{m-i-1}}{(q_n(s)-i)_{m-i}}\cdot i b_i, 
\end{aligned}
\end{align}
 where in the last step we use the fact $\e_n[\nu^s_1]=\frac{q_n(s+1)}{q_n(s)}$. Similarly, for $b_m$,
\begin{align}\label{ind3}
    b_m = \frac{(q_n(s+1))^{m-1}}{(q_n(s)-1)_{m-1}} b_1 - \sum_{i=1}^{m-1} \frac{(q_n(s+1))^{m-i-1}}{(q_n(s)-i)_{m-i}} (c_i + (i-1) d_i),
\end{align}
and in particular,
\begin{align}\label{ind2}
b_m \le \frac{(q_n(s+1))^{m-1}}{(q_n(s)-1)_{m-1}} b_1.
\end{align}
 Combining \eqref{ind1} and \eqref{ind2}, we derive that
\begin{align*}
    \left|a_m-\frac{(q_n(s+1))^m}{(q_n(s))_m}\right|&\le \sum_{i=1}^{m-1}\frac{ (q_n(s+1))^{m-i-1}}{(q_n(s)-i)_{m-i}}\cdot \frac{(q_n(s+1))^{i-1}}{(q_n(s)-1)_{i-1}}\cdot i b_1\\
    &=\frac{m(m-1)}{2}\frac{q_n(s+1)^{m-2}}{(q_n(s)-1)_{m-1}}b_1.
\end{align*}
By \eqref{H1}, we have
\[
\frac{(q_n(s+1))^m}{(q_n(s))_m} = 1+o_n(1),
\qquad
\frac{(q_n(s+1))^{m-2}}{(q_n(s)-1)_{m-1}} = O_n\!\left(\tfrac{1}{n}\right),
\]
where both $o_n(1)$ and $O_n(\tfrac{1}{n})$ are uniform in $\delta n \le s \le h(q_n)-\delta n$. Hence, \eqref{item1_1} follows.

The bound \eqref{item2} follows similarly. Indeed, in the same way as \eqref{ind2},
\[
c_m \le \frac{(q_n(s+1))^{m-1}}{(q_n(s)-1)_{m-1}} c_1,\quad d_m \le \frac{(q_n(s+1))^{m-2}}{(q_n(s)-1)_{m-2}} d_2.
\]
Approximating $b_1$ by \eqref{H1}, $c_1$ by \eqref{eq:3rd} and $d_2$ by Lemma~\ref{22limit}, we deduce \eqref{item2} from \eqref{ind3}. 
Finally, by \eqref{H1}, uniformly over $i=1,\cdots,m-1$,
\[
\frac{(q_n(s+1))^{m-i-1}}{(q_n(s)-i)_{m-i}} = \frac{1}{q_n(s)}\big(1+o_n(1)\big),
\]
and thus \eqref{item1_2} follows directly from \eqref{item2} together with \eqref{ind1}.
\end{proof}

\begin{corollary}\label{distinctprob}
Assume \eqref{H1} and \eqref{eq:3rd}. For any $\delta>0$ and $m\ge 1$, uniformly for $\delta n \le s \le h(q_n)-\delta n$, we have:
        \begin{gather}
            \p_n\Big(\mathfrak{p}(u_{1}^{s+1}),\ldots, \mathfrak{p}(u_{m}^{s+1})~\text{are  distinct}\Big)=1-\dfrac{\binom{m}{2}\sigma_n(s)^2}{q_n(s)}\big(1+o_n(1)\big);    \label{eq:precise}  \\
            \p_n\Big(\{\mathfrak{p}(u_{1}^{s+1})=\mathfrak{p}(u_{2}^{s+1})\}\cap\{ \mathfrak{p}(u_{2}^{s+1}),\ldots, \mathfrak{p}(u_{m}^{s+1})~\text{are  distinct}\}\Big)
            =\,\frac{1}{q_n(s)} \big(\sigma_n(s)^2 + o_n(1)\big).     \label{eq:distinct2}  
        \end{gather}
\end{corollary}

\begin{proof}
Recall that by \eqref{H1}, $\frac{q_n(s+1)}{q_n(s)} = 1 +o_n(1)$ uniformly over $\delta n \le s \le h(q_n)-\delta n$. 
By \eqref{item:puni1} (with $l=1$) in Corollary~\ref{puni}  and \eqref{item1_2}, we deduce that 
\[
    \begin{aligned}
    &\p_n\!\Big(\mathfrak{p}(u_{1}^{s+1}),\dots,\mathfrak{p}(u_{m}^{s+1})~\text{are distinct}\Big) = \frac{(q_n(s))_m}{(q_n(s+1))_m}\,\e_n\!\Big[\prod_{i=1}^{m}\nu_i^s\Big]\\
    =\, &\frac{(q_n(s+1))^m}{(q_n(s+1))_m} -\dfrac{\binom{m}{2}\big(1+\sigma_n^2(s)\big)}{q_n(s)}\big(1+o_n(1)\big).
    \end{aligned}
\]
Expanding the first term on the last line gives
\[
\frac{(q_n(s+1))^m}{(q_n(s+1))_m} = \prod_{i=0}^{m-1}\left( 1-\frac{i}{q_n(s+1)}\right)^{-1} = 1 + \binom{m}{2}\frac{1}{q_n(s+1)} + o_n\left(\frac{1}{q_n(s+1)}\right),
\]
which yields \eqref{eq:precise}.
Moreover, by \eqref{item1_1} and \eqref{item2},
\[
\e_n\!\Big[(\nu_1^s)_2\prod_{i=2}^{m-1}\nu_i^s\Big] = \e_n\!\Big[(\nu_1^s)^2\prod_{i=2}^{m-1}\nu_i^s\Big] - \e_n\!\Big[\prod_{i=1}^{m-1}\nu_i^s\Big] = \sigma_n(s)^2 + o_n(1).
\]
Therefore \eqref{eq:distinct2} follows from \eqref{item:puni1} (with $l=2$).
\end{proof}

\subsection{Coalescent process}\label{sec:coalescence}

In this subsection, we introduce the coalescent process and study its transition probabilities.
The central result is Theorem~\ref{thm:Markov}, which identifies the coalescent process as a (time-inhomogeneous) Markov chain
and provides its transition formula.

\medskip

We begin with some relevant notation. For any vertex \(v \in \mathscr{T}^n\) and \(0 \le i \le |v|\), let \(\mathrm{Anc}(v,i)\) denote the ancestor of \(v\) at height \(i\).  
Define the set of descendants of \(v\) by
\[
\mathrm{Desc}(v)=\mathrm{Desc}(v,n)
:=\big\{u \in V(\mathscr{T}^n): |u|\ge |v| \text{ and } \mathrm{Anc}(u,|v|)=v\big\}.
\]

\begin{definition}[Coalescent process]\label{def:coal_process}
Fix a height $0 \le h^* \le h(q_n)$ and an integer $1\le k\le q_n(h^*)$. For $0 \le j \le h^*$, set
\begin{equation}\label{def:Xj}
\begin{aligned}
  \cX^{(k)}_j = \cX^{(k)}_j(n,h^*)&=\big\{\text{Anc}(u^{h^*}_i,j):1\le i\le k\big\} \\
  &=\big\{v \in V(\mathscr{T}^n) : |v| = j,\ \operatorname{Desc}(v) \cap \{u^{h^*}_i:1\le i\le k\} \neq \emptyset \big\}.  
\end{aligned}
\end{equation}
The coalescent process (based at height $h^*$) is then defined as
\[
X^{(k)}_j = X^{(k)}_j(n,h^*) := \#\cX^{(k)}_j,\qquad 0 \le j \le h^*.
\]
In the special case $k=q_n(h^*)$, we write $\cX_j$ and $X_j$ in place of $\cX^{(k)}_j$ and $X^{(k)}_j$, respectively.

In other words, $X^{(k)}_j$ represents the number of distinct ancestral lineages at height $j$ traced from $k$ uniformly sampled vertices at height $h^*$.  
At $j=h^*$, we have $X^{(k)}_0 = k$, since all sampled vertices are distinct.  
As $j$ decreases, the lineages are traced backward through the tree and may coalesce whenever two or more sampled vertices share a common ancestor.  
Thus, the process $\big(X^{(k)}_j : 0 \le j \le h^*\big)$ is non-increasing as $j$ decreases, and each downward jump corresponds to a coalescence event between ancestral lineages.

\end{definition}

We consider subtrees of $\mathscr{T}^n$ rooted at some vertex $v$ and truncated at height $h^*$. To make this precise, we recall the representation of a tree as a subset of \(\mathcal{U}\) in Section~\ref{sec:Cannings_tree}.  
For \(u=(u_1,\ldots,u_m)\), \(v=(v_1,\ldots,v_l)\in\mathcal{U}\), we write \(uv\) for the concatenated sequence \((u_1,\ldots,u_m,v_1,\ldots,v_l)\). For each $v\in V(\mathscr{T}^n)$, define the truncated subtree
\begin{align}\label{def:subtree}
    \mathscr{T}^{n}(v)=\mathscr{T}^{n}(v,h^*)
    :=\big\{u\in\mathcal{U}:vu\in\mathscr{T}^n,\; \abs{vu}\le h^* \big\}.
\end{align}
Moreover, for any $1\le k\le q_n(h^*)$ and $v\in V(\mathscr{T}^n)$, we define the truncated subtree consisting only of the ancestral lineages of the sampled vertices \(\{u^{h^*}_i:1\le i\le k\}\):
\[\mathscr{T}^{n,(k)}(v)=\mathscr{T}^{n,(k)}(v,h^*):=\big\{u\in\mathcal{U}:vu\in\mathscr{T}^n(v),\; \exists\, 1\le i\le k\text{ s.t. }u^{h^*}_i\in \text{Desc}(vu)\big\}.\]
Note that both \(\mathscr{T}^{n}(v)\) and \(\mathscr{T}^{n,(k)}(v)\) record only the suffix 
$u$ in the decomposition $vu$; the prefix $v$ itself is forgotten in this representation.
For \(0\le j\le h^*\) and \(1\le k\le q_n(h^*)\), set
\begin{align}\label{def:G}
    \mathcal{G}^{(k)}_j=\mathcal{G}^{(k)}_j(h^*)
    := \sigma\!\big(\big\{\big(\mathscr{T}^{n}(v),\mathscr{T}^{n,(k)}(v)\big): \abs{v}=j\big\}\big).
\end{align}
Here the collection \(\big\{\big(\mathscr{T}^{n}(v),\mathscr{T}^{n,(k)}(v)\big): |v|=j\big\}\) is regarded as a set (rather than a sequence), so no ordering information at generation \(j\) is retained. Clearly, as $j$ decreases, the $\sigma$-fields $\mathcal{G}^{(k)}_j$ form an increasing family. See Figure~\ref{fig:G filtration} for an illustration of the $\sigma$-fields.

When \(k=q_n(h^*)\), we simply write \(\mathcal{G}_j=\mathcal{G}^{(k)}_j\); in this case, since \(\mathscr{T}^{n}(v)\) already determines \(\mathscr{T}^{n,(q_n(h^*))}(v)\).
\begin{align}\label{def:Gj}
    \mathcal{G}_j=\sigma\!\big(\{\mathscr{T}^{n}(v): \abs{v}=j\}\big),
\end{align}

\begin{figure}
    \centering
    \includegraphics[width=0.45\linewidth]{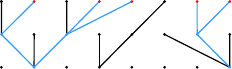}
    \caption{Illustration of the $\sigma$-field $\mathcal{G}^{(k)}_{j}$ with $h^*=3$, $j=1$ and $k=5$. 
    The red points represent the $k=5$ uniformly sampled vertices at height $h^*=3$.
    For each $\abs{v}=1$, the tree $\mathscr{T}^{n}(v)$ consists of the entire 
    descendant subtree of $v$ up to height $h^*$ (shown in both the black and blue), 
    while $\mathscr{T}^{n,(k)}(v)$ (blue only) records the portions of $\mathscr{T}^{n}(v)$ 
    that lie on the ancestral lineages of the sampled vertices.  
    The $\sigma$-field $\mathcal{G}^{(k)}_{j}$ is generated by these subtrees regarded as a set.}
    \label{fig:G filtration}
\end{figure}

We now present some basic properties of the coalescent process $(X^{(k)}_j)$.

\begin{proposition}\label{prop:uniform}
    For any $1 \le j \le h^*$ and $1\le k\le q_n(h^*)$, conditionally on $\mathcal{G}^{(k)}_j$, if $X^{(k)}_j=m$, then the set $\cX_j$ is distributed as $\{u_1^j, \dots, u_m^j\}$. 
\end{proposition}

\begin{proof}
    By the exchangeability of the offspring vector $\nu^j=\big(\nu_1^j,\ldots,\nu_{q_n(j)}^j\big)$, we can readily verify that the law of the vector of subtrees
    \[
        (\mathscr{T}^n(v):|v|=j)
    \]
    is also exchangeable.  
    The claimed conditional distribution of $\cX^{(k)}_j$ now easily follows.
\end{proof}

\begin{theorem}\label{thm:Markov}
    For any $1\le k\le q_n(h^*)$, the process $(X^{(k)}_{h^*-j} : 0 \le j \le h^*)$ is a Markov chain with respect to $(\mathcal{G}^{(k)}_{h^*-j} : 0 \le j \le h^*)$. Moreover, for any $1 \le j \le h^*$, its transition probabilities satisfy
    \begin{align}\label{eq:transition}
        \p_n\big(X^{(k)}_{j-1}=m_2 \,\big|\, X^{(k)}_j=m_1 \big)
        = \p_n\!\left(\#\{\kp(u^j_i):1\le i\le m_1\}=m_2\right).
    \end{align}
\end{theorem}

We give the proof only for the special case \(k = q_n(h^*)\).  
The argument for general \(k\) is entirely analogous and is therefore omitted.  
In the proof, we will use the following elementary property of conditional expectation.

\begin{lemma}\label{lem:basic}
Let $X,Y,Z$ be random elements on the same probability space, with $Y$ independent of $(X,Z)$. For any test function $f(x,y)$ on the value space of $(X,Y)$, define
\[
g(x) := \mathbb{E}[f(x,Y)].
\]
Then
\begin{align}\label{eq:basic1}
    \mathbb{E}\big(f(X,Y)\mid Z\big) = \mathbb{E}\big(g(X)\mid Z\big).
\end{align}
In particular, if there exists a random element $W$, independent of $Y$ and distributed as $(X \mid Z)$, then
\begin{align}\label{eq:basic2}
    \mathbb{E}\big(f(X,Y)\mid Z\big) = \mathbb{E}\big(f(W,Y)\big).
\end{align}
\end{lemma}

\begin{proof}
By the standard property of conditional expectation for independent random elements (see, e.g., \cite[Example~4.1.7]{durrett2019probability}),
\[
\mathbb{E}\big(f(X,Y)\mid X,Z\big) = g(X).
\]
Applying the tower property of conditional expectation yields \eqref{eq:basic1}.  
For \eqref{eq:basic2}, note that
\[
\mathbb{E}\big(g(X)\mid Z\big) = \mathbb{E}[g(W)] = \mathbb{E}[f(W,Y)],
\]
since $W$ and $Y$ are independent.  
\end{proof}

\begin{lemma}
    Let $1\le m_2\le m_1\le q_n(j)$ and $k_1,\ldots,k_{m_2}\in[1,q_n(j-1)]$ be distinct. Then on the event $\{X_j=m_1\}$,
    \begin{align}\label{eq:markov1}
        \mathbb{P}_n\big(\cX_{j-1} = \{x^{j-1}_{k_1},\ldots, x^{j-1}_{k_{m_2}}\}\,\big|\,\mathcal{G}_j\big)=\binom{q_n(j-1)}{m_2}^{-1} \p_n\big(\#\{\kp(u^{j}_i):1\le i\le m_1\}=m_2\big).
    \end{align}
\end{lemma}
\begin{proof}
    We shall apply Lemma \ref{lem:basic} to
\begin{align}\label{eq:apply_lem}
\begin{aligned}
    \mathbb{P}_n\big(\cX_{j-1} = \{x^{j-1}_{k_1},\ldots, x^{j-1}_{k_{m_2}}\}\,\big|\,\mathcal{G}_{j}\big)
    =\; \p_n \Big( \big\{\mathfrak{p}(x):x\in \cX_j\big\} = \{x^{j-1}_{k_1},\ldots, x^{j-1}_{k_{m_2}}\}\,\big|\, \big(\mathscr{T}^n(v):|v|=j\big) \Big).    
\end{aligned}
\end{align}
Observe the following:
\begin{itemize}
    \item[(a)] Both $\cX_j$ and $\big(\mathscr{T}^n(v):|v|=j\big)$ are measurable with respect to 
    \(\sigma\big(\nu^{j},\nu^{j+1}\ldots,\nu^{h^*}\big),\) therefore independent of $\nu^{j-1}$ due to generation-wise independence of $\nu^s$.
    \smallskip
    \item[(b)] The event $\Big\{\big\{\mathfrak{p}(x):x\in \cX_j\big\} = \big\{x^{j-1}_{k_1},\ldots, x^{j-1}_{k_{m_2}}\big\}\Big\}$ is determined by $\cX_j$ and $\nu^{j-1}$, i.e.,
    \[
    \Big\{\big\{\mathfrak{p}(x):x\in \cX_j\big\} = \big\{x^{j-1}_{k_1},\ldots, x^{j-1}_{k_{m_2}}\big\}\Big\} \in \sigma(\cX_j,\nu^{j-1}).
    \]
    \item[(c)] By Proposition \ref{prop:uniform}, on the event $\{X_j=m_1\}$,
    \[
    \big(\cX_{j} \mid \big(\mathscr{T}^n(v):|v|=j\big)\big) \overset{d}{=} \{u^{j}_1,\ldots,u^{j}_{m_1}\}.
    \]
\end{itemize}

\noindent Applying Lemma \ref{lem:basic} with 
\[
X=\cX_{j},\quad Y=\nu^{j-1},\quad Z=\big\{\mathscr{T}^n(v):|v|=j\big\},\quad W=\{u^j_1,\ldots,u^j_{m_1}\},
\]
and $f(X,Y)$ the indicator function determined by $\cX_{j}$ and $\nu^{j-1}$ corresponding to the event
\[
\Big\{\big\{\mathfrak{p}(x):x\in \cX_j\big\} = \{x^{j-1}_{k_1},\ldots, x^{j-1}_{k_{m_2}}\}\Big\},
\]
we obtain that on the event $\{X_{j}=m_1\}$,
\begin{equation}\label{eq:j+1 uniform}
\begin{aligned}
    &\p_n \Big( \big\{\mathfrak{p}(x):x\in \cX_j\big\} = \{x^{j-1}_{k_1},\ldots, x^{j-1}_{k_{m_2}}\}\,\big|\, \big(\mathscr{T}^n(v):|v|=j\big) \Big) \\
    =\;& \p_n\big(\{\kp(u^{j}_i):1\le i\le m_1\} = \{x^{j-1}_{k_1},\ldots, x^{j-1}_{k_{m_2}}\}\big)\\
    =\;&\binom{q_n(j-1)}{m_2}^{-1} \p_n\big(\#\{\kp(u^j_i):1\le i\le m_1\}=m_2\big).
\end{aligned}
\end{equation}
Substituting back into \eqref{eq:apply_lem} gives the desired conclusion.
\end{proof}

\begin{proof}[Proof of Theorem \ref{thm:Markov}]
It follows from \eqref{eq:markov1} that on $\{X_{j}=m_1\}$,
\begin{align*}
    \p_n\big(X_{j-1}=m_2 \,\big|\, \mathcal{G}_{j}\big)
        &=\sum_{1\le k_1<\cdots<k_{m_1}\le q_n(j-1)}\mathbb{P}_n\big(\cX_{j-1} = \{x^{j-1}_{k_1},\ldots, x^{j-1}_{k_{m_1}}\}\,\big|\,\mathcal{G}_{j}\big)\\
        &=\p_n\!\left(\#\{\kp(u^{j}_i):1\le i\le m_{1}\}=m_2\right).
\end{align*}
This confirms the Markov property and the transition formula \eqref{eq:transition}.
\end{proof}

We close this section with a simple estimate.
\begin{lemma}
Assume \eqref{H1}. Then for any $\delta>0$, uniformly for $1\le k\le q_n(h^*)$, $j\in [\delta n,\, h(q_n)-\delta n]$ and $1\le m\le k$, 
    \begin{align}\label{eq:upper}
        \p_n\big(X^{(k)}_{j-1} \le m-1 \,\big|\, X^{(k)}_j = m\big) \lesssim \frac{m^2}{n}.
    \end{align}
\end{lemma}

\begin{proof}
    
By \eqref{item:puni1} in Corollary~\ref{puni} and Theorem~\ref{thm:Markov},  
\begin{align*}
   & \p_n\big(X^{(k)}_{j-1} \le m-1 \,\big|\, X^{(k)}_j = m\big) 
    = \p_n\Big(\#\{\kp(u^{j}_i):1\le i\le m\}\le m-1\Big)\\
    &\le \binom{m}{2}\,\p_n\big(\mathfrak{p}(u^j_1)=\mathfrak{p}(u^j_2)\big)= \binom{m}{2}\,\frac{q_n(j-1)}{(q_n(j))_2}\,\e_n\big[(\nu_1^{j})_2\big].
\end{align*}
By \eqref{H1}, this is $O(m^2/n)$, yielding \eqref{eq:upper}.
\end{proof}

\subsection{Coming-down-from-infinity property}\label{sec:CDFI}
In this subsection, we establish the CDFI property for the coalescent process $X$.

Fix two sequences $h^*(n)$ and $h_*(n)$, $n \ge 1$, such that 
\begin{equation}\label{eq:h*condition}
\begin{aligned}
 & 0 < h_*(n) < h^*(n) < h(q_n), & \forall\, n \ge 1, \\
 &\frac{h_*(n)}{n} \to H_* \quad \text{and} \quad \frac{h^*(n)}{n} \to H^*, & n\to\infty,
\end{aligned}
\end{equation}
with some constants $0 < H_* < H^* < h(\ell)$.
In particular, \eqref{eq:h*condition} together with \eqref{H1} implies that there exists $\delta > 0$ such that, for all $n \ge 1$, 
\begin{align}\label{eq:h*condition2}
    \delta n \le h_*(n) < h^*(n) \le h(q_n) - \delta n\quad\text{and}\quad h^*(n)-h_*(n)\ge \delta n.
\end{align}
We consider the coalescent process $X_j=X_j(n,h^*(n))$ based at height $h^*(n)$ and focus on its value at height $h_*(n)$, namely $X_{h_*(n)}$.
For brevity, we write $h^*=h^*(n)$ and $h_*=h_*(n)$ when $n$
is clear from the context.

\begin{theorem}[CDFI property]\label{thm:CDFI}
    Assume \eqref{H1} and \eqref{eq:liminf}. 
    Then the sequence $(X_{h_*(n)})$ is tight, i.e.
    \[
    \lim_{M \to \infty}\, \limsup_{n \to \infty}\, \p_n(X_{h_*} > M) = 0.
    \]
\end{theorem}

\begin{remark*}
This can be regarded as a sequential version of the CDFI property, 
since the bound \( M \) must be chosen independently of \( n \).
\end{remark*}

\begin{proof}
For each \( M \ge 1 \), define the hitting time 
\[
    T_M = T_M(n) := \sup\{0 \le j \le h^*(n) : X_{j} \le M \}.
\]
Clearly,
\[
    \{X_{h_*} > M\} = \{T_M < h_*\}\subset \{T_M \le  h^*-\delta n\}
\]
by \eqref{eq:h*condition2}. Hence it suffices to prove
\begin{align}\label{eq:TM}
	\lim_{M \to \infty} \limsup_{n \to \infty} \p_n(T_M \le h^*-\delta n) = 0.
\end{align}
It is enough to show that for any \( \eta > 0 \), there exists \( M = M(\eta) \ge 1 \) such that, for all sufficiently large \( n \),\phantom\qedhere
\begin{align}
	&\p_n\Big(T_{\lfloor2(\log n)^{1+\frac{\varepsilon}{2}}\rfloor}\le h^*-\tfrac{1}{3}\delta n\Big)<\tfrac{\eta}{3},\label{bound1}\\
    &\p_n( T_{\lfloor (\log n)^{\frac{3+2\varepsilon}{6}}\rfloor}\le T_{\lfloor2(\log n)^{1+\frac{\varepsilon}{2}}\rfloor} -\tfrac13 \delta n \,\Big|\, T_{\lfloor 2(\log n)^{1+\frac{\varepsilon}{2}}\rfloor}>  h^*-\tfrac{1}{3}\delta n\Big)<\tfrac{\eta}{3}, \label{bound3}\\
	&\p_n\Big(T_M\le T_{\lfloor(\log n)^{\frac{3+2\varepsilon}{6}}\rfloor}-\tfrac{1}{3}\delta n \,\Big|\, T_{\lfloor (\log n)^{\frac{3+2\varepsilon}{6}}\rfloor}> T_{\lfloor2(\log n)^{1+\frac{\varepsilon}{2}}\rfloor} -\tfrac13 \delta n ,\, T_{\lfloor 2(\log n)^{1+\frac{\varepsilon}{2}}\rfloor}>  h^*-\tfrac{1}{3}\delta n\Big)<\tfrac{\eta}{3}.\label{bound2}
\end{align}

\subsubsection{Proof of \eqref{bound3} and \eqref{bound2}}
In this part, we prove \eqref{bound3} and \eqref{bound2}. The key step is to obtain a lower bound on the one-step transition probabilities of the process $X_j$ using \eqref{eq:liminf}. This allows us to stochastically dominate  
$T_{\lfloor 2(\log n)^{1+\frac{\varepsilon}{2}}\rfloor}
      - T_{\lfloor(\log n)^{\frac{3+2\varepsilon}{6}}\rfloor}$  
and  
$T_{\lfloor(\log n)^{\frac{3+2\varepsilon}{6}}\rfloor} - T_M$  
by sums of independent geometric random variables.  
Consequently, with high probability,
both increments can be made arbitrarily small.

Define for $n \ge 1$, $0 \le s \le h(q_n)$ and $1\le m\le q_n(s)$,
\[
p_n(m,s)
:= \p_n\!\left(\#\{\mathfrak{p}(u^{s}_1),\ldots,\mathfrak{p}(u^{s}_m)\}\le m-1\right).
\]
By Theorem \ref{thm:Markov}, this quantity also equals $\p_n(X_{s-1}\le m-1\,|\,X_s=m)$ for any coalescent process $X$ based at a height greater than $s$.

\begin{lemma}
    Assume \eqref{H1} and \eqref{eq:liminf}. For every $\delta>0$, there exists $\varepsilon'>0$ such that uniformly for all $n\ge 1$, 
    $\delta n\le s\le h(q_n)-\delta n$ and $2\le m\le (\log n)^{\frac{1+\varepsilon'}{2}}$,
    \begin{align}\label{gelog1}
        p_n\!\left(m,s\right)
        \gtrsim \frac{m^2}{n}.
    \end{align}
\end{lemma}

\begin{proof}
By \eqref{eq:liminf}, for each $\delta>0$ there exist $N=N(\delta)$ such that
\begin{align}\label{eq:til_sigma}
	\tilde{\sigma}^2 := \inf_{n\ge N,\,\delta n\le s\le h(q_n)-\delta n}
	\e_n \bigg[(\nu^s_1)_2\mathbf{1}\bigg\{\nu_1^s\le  \frac{n}{(\log n)^{1+\varepsilon}} \bigg\}\bigg]>0.
\end{align} 
Define $K_n=\lfloor \frac{n}{(\log n)^{1+\varepsilon}}\rfloor$ and truncated coalescent events
\[
B_{i_1i_2} := \big\{\mathfrak{p}(u^s_{i_1}) = \mathfrak{p}(u^s_{i_2}),\, \xi(\mathfrak{p}(u^s_{i_1})) \le K_n\big\}, 
\qquad 1\le i_1 < i_2\le m.
\]
Then
\begin{align}\label{eq:inclusion}
    \Big\{\#\{\kp(u^s_i):1\le i\le m\}\le m-1\Big\}
    = \bigcup_{1 \le i_1 < i_2 \le m}\big\{\mathfrak{p}(u^s_{i_1}) = \mathfrak{p}(u^s_{i_2})\big\}
    \supset \bigcup_{1 \le i_1 < i_2 \le m} B_{i_1i_2}.
\end{align}
Applying the inclusion--exclusion principle, we get
\begin{align*}
	\p_n\Big(\bigcup_{1 \le i_1 < i_2 \le m} B_{i_1i_2}\Big) 
	&\ge \binom{m}{2} \p_n(B_{12}) 
	- 3 \binom{m}{3} \p_n(B_{12} \cap B_{13}) 
	- 6 \binom{m}{4} \p_n(B_{12} \cap B_{34}).
\end{align*}
By Corollary~\ref{coleprob}, for all $n\ge N$,
\begin{align*}
    \p_n(B_{12}) &= \frac{q_n(s-1)}{(q_n(s))_2} \,
        \E_n\!\big[(\nu^{s-1}_1)_2 \mathbf{1}\{\nu^{s-1}_1 \le K_n\}\big]
        \;\ge\; \frac{q_n(s-1)}{(q_n(s))_2}\, \tilde{\sigma}^2, \\[0.5em]
    \p_n(B_{12} \cap B_{13}) &= \frac{q_n(s-1)}{(q_n(s))_3} \,
        \E_n\!\big[(\nu^{s-1}_1)_3 \mathbf{1}\{\nu^{s-1}_1 \le K_n\}\big]
        \;\le\; \frac{q_n(s-1)}{(q_n(s))_3}\, K_n \E_n\!\big[(\nu^{s-1}_1)_2 \big], \\[0.5em]
    \p_n(B_{12} \cap B_{34}) &\le \frac{q_n(s-1)}{(q_n(s))_4} \,
        \E_n\!\big[(\nu^{s-1}_1)_4 \mathbf{1}\{\nu^{s-1}_1 \le K_n\}\big]  + \frac{(q_n(s-1))_2}{(q_n(s))_4} \,
        \E_n\!\big[(\nu^{s-1}_1)_2 (\nu^{s-1}_2)_2
        \mathbf{1}\{\nu^{s-1}_1, \nu^{s-1}_2 \le K_n\}\big] \\
    &\le \frac{q_n(s-1)}{(q_n(s))_4}K_n^2 \E_n\!\big[(\nu^{s-1}_1)_2 \big] + \frac{(q_n(s-1))_2}{(q_n(s))_4}\, K_n \E_n\!\big[(\nu^{s-1}_1)_2\, \nu^{s-1}_2\big].
\end{align*}
We recall that $\E_n\!\big[(\nu^{s-1}_1)_2\, \nu^{s-1}_2\big]\le \frac{q_n(s)}{q_n(s-1)-1} \E_n\!\big[(\nu^{s-1}_1)_2\big]$ by \eqref{ind2}.
Hence by \eqref{H1}, there exist constants $c_i=c_i(\delta,\tilde{\sigma},N)>0$ $(i=1,2,3)$ such that, for all $n\ge N$, $\delta n\le s\le  h(q_n)-\delta n$ and $1\le m\le q_n(s)$,
\[
    \p_n\Big(\bigcup_{1\le i<j\le m} B_{ij}\Big)
    \;\ge\; c_1 \frac{m^2}{n} - c_2 \frac{m^3}{n(\log n)^{1+\varepsilon}} - c_3 \frac{m^4}{n(\log n)^{1+\varepsilon}}.
\]
Combining this with \eqref{eq:inclusion} and taking $\varepsilon'\in(0,\varepsilon)$, we obtain \eqref{gelog1}. 
\end{proof}

We are now ready to prove \eqref{bound3} and \eqref{bound2}.
Define
\[
\zeta_m := 
\begin{cases}
	\inf\{i \ge 1: X_{T_m-i}<m\}, & \text{if } X_{T_m}=m,\\
	0, & \text{otherwise}.
\end{cases}
\]
Then, for any \( m_1 < m_2 \),
\begin{align}\label{eq:zeta}
	T_{m_2} - T_{m_1} = \sum_{m = m_1 + 1}^{m_2} \zeta_m.
\end{align}
By the Markov property, conditionally on $\mathcal{G}_j$ with $T_m=j$ and $X_{T_m}=m$, 
\begin{align}\label{eq:zeta1}
    \zeta_m \overset{d}{=} \inf\{1\le i\le j : Y_i=1\},
\end{align}
where $Y_i$ are independent Bernoulli random variables with success probability
\[
    \p(Y_i=1)=\p_n(X_{j-i} \le m-1 \mid X_{j-i+1}=m).
\]
By \eqref{gelog1} with $\varepsilon'=2\varepsilon/3$, there exists constant $C^*>0$ such that for all $n\ge 1$, $h_*\le i\le h^*$ and $2\le m\le \lfloor (\log n)^{\frac{3+2\varepsilon}{6}}\rfloor$,
\begin{equation}\label{ydelta}
     \p(Y_i=1)\ge C^*\, \frac{m^2}{n}.
\end{equation}

We next record two auxiliary lemmas to be used in the proof.
\begin{lemma}\label{lem:auxiliary1}
    Assume \eqref{H1} and \eqref{eq:liminf}. Conditionally on $T_{\lfloor (\log n)^{\frac{3+2\varepsilon}{6}}\rfloor}> T_{\lfloor2(\log n)^{1+\frac{\varepsilon}{2}}\rfloor} -\tfrac13 \delta n$ and $T_{\lfloor 2(\log n)^{1+\frac{\varepsilon}{2}}\rfloor}>  h^*-\tfrac{1}{3}\delta n$, we have the stochastic bound:
\begin{equation}\label{zetasto}
     \zeta_m\wedge\left(\frac{1}{3}\delta n-\sum_{j=m+1}^{\lfloor (\log n)^{\frac{3+2\varepsilon}{6}}\rfloor}\zeta_j\right)^+\;\le_{\rm sto}\; {\rm Geometric}\bigg(\frac{C^* m^2}{n}\bigg),\quad \forall\, 2\le m<\lfloor (\log n)^{\frac{3+2\varepsilon}{6}}\rfloor.
\end{equation}
\end{lemma}
\begin{lemma}\label{lem:auxiliary2}
    If $T_{\lfloor (\log n)^{\frac{3+2\varepsilon}{6}}\rfloor}-T_M=\sum_{m=M+1}^{\lfloor(\log n)^{\frac{3+2\varepsilon}{6}}\rfloor}\zeta_m\ge\tfrac{1}{3}\delta n$, then
    \begin{align}\label{eq:obs12}
        \sum_{m=M+1}^{\lfloor(\log n)^{\frac{3+2\varepsilon}{6}}\rfloor} \zeta_m\wedge\left(\frac{1}{3}\delta n-\sum_{j=m+1}^{\lfloor (\log n)^{\frac{3+2\varepsilon}{6}}\rfloor}\zeta_j\right)^+\ge\frac{1}{3}\delta n.
    \end{align}
\end{lemma}

\begin{proof}[Proof of \eqref{bound2} assuming Lemmas \eqref{lem:auxiliary1} and \eqref{lem:auxiliary2}]
Set 
\[
A=A_n :=\left\{T_{\lfloor (\log n)^{\frac{3+2\varepsilon}{6}}\rfloor}> T_{\lfloor2(\log n)^{1+\frac{\varepsilon}{2}}\rfloor} -\tfrac13 \delta n ,\, T_{\lfloor 2(\log n)^{1+\frac{\varepsilon}{2}}\rfloor}>  h^*-\tfrac{1}{3}\delta n\right\}.
\]
Then
\begin{align}\label{eq:geometric}
\begin{aligned}
    &\quad\,\p_n\Big(T_{\lfloor(\log n)^{\frac{3+2\varepsilon}{6}}\rfloor}-T_M\ge\tfrac{1}{3}\delta n \,\Big|\, A\Big)=\p_n\bigg(\sum_{m=M+1}^{\lfloor(\log n)^{\frac{3+2\varepsilon}{6}}\rfloor}\zeta_m\ge\tfrac{1}{3}\delta n \,\bigg|\, A\bigg)\\
    &\overset{\eqref{eq:obs12}}{\le}\p_n\left(\sum_{m=M+1}^{\lfloor(\log n)^{\frac{3+2\varepsilon}{6}}\rfloor} \zeta_m\wedge\left(\frac{1}{3}\delta n-\sum_{j=m+1}^{\lfloor (\log n)^{\frac{3+2\varepsilon}{6}}\rfloor}\zeta_j\right)^+\ge\tfrac{1}{3}\delta n \,\bigg|\, A\right)\\
    &\le     \frac{3}{\delta n}\e_n\left(\sum_{m=M+1}^{\lfloor(\log n)^{\frac{3+2\varepsilon}{6}}\rfloor} \zeta_m\wedge\left(\frac{1}{3}\delta n-\sum_{j=m+1}^{\lfloor (\log n)^{\frac{3+2\varepsilon}{6}}\rfloor}\zeta_j\right)^+ \,\bigg|\, A\right)\overset{\eqref{zetasto}}{\le}\frac{3}{\delta n} \sum_{m=M+1}^{\lfloor(\log n)^{\frac{3+2\varepsilon}{6}}\rfloor}\frac{n}{C^* m^2}\le \frac{3}{C^* \delta M},
\end{aligned}
\end{align}
for some constant $c>0$.
This gives 
\eqref{bound2} once $M$ is chosen sufficiently 
\end{proof}

\begin{proof}[Proof of \eqref{bound3}]
    Observe that $\p_n(X_j \le m-1 \mid X_{j-1}=m) = \p_n\big(\#\{\kp(u^{j-1}_i):1\le i\le m\}\le m-1\big)$ is increasing in $m$. In particular, for any $n\ge 1$, 
    $\delta n\le s\le h(q_n)-\delta n$ and $(\log n)^{\frac{3+2\varepsilon}{6}}\le m<2(\log n)^{1+\frac{\varepsilon}{2}}$,
\[
    p_n(m,s)
    \;\ge\; C^*\dfrac{(\log n)^{1+\frac23\eps}}{n}.
\]
Repeating the same argument as above yields
\begin{align}\label{eq:geometric2}
\begin{split}
   &\p_n( T_{\lfloor (\log n)^{\frac{3+2\varepsilon}{6}}\rfloor}\le T_{\lfloor2(\log n)^{1+\frac{\varepsilon}{2}}\rfloor} -\tfrac13 \delta n \,\Big|\, T_{\lfloor 2(\log n)^{1+\frac{\varepsilon}{2}}\rfloor}>  h^*-\tfrac{1}{3}\delta n\Big)\\
   &\;\le\;\frac{3}{\delta n} \sum_{m=\lfloor(\log n)^{\frac{3+2\varepsilon}{6}}\rfloor+1}^{\lfloor2(\log n)^{1+\frac{\varepsilon}{2}}\rfloor}\dfrac{n}{C^*(\log n)^{1+\frac23\eps}}\le \frac{6}{C^*\delta (\log n)^{\frac{\varepsilon}{6}}},
   \end{split}
\end{align}
which implies \eqref{bound3}.
\end{proof}
\begin{proof}[Proof of Lemma \ref{lem:auxiliary1}]
    It is enough to show that conditionally on $\big(X_j:T_m\le j\le h^*\big)$ with $X_{T_m}=m$, $T_{\lfloor (\log n)^{\frac{3+2\varepsilon}{6}}\rfloor}> T_{\lfloor2(\log n)^{1+\frac{\varepsilon}{2}}\rfloor} -\tfrac13 \delta n ,\  T_{\lfloor 2(\log n)^{1+\frac{\varepsilon}{2}}\rfloor}>  h^*-\tfrac{1}{3}\delta n$ and $\frac{1}{3}\delta n-\sum_{j=m+1}^{\lfloor (\log n)^{\frac{3+2\varepsilon}{6}}\rfloor}\zeta_j>0$,
    \begin{align}\label{eq:sto12}
        \zeta_m\wedge\left(\frac{1}{3}\delta n-\sum_{j=m+1}^{\lfloor (\log n)^{\frac{3+2\varepsilon}{6}}\rfloor}\zeta_j\right)\;\le_{\rm sto}\; {\rm Geometric}\bigg(\frac{C^* m^2}{n}\bigg).
    \end{align}
  Note that under the conditioning, we have the following bound for the height $T_m$:
    \begin{align*}
        T_m=T_{\lfloor (\log n)^{\frac{3+2\varepsilon}{6}}\rfloor}-\sum_{i=j+1}^{\lfloor (\log n)^{\frac{3+2\varepsilon}{6}}\rfloor}\zeta_j&> h^*-\delta n + \left(\frac{1}{3}\delta n-\sum_{j=m+1}^{\lfloor (\log n)^{\frac{3+2\varepsilon}{6}}\rfloor}\zeta_j\right)\overset{\eqref{eq:h*condition2}}{>} h_*+\left(\frac{1}{3}\delta n-\sum_{j=m+1}^{\lfloor (\log n)^{\frac{3+2\varepsilon}{6}}\rfloor}\zeta_j\right).
    \end{align*}
    Therefore the LHS of \eqref{eq:sto12} is greater than $h_*$, and thus \eqref{eq:sto12} follows from \eqref{eq:zeta1} and \eqref{ydelta}. \pagebreak
\end{proof}

\begin{proof}[Proof of Lemma \ref{lem:auxiliary2}]
    Assume 
\(
\sum_{m=M+1}^{\lfloor (\log n)^{\frac{3+2\varepsilon}{6}}\rfloor}\zeta_m \ge \tfrac{1}{3}\delta n.
\)
Define
\[
m^*:=\sup\Big\{\,m\ge M+1:\ \sum_{j=m}^{\lfloor (\log n)^{\frac{3+2\varepsilon}{6}}\rfloor}\zeta_j\ge \tfrac{1}{3}\delta n\,\Big\}.
\]
Then by definition,
\[
\sum_{j=m^*}^{\lfloor (\log n)^{\frac{3+2\varepsilon}{6}}\rfloor}\zeta_j \ge \frac{1}{3}\delta n\quad\text{and}
\quad
\sum_{j=m}^{\lfloor (\log n)^{\frac{3+2\varepsilon}{6}}\rfloor}\zeta_j < \frac{1}{3}\delta n
\quad \text{for all } m>m^*.
\]

   \noindent Therefore, by separating the contributions of the cases $m=m^*$ and $m\ge m^*+1$, we derive that 
    \begin{align*}
        \sum_{m=m^*}^{\lfloor(\log n)^{\frac{3+2\varepsilon}{6}}\rfloor} \zeta_m\wedge\left(\frac{1}{3}\delta n-\sum_{j=m+1}^{\lfloor (\log n)^{\frac{3+2\varepsilon}{6}}\rfloor}\zeta_j\right)^+ 
        = \left(\frac{1}{3}\delta n-\sum_{j=m^*+1}^{\lfloor (\log n)^{\frac{3+2\varepsilon}{6}}\rfloor}\zeta_j\right)^+  +  \sum_{m=m^*+1}^{\lfloor(\log n)^{\frac{3+2\varepsilon}{6}}\rfloor} \zeta_m \ge \frac{1}{3}\delta n.
    \end{align*}
    This proves \eqref{eq:obs12}.
\end{proof}

\subsubsection{Proof of \eqref{bound1}}
We close this section with the proof of \eqref{bound1}. The main technical estimate we require is the following proposition.
\begin{proposition}\label{prop:rely1}
    Assume \eqref{H1} and \eqref{eq:liminf}. Then there exists constant $c>0$ such that for any $n$ sufficiently large, $h_*+\frac{n}{(\log n)^{1+\frac{\varepsilon}{8}}}\le  s \le h^*$ and $1\le k\le q_n(s)$, we have
    \begin{equation}\label{kinequal}
    \p_n\left(X_{s-{\frac{n}{(\log n)^{1+\frac{\varepsilon}{8}}}}}\le \dfrac{k}{(\log n)^{\frac{\varepsilon}{6}}}+(\log n)^{1+\frac{\varepsilon}{2}}\,\Big|\, X_s=k\right)\ge 1-c(\log n)^{-\frac{13}{24}\varepsilon}.
    \end{equation}
\end{proposition}
\begin{proof}[Proof of \eqref{bound1} assuming Proposition \ref{prop:rely1}]
Set 
\[l_i=h^*-\dfrac{n\cdot i}{(\log n)^{1+\frac{\varepsilon}{8}}}\] 
for $i=0,\cdots, \lfloor\frac{1}{3}\delta(\log n)^{1+\frac{\varepsilon}{8}}\rfloor$.
We use the following elementary fact:
\begin{lemma}\label{lem:rely2}
    For sufficiently large $n$, we have if 
    \begin{align*}
        \#\left\{i=0,\ldots,\left\lfloor\frac{1}{3}\delta(\log n)^{1+\frac{\varepsilon}{8}}\right\rfloor: X_{l_{i+1}}\le \dfrac{X_{l_{i}} }{(\log n)^{\frac{\varepsilon}{6}}}+(\log n)^{1+\frac{\varepsilon}{2}}\right\}\ge \frac{1}{6}\delta (\log n)^{1+\frac{\varepsilon}{8}},
    \end{align*}
    then $X_{h^*-\lfloor\frac{1}{3}\delta n\rfloor}< 2(\log n)^{1+\frac{\varepsilon}{2}}$.
\end{lemma}
Applying this, we obtain for $n$ large enough,
    \begin{align}\label{eq:qwer}
        \begin{aligned}
        &\quad\, \p_n\left(T_{\lfloor2(\log n)^{1+\frac{\varepsilon}{2}}\rfloor}\le h^*-\frac{1}{3}\delta n\right) \le \p_n\left(X_{h^*-\lfloor\frac{1}{3}\delta n\rfloor}\ge 2(\log n)^{1+\frac{\varepsilon}{2}}\right)\\
&\le\p_n\bigg(\#\left\{i=0,\ldots,\left\lfloor\frac{1}{3}\delta(\log n)^{1+\frac{\varepsilon}{8}}\right\rfloor: X_{l_{i+1}}\le \dfrac{X_{l_{i}} }{(\log n)^{\frac{\varepsilon}{6}}}+(\log n)^{1+\frac{\varepsilon}{2}}\right\}< \frac{1}{6}\delta (\log n)^{1+\frac{\varepsilon}{8}}\bigg).
   \end{aligned}
    \end{align}
    It follows from the Markov property of $X$ in Theorem \ref{thm:Markov} and \eqref{kinequal} that there exists constant $c>0$ such that for any $n$ sufficiently large and $i=0,\ldots,\left\lfloor\frac{1}{3}\delta(\log n)^{1+\frac{\varepsilon}{8}}\right\rfloor$, 
    \[\mathbb{P}_n\left(X_{l_{i+1}}\le \dfrac{X_{l_{i}} }{(\log n)^{\frac{\varepsilon}{6}}}+(\log n)^{1+\frac{\varepsilon}{2}}\,\Big|\,(X_j:l_i\le j\le h^*)\right)\ge 1-c(\log n)^{-\frac{13}{24}\varepsilon},\quad \text{a.s..}\] 
This implies that the RHS of \eqref{eq:qwer} goes to $0$ as $n\rightarrow\infty$, completing the proof of \eqref{bound1}.
\end{proof}
The remaining part is devoted to the proof of Proposition \ref{prop:rely1}.
Fix $h_*+\frac{n}{(\log n)^{1+\frac{\varepsilon}{8}}}\le s \le h^*$.
Consider the coalescent process 
\[
X^{(r)}_j = X^{(r)}_j(n,s),
\]
and define the hitting time
\[
T^{(r)}_m = T^{(r)}_m(s) := \sup\{0\le j \le s : X^{(r)}_j \le m\}.
\]
Denote
\(
r=r(n):=\lfloor(\log n)^{1+\frac23\varepsilon}\rfloor
\) throughout the proof. For $h_*+\frac{n}{(\log n)^{1+\frac{\varepsilon}{8}}}\le  s \le h^*$, denote
\[s^{\downarrow}=s^{\downarrow}(n):= s - \frac{n}{\lfloor (\log n)^{1+\frac{\varepsilon}{8}} \rfloor}.\]   
We will need the following CDFI-type estimate, whose proof is deferred to the end of this section.
\begin{lemma}\label{lem:tdecrease} 
Assume \eqref{H1} and \eqref{eq:liminf}. Then there exists constant $c>0$ such that for all $n$ sufficiently large and $h_*+\frac{n}{(\log n)^{1+\frac{\varepsilon}{8}}}\le s \le h^*$,
\[
\p_n\!\left(T^{(r)}_{\lfloor (\log n)^{1+\frac{\varepsilon}{2}} \rfloor}(s) > s^{\downarrow}\right)
\;\ge\;
1 - c(\log n)^{-\frac{13}{24}\varepsilon}.
\]
\end{lemma}
\begin{proof}[Proof of Proposition \ref{prop:rely1} assuming Lemma \ref{lem:tdecrease}]
For $1 \le k \le (\log n)^{1+\frac23\varepsilon}$, the bound \eqref{kinequal} follows by Lemma~\ref{lem:tdecrease}.  
We therefore focus on the case $k > (\log n)^{1+\frac23\varepsilon}$. 
By Theorem \ref{thm:Markov}, it suffices to prove that
\begin{equation}\label{kinequal1}
    \p_n\left(X^{(k)}_{s^{\downarrow}}\le \dfrac{k}{(\log n)^{\frac{\varepsilon}{6}}}+(\log n)^{1+\frac{\varepsilon}{2}}\right)\ge 1-c(\log n)^{-\frac{13}{24}\varepsilon}.
    \end{equation}
    
We first make an observation concerning the coalescent process $X^{(r)}$.  
Recall that $r=\lfloor(\log n)^{1+\frac23\varepsilon}\rfloor$.  
Since $(u^{s}_1,\ldots,u^{s}_{q_n(s)})$ is a uniformly random permutation of $(x^{s}_1,\ldots,x^{s}_{q_n(s)})$, we can write
\[
(u^{s}_1,\ldots,u^{s}_{q_n(s)})=(x^{s}_{\pi^n_1},\ldots,x^{s}_{\pi^n_{q_n(s)}}),
\]
where $\pi^n$ is a uniform permutation in the symmetric group $S_{q_n(s)}$.  
Let $\pi^n_{(1)}<\cdots<\pi^n_{(r)}$ denote the order statistics of $\{\pi^n_1,\ldots,\pi^n_r\}$.
For any $1 \le m \le r$, the event $\{X^{(r)}_{s^{\downarrow}} = m\}$ means that there exist indices
\[
0 = i_0 < i_1 < \cdots < i_m = r
\]
and $m$ distinct vertices $v_1\prec\cdots\prec v_m \in \cX^{(r)}_{s^{\downarrow}}$ such that for each $q = 1,\ldots,m$ and every
$t \in \{i_{q-1}+1,\ldots,i_q\}$,
\[
\text{Anc}(x^s_{\pi^n_{(t)}},\, s^{\downarrow}) = v_q.
\]
That is, the $r$ sampled vertices possess exactly 
$m$ distinct ancestors, and these ancestral identities partition
the index set $\{\pi^n_{(1)},\ldots,\pi^n_{(r)}\}$ into $m$ consecutive blocks
\[
I_q := \{\pi^n_{(i_{q-1}+1)},\ldots,\pi^n_{(i_q)}\}, \qquad q = 1,\ldots,m,
\]
such that all samples in $I_q$ share the same ancestor $v_q$ at height $s^{\downarrow}$.
On the event $\{X^{(r)}_{s^{\downarrow}} = m\}$, this gives a natural decomposition of all vertices at height $s$:
\begin{align*}
    B_1 &:= \{x^s_i : \text{Anc}(x^s_i, s^{\downarrow}) \in \{v_1,\ldots,v_m\}\}
           = \{x^s_i : i \in \cup_{q=1}^m I_q\},\\
    B_2 &:= \{x^s_1,\ldots,x^s_{q_n(s)}\} \setminus B_1.
\end{align*}

\noindent We now return to the coalescent process $X^{(k)}$.  
Recall that
\[
X^{(k)}_{s^{\downarrow}}
    = \#\{\text{Anc}(u^s_i, s^{\downarrow}): 1 \le i \le k\}.
\]
By splitting the $k$ sampled vertices according to whether $u^s_i \in B_1$ or $u^s_i \in B_2$, we obtain
\begin{align}\label{eq:B1B2}
    X^{(k)}_{s^{\downarrow}}
    \le
    X^{(r)}_{s^{\downarrow}}
    + \#\big( B_2 \cap \{u^s_i : 1 \le i \le k\} \big).
\end{align}
It follows from Lemma \ref{lem:tdecrease} that
\begin{align}\label{eq:Xr}
    \mathbb{P}_n\left(X^{(r)}_{s^{\downarrow}}\le \lfloor (\log n)^{1+\frac{\varepsilon}{2}} \rfloor\right)\ge \p_n\!\left(T^{(r)}_{\lfloor (\log n)^{1+\frac{\varepsilon}{2}} \rfloor} > s - \frac{n}{\lfloor (\log n)^{1+\frac{\varepsilon}{8}} \rfloor}\right)\ge  1-e^{-c(\log n)^{\frac{\eps}{12}}}.
\end{align}
To complete the proof, it suffices to prove that 
\[
\mathbb{P}_n\left[\#\big( B_2 \cap \{u^s_i : 1 \le i \le k\} \big)\le \frac{k}{(\log n)^{\frac\eps6}}\right]\ge 1-c(\log n)^{-\frac{13}{24}\varepsilon}.
\]
This follows from \eqref{eq:Xr} and the following lemma.
\end{proof}
\begin{lemma}\label{lem:key_bound_B2}
There exists constant $c>0$ such that for any $n$ sufficiently large, $h_*+\frac{n}{(\log n)^{1+\frac{\varepsilon}{8}}}\le s \le h^*$, $(\log n)^{1+\frac23\varepsilon}\le k\le q_n(s)$ and $1 \le m \le \lfloor (\log n)^{1+\frac{\varepsilon}{2}} \rfloor$,
\begin{align}\label{eq:condition_B2}
    \p_n\!\left(
        \#\big( B_2 \cap \{u^s_i : 1 \le i \le k\} \big)
        > \frac{k}{(\log n)^{\frac\eps6}}
        \,\Big|\, X^{(r)}_{s^{\downarrow}} = m
    \right)
    \le e^{-c(\log n)^{\frac\eps4}}.
\end{align}
\end{lemma}
\begin{proof}
We use the following elementary property of uniform permutations.

\begin{lemma}\label{lem:perm_gap}
Let $\pi^n_{(0)} := 0$ and $\pi^n_{(r+1)} := q_n(s)$.  
Then there exists constant $c' > 0$ such that for all $n \ge 1$ and 
$h_* + \frac{n}{(\log n)^{1+\frac{\varepsilon}{8}}} \le s \le h^*$,
\[
    \p\!\left(
        \max_{0 \le i \le r} (\pi^n_{(i+1)} - \pi^n_{(i)})
        \le \frac{n}{(\log n)^{1+\frac34 \varepsilon}}
    \right)
    \ge 1 - e^{-c' (\log n)^{\frac{\varepsilon}{4}}}.
\]
\end{lemma}

On the event $\{X^{(r)}_{s^{\downarrow}} = m\}$,
\[
    \#B_2
    = \sum_{q=1}^m \# I_q
    \le m \cdot \max_{0 \le i \le r}
        (\pi^n_{(i+1)} - \pi^n_{(i)}),
\]
and hence by Lemma \ref{lem:perm_gap}, for any $1 \le m \le \lfloor (\log n)^{1+\frac{\varepsilon}{2}} \rfloor$,
\begin{align}\label{eq:B2_est}
    \p_n\!\left(\#B_2 \le \frac{n}{(\log n)^{\frac{\varepsilon}{4}}}\,\Big|\, X^{(r)}_{s^{\downarrow}} = m\right)\ge \p_n\!\left(\#B_2 \le \frac{m}{(\log n)^{1+\frac{3}{4}\varepsilon}}\,\Big|\, X^{(r)}_{s^{\downarrow}} = m\right)
    \ge 1 - e^{-c'(\log n)^{\frac{\varepsilon}{4}}}.
\end{align}

Finally, note that $(u^s_1,\ldots,u^s_r) \subset B_1$, and conditionally on $(u^s_1,\ldots,u^s_r)$, the remaining
$(u^s_{r+1},\ldots,u^s_k)$ are chosen uniformly among the other
$q_n(s) - r$ vertices at height $s$.  
Conditionally on $\#B_2$, the random variable
\(
\#(B_2 \cap \{u^s_i : 1 \le i \le k\})
\)
follows a hypergeometric distribution 
with population size $q_n(s)$, number of successes $\#B_2$ 
and sample size $k$.
By the tail estimate of the hypergeometric distribution (Cf. \cite[Corollary 1.1]{Serfling}), we derive that there exists constant $c''>0$ such that
\[
\p_n\!\left(
        \#\big( B_2 \cap \{u^s_i : 1 \le i \le k\} \big)
        > \frac{k}{(\log n)^{\frac\eps6}}
        \,\Big|\, X^{(r)}_{s^{\downarrow}} = m,\, \#B_2 \le \frac{n}{(\log n)^{\frac{\varepsilon}{4}}}
    \right)
    \le e^{-c'' k(\log n)^{-\frac\eps6}}.
\]
for all $n$ sufficiently large and $s,k,m$ satisfying the condition in the lemma.
This together with \eqref{eq:B2_est} yields \eqref{eq:condition_B2}. 
\end{proof}
\end{proof}

Finally, let us prove Lemma~\ref{lem:tdecrease}. The key observation is that, to establish the CDFI property, it is already sufficient to figure out the contribution to the decrease of the coalescent process coming from mergers whose common parent produces only a “small’’ number of offspring (see \eqref{eq:tilde Nj} for the definition of such mergers).
For such mergers, the reduction in the number of lineages is, up to a factor of order $(\log n)^{\eta}$ (with $\eta>0$ arbitrarily small), comparable to the number of pairs of sampled individuals that share a common parent.  
Hence the task of estimating the total decrease of $X_j$ can be reduced to estimating the number of these coalescing pairs—a quantity that is much easier to estimate.  
The proof is then completed by obtaining suitable bounds on the expectation and variance of the number of such coalescing pairs.

\medskip

Let us begin with the notation.
\begin{definition}[Sampled vertex sets]\label{def:sampled-X}
    Fix $h_* \le j \le h^*$.  Conditionally on the tree $\mathscr{T}^n$, we define the sampled vertex set $\widetilde{\cX}_j$ as follows:
    \begin{itemize}
        \item If $X_j\ge  \lfloor(\log n)^{1+\frac{\varepsilon}{2}}\rfloor$, let $\widetilde{\cX}_j$ be a uniformly chosen 
        $\lfloor(\log n)^{1+\frac{\varepsilon}{2}}\rfloor$-tuple of $\cX_j$;
        \item Otherwise, let $\widetilde{\cX}_j$ be a uniformly chosen $\lfloor(\log n)^{1+\frac{\varepsilon}{2}}\rfloor$-tuple of $\{x^j_1,\ldots,x^j_{q_n(j)}\}$.
    \end{itemize}
    The choices are (conditionally) independent across heights. 
\end{definition}

Furthermore, for $h_* \le j < h^*$ and $1\le l\le q_n(j)$, define
\begin{align}\label{eq:xi_tilde}
    \widetilde\xi(x^j_l)
      := \#\big\{x\in \widetilde{\cX}_{j+1}:\ \kp(x)=x^j_l\big\},
\end{align}
the number of sampled children of $x^j_l$ among the vertices in $\widetilde{\cX}_{j+1}$, and
\begin{align}\label{eq:tilde Nj}
    \widetilde N_j
    := \sum_{\substack{x,x'\in\widetilde{\cX}_j\\ x\neq x'}}
        \mathbf{1}\!\left\{\kp(x)=\kp(x'),
        \ \xi(\kp(x))\le \Big\lfloor \frac{n}{(\log n)^{1+\frac{2\varepsilon}{3}}}\Big\rfloor
        \right\},
\end{align}
which counts the sampled sibling pairs in $\widetilde{\cX}_j$ whose common parent has at most 
\(\lfloor n/(\log n)^{1+\frac23\varepsilon} \rfloor\) children.

We prepare several auxiliary lemmas on $\widetilde\xi(x^j_l)$ and $\widetilde{N}_j$, whose proofs are deferred to the end of this section.
\begin{lemma}\label{lem:tilde_N}
    $\big(\widetilde{N}_j:h_*< j\le h^*\big)$ are independent. Each $\widetilde{N}_j$ follows the same law as 
    \begin{align}\label{def:tilde_N'}
        \widetilde N_j':= \sum_{1\le l_1<l_2\le \lfloor (\log n)^{1+\frac\varepsilon 2}\rfloor } \mathbf{1}\Big\{\kp(u^j_{l_1})=\kp(u^j_{l_2}),\, \xi(\kp(u^j_{l_1}))\le \Big\lfloor \frac{n}{(\log n)^{1+\frac23 \varepsilon}} \Big\rfloor \Big\}.
    \end{align}
\end{lemma}
\begin{lemma}\label{lem:Xi}
Assume \eqref{H1} and \eqref{eq:liminf}. Let
\begin{align*}
    \Xi_n:=\left\{\exists\, s\in[h_*,h^*],\, l\in[1, q_n(s)] \text{ s.t. } \xi(x^s_l)\le \frac{n}{(\log n)^{1+\frac23 \varepsilon}},\, \widetilde\xi(x^s_l)\ge (\log n)^{\frac\varepsilon{12}}\right\}.
\end{align*}
Then we have
\begin{align}\label{eq:Xi}
    \p_n\left(\Xi_n\right)= o_n\big( e^{-(\log n)^{\frac\varepsilon{12}}}\big).
\end{align}
\end{lemma}
Recall that we write 
\(
r=\lfloor(\log n)^{1+\frac23\varepsilon}\rfloor
\), $s^{\downarrow}=s - \frac{n}{\lfloor (\log n)^{1+\frac{\varepsilon}{8}} \rfloor}$ and $X_j^{(r)}=X_j^{(r)}(n,s)$.
\begin{lemma}\label{lem:X_N}
    For any $\delta>0$, $n\ge 1$ and $h_*+\frac{n}{(\log n)^{1+\frac{\varepsilon}{8}}}\le s \le h^*$, we have on the event $(\Xi_n)^{\rm c}\cap \Big\{T^{(r)}_{\lfloor (\log n)^{1+\frac{\varepsilon}{2}} \rfloor}(s) < s^{\downarrow} \Big\}$,
\begin{align}\label{eq:tilde N good}
    X^{(r)}_s - X^{(r)}_{s^{\downarrow}} \ge \frac{2}{(\log n)^{\frac\varepsilon{12}}}\sum_{j=s^{\downarrow}+1}^{s} \widetilde N_j.
\end{align}
\end{lemma}
\begin{proof}[Proof of Lemma~\ref{lem:tdecrease} assuming Lemmas~\ref{lem:tilde_N}--\ref{lem:X_N}]
Note that
\[\Big\{T^{(r)}_{\lfloor (\log n)^{1+\frac{\varepsilon}{2}} \rfloor}(s) < s^{\downarrow} \Big\}=\left\{X^{(r)}_s - X^{(r)}_{s^{\downarrow}}<r-\lfloor (\log n)^{1+\frac{\varepsilon}{2}} \rfloor\right\}.\]
Thus, by \eqref{eq:Xi} and \eqref{eq:tilde N good},
\begin{align*}
    \p_n\Big(T^{(r)}_{\lfloor (\log n)^{1+\frac{\varepsilon}{2}} \rfloor}(s) < s^{\downarrow}\Big)&\le o_n\big( e^{-(\log n)^{\frac\varepsilon{12}}}\big) + \p_n\bigg(\frac{2}{(\log n)^{\frac\varepsilon{12}}}\sum_{j=s^{\downarrow}+1}^{s} \widetilde N_j< r\bigg).
\end{align*}
It remains to prove
\begin{align}\label{eq:remain}
    \p_n\bigg(\frac{2}{(\log n)^{\frac\varepsilon{12}}}\sum_{j=s^{\downarrow}+1}^{s} \widetilde N_j< r\bigg)\lesssim (\log n)^{-\frac{13}{24}\varepsilon}.
\end{align}
For this, we shall estimate the expectation and variance of $\sum_{j=s^{\downarrow}+1}^{s}\widetilde N_j$. 
\begin{proposition}
    For any $\delta>0$, we have uniformly for $h_*+\frac{n}{(\log n)^{1+\frac{\varepsilon}{8}}}\le s \le h^*$,
    \begin{align}\label{eq:exp_var}
        \e_n\left[\sum_{j=s^{\downarrow}+1}^s \widetilde N_j\right]\asymp (\log n)^{1+\frac78 \varepsilon},\qquad
        \text{Var}_n\left[\sum_{j=s^{\downarrow}+1}^{s}\widetilde{N}_j \right]\lesssim (\log n)^{2+\frac{29}{24}\varepsilon}.
    \end{align}
\end{proposition}
\begin{proof}
    We write $b_n:=\lfloor (\log n)^{1+\frac{\varepsilon}{2}}\rfloor$, and denote
\[\widetilde{A}^j_{i_1,i_2}:= \bigg\{ \kp(u^j_1)=\kp(u^j_2),\, \xi(\kp(u^j_1))\le \bigg\lfloor \frac{n}{(\log n)^{1+\frac23 \varepsilon}} \bigg\rfloor \bigg\}.\]

For the expectation, it follows from Lemma \ref{lem:tilde_N} that
\begin{align*}
   \e_n\left[\sum_{j=s^{\downarrow}+1}^s \widetilde N_j\right] =  \sum_{j=s^{\downarrow}+1}^{s} \binom{b_n}{2} \p_n \big( A^j_{12} \big)&\overset{\text{Cor~\ref{coleprob}}}{=} \sum_{j=s^{\downarrow}+1}^{s} \binom{b_n}{2} \frac{q_n(j-1)}{(q_n(j))_2} \e_n\bigg[ (\nu_1^{j-1})_2 \mathbf{1}\bigg\{\nu_1^{j-1}\le \bigg\lfloor \frac{n}{(\log n)^{1+\frac{2}{3}\varepsilon}} \bigg\rfloor \bigg\} \bigg]\\
   &\ \asymp \frac{n}{(\log n)^{1+\frac\varepsilon 8}}(\log n)^{2(1+\frac\varepsilon 2)} \cdot \frac1n = (\log n)^{1+\frac78 \varepsilon},
\end{align*}
where ``$\asymp$'' follows from \eqref{H1} and the fact that there exists constants $c_1,c_2>0$ such that
\begin{align*}
    c_1 \overset{\eqref{eq:liminf}}{<} \e_n\bigg[ (\nu_1^{j-1})_2 \mathbf{1}\bigg\{\nu_1^{j-1}\le \bigg\lfloor \frac{n}{(\log n)^{1+\frac{2}{3}\varepsilon}} \bigg\rfloor \bigg\} \bigg] \le \e_n\big[(\nu^{j-1}_1)_2\big] \overset{\eqref{H1}}{<} c_2.
\end{align*}

Now turn to the variance. Using the independence in Lemma \ref{lem:tilde_N}, we obtain that
\begin{align}\label{eq:Var}
    \text{Var}_n\left[\sum_{j=s^{\downarrow}+1}^{s}\widetilde{N}_j \right]=\sum_{j=s^{\downarrow}+1}^{s}\text{Var}_n\left[\widetilde{N}_j \right]
\le  \sum_{j=s^{\downarrow}+1}^{s}\e_n\left[(\widetilde{N}_j)^2\right].
\end{align}
Denote $d_n:=(\log n)^{1+\frac{2}{3}\varepsilon}$. Then for each $j$,
\begin{align*}
    \e_n\left[(\widetilde{N}_j)^2\right]&= \binom{b_n}{2}\mathbb{P}_n(\widetilde{A}^j_{12})+3\binom{b_n}{3}\mathbb{P}_n(\widetilde{A}^j_{12}\cap \widetilde{A}^j_{13})+3\binom{b_n}{4}\mathbb{P}_n(\widetilde{A}^j_{12}\cap \widetilde{A}^j_{34})\\
    &\lesssim \frac{b_n^2}{n}+\frac{b_n^3}{n^2}\e_n\left[(\nu_1^{j-1})^3\mathbf{1}\big\{\nu_1^{j-1}\le \left\lfloor n d_n^{-1} \right\rfloor \big\}\right]\\
    &\quad\, +\frac{b_n^4}{n^3} \e_n\left[(\nu_1^{j-1})^4\mathbf{1}\big\{\nu_1^{j-1}\le \left\lfloor n d_n^{-1} \right\rfloor \big\}\right]+\frac{b_n^4}{n^2}\e_n\left[(\nu_1^{j-1})^2(\nu_2^{j-1})^2\mathbf{1}\big\{\nu_1^{j-1},\nu_2^{j-1}\le \left\lfloor n d_n^{-1} \right\rfloor \big\}\right]\\
    &\le \frac{b_n^2}{n}+\frac{b^3_n}{n d_n}\e_n\big[(\nu_1^{j-1})^2\big]+\frac{b^4_n}{n d_n^2}\e_n\big[(\nu_1^{j-1})^2\big]+\frac{b^4_n}{nd_n}\e_n\big[(\nu_1^{j-1})^2\cdot \nu_2^{j-1}\big]\\
    &\lesssim \frac{b^4_n}{nd_n}\asymp \frac{(\log n)^{3+\frac{4}{3}\varepsilon}}{n}.
\end{align*}
Substituting back to \eqref{eq:Var} yields the desired result.
\end{proof}
Return to the proof of \eqref{eq:remain}. By \eqref{eq:exp_var} and Markov's inequality, we obtain
\begin{align*}
     \text{LHS of }\eqref{eq:remain}&\le \p_n\bigg(\sum_{j=s^{\downarrow}+1}^{s} \widetilde N_j\le (\log n)^{1+\frac34 \varepsilon} \bigg) 
    \le \p_n\left(\sum_{j=s^{\downarrow}+1}^{s} \widetilde N_j-\e_n\bigg[\sum_{j=s^{\downarrow}+1}^{s} \widetilde N_j\bigg]\le -\frac12 \e_n\bigg[\sum_{j=s^{\downarrow}+1}^{s} \widetilde N_j \bigg]\right)\\
    &\le \frac{4\cdot\text{Var}_n\bigg[\sum_{j=s^{\downarrow}+1}^{s} \widetilde N_j \bigg]}{\left(\e_n\bigg[\sum_{j=s^{\downarrow}+1}^{s} \widetilde N_j \bigg]\right)^2}\lesssim (\log n)^{-\frac{13}{24}\varepsilon}.
\end{align*}
This completes the proof.
\end{proof}
Finally, let us prove Lemmas~\ref{lem:tilde_N}--\ref{lem:X_N}.
\begin{proof}[Proof of Lemma~\ref{lem:tilde_N}]
    Recall the definition of $\mathcal{G}_j$ in \eqref{def:Gj}, i.e., the $\sigma$-field generated by the set of subtrees rooted at some vertex at height $j$ and truncated at height $h^*$. 
Fix $h_*< j\le h^*$.  Conditionally on $\mathcal{G}_j$, we consider, for each
$l>j$, the vertices in the subtrees $\{\mathscr{T}^n(v): |v|=j\}$ that correspond
to the sampled set $\widetilde{\cX}_l$.  
More precisely, define
    \[\widetilde{\cX}'_l(j):=\big\{(\mathscr{T}^n(v),u):\abs{v}=j,\, vu\in \widetilde{\cX}_l\big\},\]
Here each element \((\mathscr{T}^n(v),u)\) retains the vertex \(u\) together with the full subtree \(\mathscr{T}^n(v)\) in which it lies.

    Conditionally on $\mathcal{G}_j$, each $\widetilde{\cX}'_l(j)$ has the following distribution:
    \begin{itemize}
        \item if $X_l\ge\lfloor(\log n)^{1+\frac{\varepsilon}{2}}\rfloor$, $\widetilde{\cX}'_l(j)$ is a uniformly chosen $\lfloor(\log n)^{1+\frac{\varepsilon}{2}}\rfloor$-tuple of $\{(\mathscr{T}^{n}(v),u):\abs{v}=j,\, vu\in \cX_l\}$; note that in this case, all these $\T^{(n)}(v)$'s have at least one descendant at height $h^*$.
        \item Otherwise, $\widetilde{\cX}'_l(j)$ is a uniformly chosen $\lfloor(\log n)^{1+\frac{\varepsilon}{2}}\rfloor$-tuple of $\{(\mathscr{T}^n(v),u):\abs{v}=j,\, \abs{vu}=l\}$.
    \end{itemize}
    Here since the sampled vertex sets are conditionally independent across height, the collections $\big(\widetilde{\cX}'_l(j):j<l\le h^*\big)$ are also independent across $l$ conditionally on $\mathcal{G}_j$.
    Observe that conditionally on $\mathcal{G}_j$, each $\widetilde N_l$, for
$l>j$, is determined by $\widetilde{\cX}'_l(j)$, whereas
$\widetilde N_j$ is determined by $(\widetilde{\cX}_j,\nu^{j-1})$.
Hence by conditional independence, for any test
functions $g_1$ and $g_2$,
    \begin{align}\label{eq:condi_indpt}
        \mathbb{E}_n\left[ g_1\big((\widetilde{N}_l:j<l\le h^*)\big)g_2(\widetilde{N}_j)\,\Big|\,\mathcal{G}_j\right]=\mathbb{E}_n\left[ g_1\big((\widetilde{N}_l:j<l\le h^*)\,\Big|\,\mathcal{G}_j\right]\mathbb{E}_n\left[ g_2(\widetilde{N}_j)\,\Big|\,\mathcal{G}_j\right].
    \end{align}
By Proposition~\ref{prop:uniform}, it is easily seen that
    \[\left(\widetilde{\cX}_j\,\big|\,\mathcal{G}_j\right)\overset{d}{=}\big\{u^j_1,\ldots,u^j_{\lfloor (\log n)^{1+\frac{\varepsilon}{2}}\rfloor}\big\}.\]
    Hence an application of Lemma \ref{lem:basic} with 
\[
X=\widetilde{\cX}_j,\quad Y=\nu^{j-1},\quad Z=\big\{\mathscr{T}^n(v):|v|=j\big\},\quad W=\big\{u^j_1,\ldots,u^j_{\lfloor (\log n)^{1+\frac{\varepsilon}{2}\rfloor}}\big\}
\]
and $f(X,Y)$ the function determined by $\widetilde{\cX}_j$ and $\nu^{j-1}$ corresponding to $g_2(\widetilde{N}_j)$
yields that
\[\mathbb{E}_n\left[ g_2(\widetilde{N}_j)\,\Big|\,\mathcal{G}_j\right]=\mathbb{E}_n\left[g_2(\widetilde{N}'_j)\right],\]
where $\widetilde N'_j$ is defined as in \eqref{def:tilde_N'}.
Substituting this identity into \eqref{eq:condi_indpt} and taking
expectations gives
\begin{align*}
    \mathbb{E}_n\left[ g_1\big((\widetilde{N}_l:j<l\le h^*)\big)g_2(\widetilde{N}_j)\right]=\mathbb{E}_n\left[g_1\big((\widetilde{N}_l:j<l\le h^*)\right]\mathbb{E}_n\left[ g_2(\widetilde{N}'_j)\right].
\end{align*}
Iterating this over $j=h^*,\,h^*\!-\!1,\ldots, h_*$,
we obtain the desired conclusion.
\end{proof}
\begin{proof}[Proof of Lemma~\ref{lem:Xi}]
Set
\[
a_1 := \frac{n}{(\log n)^{1+\frac23 \varepsilon}}, 
\qquad 
a_k := \frac{n}{\exp\big((\log n)^{\frac{k}{24}\varepsilon}\big)},\quad k\ge 2.
\]
For $k\ge 1$, define the event
\[
\Xi_{n,k} := \Big\{\exists\, s\in[\delta n,h(q_n)-\delta n],\, l\le q_n(s)\text{ s.t. }\ 
        a_{k+1}\le \xi(x^s_l)\le a_k,\ 
        \widetilde\xi(x^s_l)\ge (\log n)^{\frac{\varepsilon}{12}}\Big\}.
\]

Since $a_k<1$ for $k>24/\varepsilon$, we have  
\[
\p_n(A_k)=0,\qquad k>24/\varepsilon.
\]
On the other hand, by \eqref{H1} and the exchangeability of $\nu^{s}$, we obtain that uniformly for $n$ sufficiently large and 
\(1\le k\le 24/\varepsilon\),
\begin{align*}
\p_n(\Xi_{n,k}) 
&\lesssim 
n^2\,\E_n\!\left[
   \p_n\!\left(\widetilde\xi(x^s_1)\ge (\log n)^{\frac{\varepsilon}{12}}\mid \nu^s_1\right)
   \mathbf 1\{a_{k+1}\le \nu^s_1\le a_k\}
\right]
\\
&\overset{(*)}{\le}
n^2\,\E_n\!\left[
    \binom{\lfloor(\log n)^{1+\frac\varepsilon2}\rfloor}{\lfloor(\log n)^{\frac{\varepsilon}{12}}\rfloor}
    \frac{(\nu^s_1)_{\big\lfloor(\log n)^{\frac{\varepsilon}{12}}\big\rfloor}
    }{(q_n(s+1))_{\big\lfloor(\log n)^{\frac{\varepsilon}{12}}\big\rfloor}}
    \mathbf 1\{a_{k+1}\le \nu^s_1\le a_k\}
\right]
\\
&\lesssim
n^2\,[\,(\log n)^{1+\frac\varepsilon2}\,]^{(\log n)^{\frac{\varepsilon}{12}}}
\Big(\frac{a_k}{n}\Big)^{(\log n)^{\frac{\varepsilon}{12}}}
\,\p_n(\nu^s_1\ge a_{k+1})
\\
&\overset{(**)}{\lesssim}
n^2\,[\,(\log n)^{1+\frac\varepsilon2}\,]^{(\log n)^{\frac{\varepsilon}{12}}}
\Big(\frac{a_k}{n}\Big)^{(\log n)^{\frac{\varepsilon}{12}}}
(a_{k+1})^{-2}
\\
&\le 
\exp\!\Big(-c\,(\log n)^{\frac{\varepsilon}{12}}\log\log n\Big),
\end{align*}
for some constant $c=c(\varepsilon)>0$.  
Here we deduce \((*)\) by the standard upper bound for hypergeometric tails and \((**)\) by Markov's inequality.
Summing over all \(k\ge 1\) and applying a union bound yields \eqref{eq:Xi}.
\end{proof}
\begin{proof}[Proof of Lemma~\ref{lem:X_N}]
    Note that $\Big\{T^{(r)}_{\lfloor (\log n)^{1+\frac{\varepsilon}{2}} \rfloor}(s) < s^{\downarrow}\Big\} $ implies that $\widetilde{\cX}_j\subset \cX_j$ for all $j=s^{\downarrow},\ldots, s$.
    Thus on the event $(\Xi_n)^{\rm c}\cap \Big\{T^{(r)}_{\lfloor (\log n)^{1+\frac{\varepsilon}{2}} \rfloor}(s) < s^{\downarrow} \Big\}$,
    for any $j=s^{\downarrow}+1,\ldots, s$,
    \begin{align*}
        X^{(r)}_j-X^{(r)}_{j-1}\ge \sum_{x^j_l:\xi(x^j_l)\le \frac{n}{(\log n)^{1+\frac{2}{3}\varepsilon}}}\big(\widetilde{\xi}(x^j_l)-1\big)\overset{(*)}{\ge} \frac{2}{(\log n)^{\frac{\varepsilon}{12}}}\sum_{x^j_l:\xi(x^j_l)\le \frac{n}{(\log n)^{1+\frac{2}{3}\varepsilon}}}\binom{\widetilde{\xi}(x^j_l)}{2}=\frac{2}{(\log n)^{\frac{\varepsilon}{12}}}\widetilde{N}_j,
    \end{align*}
    where $(*)$ follows from the fact that on $(\Xi_n)^{\rm c}$, every $x^j_l$ with $\xi(x^j_l)\le \frac{n}{(\log n)^{1+\frac{2}{3}\varepsilon}}$ must satisfy that $\widetilde{\xi}(x^j_l)\le (\log n)^{\frac{\varepsilon}{12}}$. Summing over $j=s^{\downarrow}+1,\ldots,s$ yields \eqref{eq:tilde N good}.
\end{proof}

\section{Tightness}\label{sec:tight}
Recall that $\widetilde{\mathscr{T}}^n$ is the geometric tree obtained from \( \mathscr{T}^n \) by assigning length \(1/n\) to each edge.
In this section, we prove the tightness of the sequence \((\widetilde{\mathscr{T}}^n)_{n \in \mathbb{N}}\) (Theorem \ref{thm:tight}), which is the main technical step in establishing Theorem~\ref{main}. 

\medskip

Recall the definition of \(\Delta(n,k)\) from \eqref{def:Delta}. Throughout this section, we use this notation specifically for the sequence \((\mathscr{T}^n)\), i.e.,  
\[
\Delta(n,k) := \max_{v \in V(\mathscr{T}^n)} \min_{u \in V(\mathscr{T}^n_k)} d(v,u), \quad \text{for } 1 \leq k \leq \#V(\mathscr{T}^n),
\]
where \(\mathscr{T}^n_k = \mathscr{T}^n(V^n_1, \dots, V^n_k)\) is the subtree of \(\mathscr{T}^n\) spanned by \((V^n_1, \dots, V^n_k)\) and root $\rho$, with \((V^n_1, \dots, V^n_{\#V(\mathscr{T}^n)})\) being a uniform ordering of \(V(\mathscr{T}^n)\).  

\begin{theorem}\label{thm:tight}
Assume \eqref{H1}--\eqref{H2}. Then the sequence \((\widetilde{\mathscr{T}}^n)\) is tight. Equivalently, for any \(\delta > 0\),  
\[
\lim_{k\rightarrow\infty} \limsup_{n\rightarrow\infty} \p_n\left(\Delta(n,k) > \delta n \right) = 0.
\]  
\end{theorem}
\subsection{Strategy of the proof}
We begin by introducing the notion of \textit{$\delta$-coalescent vertices}, which plays a central role in the argument.

\begin{definition}[$\delta$-coalescent vertices]\label{def:coalescent}
    For \(n \ge 1\) and \(\delta > 0\),  define
\begin{equation}\label{HH}  H = H(n,\delta) := \left\lfloor \frac{h(q_n)}{\lfloor\delta n\rfloor} \right\rfloor. 
\end{equation}
The interval $\big[s\lfloor\delta n\rfloor,(s+1)\lfloor\delta n\rfloor\big)$ is called the \textit{$s$-th height interval}. We focus on the coalescent process starting from height \((s+1)\lfloor\delta n\rfloor\) for \(s = 1, \dots, H-2\). Recall that $\cX_{s\lfloor\delta n\rfloor}\big(n, (s+1)\lfloor\delta n\rfloor\big)$
 consists of all vertices at height \(s\lfloor\delta n\rfloor\) that have at least one descendant at height \((s+1)\lfloor\delta n\rfloor\). Those vertices are referred to as the \textit{\(\delta\)-coalescent vertices} at the \( s \)-th height interval. The collection of all \(\delta\)-coalescent vertices at these height intervals is denoted by
    \[
        \bar{\cX}(n,\delta) := \bigcup_{s=1}^{H-2} \cX_{s\lfloor \delta n \rfloor}\big(n,(s+1)\lfloor \delta n \rfloor\big).
    \]
\end{definition}

These \(\delta\)-coalescent vertices serve as structural anchors within the tree, ensuring that every vertex lies within a distance of at most \(3\lfloor\delta n\rfloor\) from one of these vertices. (In general, the distance bound is $\lfloor\delta n\rfloor$, but it increases to $2\lfloor\delta n\rfloor$ near the root since  $s=0$ is not considered, and up to $3\lfloor\delta n\rfloor$ near height $h(q_n)$ as  $s=H-1$ is also excluded.)

To establish the tightness of the sequence \(\widetilde{\mathscr{T}}^n\), we need to show that the maximum distance $\Delta(n,k)$ divided by $n$ between any vertex \(v\) in \(\mathscr{T}^n\) and its closest vertex in the subtree \(\mathscr{T}^n_k\) vanishes as \(k \to \infty\). This is achieved by proving two key results: 
\begin{itemize}
    \item[(1)]  \(\#\bar{\cX}(n,\delta)\) is tight (see Proposition \ref{prop:coalescent});
    \item[(2)] with probability arbitrarily close to 1, each \(\delta\)-coalescent vertex has at least \(\Theta(n^2)\) descendants within distance \(\lfloor\delta n\rfloor\) (see Proposition \ref{prop:size}).
\end{itemize}
These two results imply that, for all $k$ sufficiently large, with high probability, every \(\delta\)-coalescent vertex will have a descendant among the \(k\) selected vertices $\{V^n_1,\ldots,V^n_k\}$ within its \(\lfloor\delta n\rfloor\)-neighborhood. Consequently, \(\frac{1}{n}\Delta(n,k)\) remains small as \(n \to \infty\) and \(k \to \infty\). 

We now formulate the two results as the following propositions.
 
\begin{proposition}\label{prop:coalescent}
Assume \eqref{H1} and \eqref{eq:liminf}. Then for any $\delta>0$, the number of 
$\delta$-coalescent vertices, $\#\bar{\cX}(n,\delta)$, is tight, i.e.,
\[
\lim_{M\rightarrow\infty}\limsup_{n\rightarrow\infty} \p_n\big(\#\bar{\cX}(n,\delta)>M\big)=0.
\]
\end{proposition}
\begin{proof}
    By \eqref{H1}, the quantity \(H = H(n,\delta)\) converges as $n\rightarrow\infty$. Therefore, it suffices to prove that for any $\delta>0$ and \(1 \le s \le H-2\),
\begin{align*}
    \lim_{M\to\infty}\limsup_{n\to\infty} \p_n\Big( X_{s\lfloor\delta n\rfloor}\big(n,(s+1)\lfloor\delta n\rfloor\big) > M \Big) = 0.
\end{align*}
This follows directly from Theorem~\ref{thm:CDFI}, with $h^*=(s+1)\lfloor\delta n\rfloor$ and $h_*=s\lfloor\delta n\rfloor$.
\end{proof}

Recall that for \(v \in \mathscr{T}^n\), we write \(\text{Desc}(v)\) for the set of its descendants.  
For any \(\gamma > 0\), define \(B(n,\gamma,\delta)\) as the event that every \(\delta\)-coalescent vertex has at least \(\gamma n^2\) descendants within graph distance \(\delta n\) in \(\mathscr{T}^n\). That is,
\begin{align}\label{def:B}
    B(n,\gamma,\delta) := \Big\{\forall\,v \in \bar{\cX}(n,\delta):\ 
    \#\{u \in \text{Desc}(v) : d(u,v) < \delta n\} \ge \gamma n^2 \Big\}.
\end{align}

\begin{proposition}\label{prop:size}
Assume \eqref{H1}--\eqref{H2}. Then
\[
    \lim_{\gamma \to 0+}\,\liminf_{n \to \infty} 
    \p_n\big(B(n,\gamma,\delta)\big) = 1.
\]
\end{proposition}

We defer the proof of Proposition~\ref{prop:size}.

\subsection{Proof of Theorem \ref{thm:tight} assuming Proposition \ref{prop:size}} 
Define the event  
\[
A(n,k,\delta):=\big\{\forall\,v\in\bar{\cX}(n,\delta),\ \exists\,1\leq i\leq k\text{ such that }V^n_{i}\in \text{Desc}(v)\text{ and }d(v,V^n_{i})\leq \delta n\big\}.
\] 

\noindent By \eqref{eq:total_popuplation}, there exists a constant \(C>0\) such that  
\(\#V(\mathscr{T}^n)\leq Cn^2.\)
It follows that for any \(M\geq 1\) and \(\gamma>0\),  
\[
\p_n\Big(A(n,k,\delta)\,\big|\,\{\#\bar{\cX}(n,\delta)\leq M\}\cap B(n,\gamma,\delta)\Big)\geq 1- M \big(1-\gamma/C\big)^k.
\]
Combining Propositions \ref{prop:coalescent} and \ref{prop:size}, we conclude that  
\[
\lim_{k\rightarrow\infty}\liminf_{n\rightarrow\infty}\p_n\big(A(n,k,\delta)\big) = 1.
\]

The proof is then completed by the following key observation:  
\begin{align}\label{eq:A}
	A(n,k,\delta)\subset \big\{\Delta(n,k)\leq 4\delta n\big\}.
\end{align}
Indeed, since the root \(\rho\) is a vertex in \(\mathscr{T}^n_k\), for any vertex \(v\) with \(\abs{v}\in\big[0, 4\delta n\big]\), we have  
\[
\min_{u \in V(\mathscr{T}^n_k)} d(v,u)\leq d(v,\rho) \leq 4\delta n.
\]
For all vertex \(v\) with \(\abs{v}\in\big(s\delta n, (s+1)\delta n\big]\) for some \(s=4,\dots,H\), let \(v^*\) be its ancestor at height \((s-2)\lfloor \delta n\rfloor\). By definition,  \(v^*\in \bar{\cX}(n,\delta)\). Moreover, on the event \(A(n,k,\delta)\), there exists \(1\leq i_0\leq k\) such that \(V^n_{i_0}\in \text{Desc}(v^*)\) and $d(v^*, V^n_{i_0})<\delta n$. Consequently, 
\[
\min_{u \in V(\mathscr{T}^n_k)} d(v,u)\leq d(v,V^n_{i_0})\leq d(v,v^*)+d(v^*,V^n_{i_0})\leq 4\delta n, \quad\text{on }A(n,k,\delta).
\]
This verifies \eqref{eq:A} and thus concludes the proof.


\subsection{Proof of Proposition \ref{prop:size}}\label{sec:pf_size}
Recall the definition of $B(n, \gamma, \delta)$ from \eqref{def:B}. For $s = 1, \ldots, H-2$, define the event
\begin{align}\label{eq:B_nk0}
    B(n, s, \gamma, \delta) := \Big\{\forall\,v\in X_{s\lfloor \delta n \rfloor}\big(n,(s+1)\lfloor \delta n \rfloor\big),\ 
    \#\{u\in \text{Desc}(v):d(u,v)<\delta n\}\ge \gamma n^2\Big\}.
\end{align}  
This represents the ``restriction'' of $B(n, \gamma, \delta)$ to the $s$-th height interval. It suffices to fix $\delta>0$ and $1\le s\le H-2$, and show that for any $\eta>0$, there exists $\gamma=\gamma(\eta)>0$ such that for all $n$ sufficiently large,  
\begin{align}\label{eq:B_nk}
     \p_n\big(B(n,s,\gamma,\delta)\big)\ge 1-\eta.
\end{align}

Our strategy is as follows. Recall the base height $h^*$ in the definitions of \(\cX_j=\cX_j(n,h^*)\) and \(X_j=X_j(n,h^*)\). In this subsection, we fix
\[
h^* := (s+1)\lfloor \delta n\rfloor
\]
and consider the coalescent process \((X_j: h^*-\lfloor \delta n\rfloor\le j\le h^*)\) from this base height. 
We first analyze the behavior of the coalescent process $(X_{j})$ and identify a typical event that occurs with high probability and is more tractable for subsequent estimates. 
We then condition on this event and establish the desired bound for the probability of $B(n,s,\gamma,\delta)$.

\subsubsection*{A typical event}

As preparation, we identify a typical event in the coalescent process, which will serve as the basis for our later estimates.  
By Theorem~\ref{thm:CDFI},  
\[
\lim_{M\to\infty} \liminf_{n\to\infty} \p_n\Big( X_{h^*-\lfloor \frac12 \delta n \rfloor} \leq M \Big) = 1.
\]  
Hence, for any $\eta>0$, there exists $M_0 = M_0(\eta) \ge 1$ such that  
\begin{align}\label{eq:pre1}
	\p_n\!\Big( X_{h^*-\lfloor \frac12 \delta n \rfloor} \leq M_0 \Big) \ge 1 - \tfrac\eta4, 
	\quad \text{for all } n \ge 1.
\end{align}
Next, consider
\[
T = T(n,\delta) := \sup\Big\{0\le j < h^*-\lfloor\tfrac{1}{2}\delta n\rfloor : X_{j} < X_{h^*-\lfloor \frac12 \delta n \rfloor} \Big\}.
\] 
Using the same argument as in \eqref{zetasto} together with \eqref{eq:upper}, we obtain:

\begin{proposition}
    There exists a constant $C=C(\delta,\varepsilon,M_0)>0$ such that for any $1\le m\le M_0$, conditionally on $X_{h^*-\lfloor \frac12 \delta n \rfloor}=m$,
    \[
    h^*- \lfloor \tfrac12 \delta n \rfloor-T \;\ge_{\rm sto}\; {\rm Geometric}\bigg(\frac Cn\bigg)\wedge \frac14 \delta n.
    \]
    Consequently, for any $\eta>0$, there exists $c_0=c_0(C,\delta,\eta)\in(0,\delta/2)$ such that for all $1\le m\le M_0$,  
    \begin{align}\label{eq:pre2}
        \p_n\!\Big(T \le  h^* - \left\lfloor \tfrac12 \delta n \right\rfloor - \lfloor c_0 n \rfloor \,\Big|\, X_{h^*-\lfloor \frac12 \delta n \rfloor}=m\Big) 
        \ge 1-\tfrac\eta4, 
        \quad \forall\, n\ge 1.
    \end{align}
\end{proposition}

Set
\[
\bar{h} := h^* - \left\lfloor \tfrac12 \delta n \right\rfloor - \lfloor c_0 n \rfloor.
\]
Combining \eqref{eq:pre1} and \eqref{eq:pre2}, we arrive at the following high-probability event:
\begin{align}\label{eq:pre} 
\begin{aligned}
&E(n,\delta,m):= \Big\{ T\le \bar{h},\, X_{h^*-\lfloor \frac12 \delta n \rfloor} =m \Big\}=\Big\{X_k=m,\ \forall \,k\in [\, \bar{h}, \bar{h}+\lfloor c_0 n\rfloor]\Big\},\\
&\p_n\bigg(\bigcup_{m=1}^{M_0} E(n,\delta,m)\bigg) \ge 1-\tfrac\eta 2.   
\end{aligned}
\end{align}
We shall take the event $E(n,\delta,m)$ as the typical event on which our subsequent analysis will be conditioned.
Define for $1\le m\le M_0$ the conditional probability measure
\[
\q^m_n (\cdot):= \p_n\big(\,\cdot\, \big|\, E(n,\delta,m)\big).
\]

\subsubsection*{Structure of $\mathscr{T}^n$ on $E(n,\delta,m)$}

The purpose of this subsection is to set up the structural ingredients --- coalescent subtrees, trunk paths and vertices, which will be crucial for the forthcoming estimates.
 Recall the subtree notation \(\mathscr{T}^n(v)=\mathscr{T}^n(v,h^*)\) from \eqref{def:subtree}, which denotes the subtree of \(\mathscr{T}^n\) rooted at \(v\) and truncated at height \(h^*\).

\begin{definition}[Coalescent subtrees]
For \(0\le j\le h^*\), define the collection of the coalescent subtrees at height \(j\) by
\[
\mathbf{T}_j=\mathbf{T}_j(n,h^*):=\big\{\mathscr{T}^n(v): v\in \cX_j\big\}.
\]
Each element of \(\mathbf{T}_j\) is a subtree rooted at a vertex at height \(j\) that admits at least one descendant at height \(h^*\). Note that \(|\mathbf{T}_j|=|\cX_j|=X_j\).

We then write \(\mathsf{T}_1,\ldots,\mathsf{T}_{|\mathbf{T}_j|}\) for a uniform random ordering of the elements of \(\mathbf{T}_j\). Concretely, conditionally on \(\mathscr{T}^n\), we first list the elements of \(\mathbf{T}_j\) according to the lexicographical order of their roots, denoted by \(\mathbf{T}_j^{(1)},\ldots,\mathbf{T}_j^{(|\mathbf{T}_j|)}\). Then we choose a permutation \(\pi\) uniformly at random from the symmetric group on \(|\mathbf{T}_j|\) elements, and define
\[
(\mathsf{T}_1,\ldots,\mathsf{T}_{|\mathbf{T}_j|})
:= (\mathbf{T}_j^{(\pi(1))},\ldots,\mathbf{T}_j^{(\pi(|\mathbf{T}_j|))}).
\]
Finally, we introduce the labeling map
\begin{align}\label{def:labeling}
    \Lambda_j : \mathbf{T}_j \longrightarrow \{1,\ldots,|\mathbf{T}_j|\}, \qquad
\Lambda_j(\mathsf{T}_i) := i,
\end{align}
which assigns to each coalescent subtree its index in the random ordering.  
The random orderings (and hence the labeling maps \(\Lambda_j\)) are taken independently across different heights~\(j\).
\end{definition}

\begin{definition}[Trunk paths and vertices]
    Under $\q^m_n$, we have 
\[
X_j = m, \quad \text{for all } \bar{h} \le j \le \bar{h} + \lfloor c_0 n \rfloor.
\]
In other words, for each $j$ between $\bar{h}$ and $\bar{h}+\lfloor c_0 n\rfloor$, there are exactly $m$ vertices in $\cX_j$. These vertices belong to $m$ distinct ancestral lineages that remain separate at height $\bar{h}$. 
Let $\mathsf{T}_1,\ldots,\mathsf{T}_m$ be the $m$ coalescent subtrees at height $\bar{h}$.  
For $\bar{h}\le j \le \bar{h}+\lfloor c_0 n \rfloor$ and $1\le i\le m$, denote by 
\(z^j_i \in V(\mathsf{T}_i)\)
the unique vertex of $\mathsf{T}_i$ at height $j$. 
The sequence
\[
Z_i := \big(z^{\bar{h}+\lfloor c_0 n \rfloor}_i, z^{\bar{h}+\lfloor c_0 n \rfloor-1}_i, \ldots, z^{\bar{h}}_i\big)
\]
is called the trunk path of $\mathsf{T}_i$, and its elements are referred to as the trunk vertices of $\mathsf{T}_i$.  
The remaining $q_n(j)-m$ non-trunk vertices at height $j$ are denoted by $y^j_1, \ldots, y^j_{q_n(j)-m}$, where the subscripts are assigned arbitrary.

\end{definition}

\subsubsection*{Merging vertices and key theorem}
While the $m$ trunk paths remain disjoint, the lineages of the $y^j_l$ may merge into some $Z_i$, implying $y^j_l \in {\rm Desc}(z^{\bar{h}}_i)$. To quantify this, define
\[
\mathcal{N}_i := \sum_{j=\bar{h}+1}^{\bar{h}+\lfloor c_0 n \rfloor} \sum_{\ell=1}^{q_n(j)-m} \mathbf{1}\big\{ y^j_l \in {\rm Desc}(z^{\bar{h}}_i) \big\},
\]
the total number of vertices whose lineages merge into $Z_i$. The following theorem provides a lower bound on $\mathcal{N}_i$.

\begin{theorem}\label{thm:Uj}
Assume \eqref{H1}--\eqref{H2}. For fixed $\delta, s, \eta, M_0, c_0$, there exists $\gamma>0$ such that for all $1\le m\le M_0$ and sufficiently large $n$,
\begin{align}\label{eq:Nj}
    \q^m_n \big( \mathcal{N}_i \ge \gamma n^2,\ \forall 1\le i \le m \big) \ge 1 - \tfrac\eta 2.
\end{align}
\end{theorem}

\begin{proof}[Proof of \eqref{eq:B_nk} assuming Theorem \ref{thm:Uj}]
We observe that each $\delta$-coalescent vertex at height $s\lfloor\delta n\rfloor$ has at least one descendant in $\{z^0_1, \dots, z^0_m\}$. Hence, for all large $n$ and $1\le m\le M_0$,
\[
\q^m_n \big( B(n,s,\gamma,\delta) \big) \ge \q^m_n \big( \mathcal{N}_i \ge \gamma n^2,\ \forall 1\le i \le m \big) \ge 1 - \tfrac\eta 2.
\]
Combining with \eqref{eq:pre} completes the proof of \eqref{eq:B_nk}.
\end{proof}

\subsection{Proof of Theorem \ref{thm:Uj}}\label{sec:pf_size2}
The remainder of this section is devoted to the proof of Theorem \ref{thm:Uj}. 
Throughout, we fix $\delta, s, M_0, c_0$, and regard them as constants (i.e. omit the dependence of other constants on $\delta, s, M_0, c_0$).

\subsubsection{Outline of the proof}
We begin with a brief outline of the proof strategy. 

The first step is to partition the interval $\big[\bar{h},\bar{h}+\lfloor c_0 n\rfloor\big]$ into $K_\theta$ sub-intervals $(h_r,h_{r-1}]$, where $\bar{h}+\lfloor c_0 n\rfloor=h_0>h_1>\cdots>h_{K_\theta}$ and each sub-interval has length $\lfloor \theta n\rfloor$ (see \eqref{def:Ktheta} below for precise definitions). For each $r=0,\ldots,K_\theta -1$, let $\mathcal{N}_{i,r}$ be the number of $y^j_l$ in the height interval $(h_r,h_{r-1}]$ that merge into $Z_i$ before height $h_r$. 

We then establish the independence of the family $(\mathcal{N}_{i,r})$ under $\mathbb{Q}^m_n$ (Proposition~\ref{lem:Ur independence}) and obtain a uniform lower bound for their tail probabilities (Proposition~\ref{prop:Ur}). Since the event that $\mathcal{N}_i>\gamma n^2$ is implied by the event that some $\mathcal{N}_{i,r}>\gamma n^2$, the independence and uniform lower bound together yield the desired probability estimate in \eqref{eq:Nj}.

To prove the independence and the uniform lower bound, we introduce the notion of \emph{uniform selection pairs} (Definition~\ref{def:uniform}).  
Roughly speaking, a random vector $\mathcal{V}$ consisting of $k$ vertices of $\mathscr{T}^n$ at some height $h$, together with an event $A$ depending only on the genealogy above that height, forms a uniform selection pair if, conditionally on $A$, $\mathcal{V}$ is distributed as a uniformly chosen $k$-tuple of distinct vertices at height $h$.  
Such structures naturally appear in the events we consider.

For these events, by using the general lemma (Lemma \ref{lem:selection_pair}) for uniform selection pairs, we show that their probabilities can be decomposed into contributions from the genealogies above and below a given height.  
This decomposition then yields the desired independence property and simplifies the probability estimates required for proving the uniform lower bound.

Finally, the uniform lower bound is obtained via the second moment method, by relating the merging probabilities to the coalescent probabilities in Corollary \ref{distinctprob} (see Proposition~\ref{p:tau_j}).

\subsubsection{Key propositions}

Let $\theta\in (0,c_0)$. Define
\begin{align}\label{def:Ktheta}
    K_\theta := \bigg\lfloor \frac{\lfloor c_0 n\rfloor}{\lfloor \theta n \rfloor} \bigg\rfloor, 
\qquad 
h_r := \bar{h}+\lfloor c_0 n\rfloor - r \lfloor \theta n \rfloor, \quad 0 \le r \le K_\theta,
\end{align}
and for $1 \le i \le m$ and $1 \le r \le K_\theta$, set
\[
\mathcal{N}_{i,r} := \sum_{j=h_r+1}^{h_{r-1}} 
\sum_{l=1}^{q_n(j)-m} 
\mathbf{1}\big\{ y^j_l \in {\rm Desc}(z_i^{h_r}) \big\}.
\]
Clearly,
\(\mathcal{N}_i \;\ge\; \sum_{r=1}^{K_\theta} \mathcal{N}_{i,r}\,,\)
with equality when $\lfloor c_0 n \rfloor/\lfloor \theta n \rfloor \in \N$.

For later use, we introduce the quantity $\widetilde{\mathcal{N}}_{i,r}$, which serve as an analog of $\mathcal{N}_{i,r}$ for uniformly chosen vertices 
$(u^{h_{r-1}}_1,\ldots,u^{h_{r-1}}_m)$ and will naturally appear in subsequent arguments.

Recall the notion of $\text{Anc}(v,j)$ appearing at the beginning of Section~\ref{sec:coalescence}, which represents the ancestor of $v$ at height $j$.
For $1 \le r \le K_\theta$ and $1 \le i \le m$, denote the ancestral lineage of $u^{h_{r-1}}_i$ up to height $h_r$ by
\[
U_{i,r}:= \Big(u^{h_{r-1}}_i, \text{Anc}(u^{h_{r-1}}_i,h_{r-1}-1), \ldots, \text{Anc}(u^{h_{r-1}}_i,h_{r})\Big).
\]
For each $j$ with $h_r < j \le h_{r-1}$, denote the remaining $q_n(j)-m$ vertices at height $j$ by $v^j_1, \ldots, v^j_{q_n(j)-m}$, where the subscripts are assigned, conditionally on $\mathscr{T}^n$, uniformly at random and independently across heights. \footnote{Unlike $(y^j_l)$ before, we specify the distribution of subscripts to simplify later analysis.}
Then we set, for $1 \le r \le K_\theta$ and $1 \le i \le m$,
\[
\widetilde{\mathcal{N}}_{i,r} := 
\sum_{j=h_r+1}^{h_{r-1}} \sum_{l=1}^{q_n(j)-m} 
\mathbf{1}\Big\{v^j_l \in \text{Desc}\big(\text{Anc}(u^{h_{r-1}}_i,h_{r})\big)\Big\}.
\]
That is, $\widetilde{\mathcal{N}}_{i,r}$ counts the number of vertices $v^j_l$ at heights $(h_r, h_{r-1}]$ that are not part of any ancestral lineage $(U_{k,r}: k=1,\ldots,m)$ but eventually merge into $U_{i,r}$ at some height greater than or equal to $h_r$.

For $1 \le r \le K_\theta$, we define the event
\[
\widetilde{B}_r := \big\{U_{1,r}, \ldots, U_{m,r} \text{ are disjoint}\big\},
\]

\begin{proposition}\label{lem:Ur independence}
The random variables $(\mathcal{N}_{i,r}:1\le r\le K_\theta)$ are independent under $\q_n^m$. Moreover, for any $1\le r\le K_\theta$ and test function $f$ on $\mathbb{N}$,
\begin{align}\label{eq:condi_law}
    \mathbb{Q}^m_n\big(f(\mathcal{N}_{i,r})\big)=\mathbb{P}_n\left(f(\widetilde{\mathcal{N}}_{i,r})\,|\,\widetilde{B}_{r}\right)
\end{align}
\end{proposition}

\begin{proposition}\label{prop:Ur}
Assume \eqref{H1}--\eqref{H2}. Then there exists constant $C_0 \in (0,1)$ such that for any $\theta\in (0,c_0)$ sufficiently small, there exists $\gamma=\gamma(\theta, C_0)>0$ satisfying
\[
\q_n^m (\mathcal{N}_{i,r} \ge \gamma n^2) \;\ge\; C_0,
\]
for all sufficiently large $n$, $1\le i\le m \le M_0$ and $0\le r\le K_\theta -1$.
\end{proposition}

The proofs of Propositions \ref{lem:Ur independence} and \ref{prop:Ur} are postponed. We first explain how these results imply Theorem \ref{thm:Uj}.

\begin{proof}[Proof of Theorem~\ref{thm:Uj} assuming Propositions \ref{lem:Ur independence} and \ref{prop:Ur}]
Observe that if $\mathcal{N}_{i,r}\ge \gamma n^2$ for some $1\le r\le K_\theta$, then $\mathcal{N}_i\ge \gamma n^2$. Hence, by Propositions \ref{lem:Ur independence} and \ref{prop:Ur}, for any sufficiently small $\theta\in(0,c_0)$, there exists $\gamma>0$ such that for sufficiently large $n$, $1\le i\le m \le M_0$ and $0\le r\le K_\theta -1$,
\begin{align}\label{eq:Ni}
    \q_n^m(\mathcal{N}_i\ge \gamma n^2)\;\ge\; 1- \q_n^m\!\big(\mathcal{N}_{i,r}< \gamma n^2,\ \forall\, 1\le r\le K_\theta\big) 
\;\ge\; 1-(1-C_0)^{K_\theta}.
\end{align}

On the other hand, since $K_\theta \sim c_0/\theta$ as $n\to\infty$, for any $\eta>0$ one can choose $\theta=\theta(\eta)>0$ small enough such that
\begin{align}\label{eq:small_theta}
    (1-C_0)^{K_\theta} \le  \frac{\eta}{2M_0},\qquad\text{for all }n\ge 1.
\end{align}

Combining \eqref{eq:Ni} and \eqref{eq:small_theta}, we obtain that for any $\eta>0$, there exists $\gamma>0$ such that for all sufficiently large $n$,
\[
\q_n^m\big(\mathcal{N}_i\ge \gamma n^2,\, \forall\, 1\le i\le m\big)
    \;\ge\; 1- \sum_{j=1}^m \q_n^m\big(\mathcal{N}_i< \gamma n^2\big) 
    \;\ge\; 1 - \frac\eta 2,
\]
which completes the proof.
\end{proof}

\subsubsection{Proof of Propositions \ref{lem:Ur independence} and \ref{prop:Ur}}
We begin with a preliminary lemma, which will be used repeatedly in the sequel.
Recall the definitions of $\mathcal{G}^{(k)}_j$ in \eqref{def:G} and labeling map $\Lambda_j$ in \eqref{def:labeling}.

\begin{definition}[Uniform selection pair]\label{def:uniform}
Fix a base height $h^*$. For $0\le h\le h^*$, $1\le p\le q_n(h^*)$ and $1\le k\le q_n(h)$, let $(A,\mathcal{V})$ be a pair consisting of
\begin{itemize}
    \item an event $A$ measurable with respect to $\mathcal{G}_{h}^{(p)}\vee \sigma(\Lambda_h)$, and
    \item a random vector $\mathcal{V}$ of $k$ vertices taking values in $\{x^{h}_1,\ldots, x^{h}_{q_n(h)}\}$.
\end{itemize}
We say that $(A,\mathcal{V})$ is a $(p,h,k)$-uniform selection pair if, under $\mathbb{P}_n$,
\[
(\mathcal{V}\mid A)\,\overset{d}{=}\,\{u^{h}_1,\ldots,u^{h}_k\},
\]
that is, conditionally on $A$, the vector $\mathcal{V}$ has the same distribution as a uniformly chosen $k$-tuple of distinct vertices at height $h$. When \(p=q_n(h^*)\), we abbreviate this as an \((h,k)\)-uniform selection pair.
\end{definition}

Recall the notion of offspring vector $\nu^s$.

\begin{lemma}\label{lem:selection_pair}
Let $0\le h\le h^*$, $1\le p\le q_n(h^*)$ and $1\le k\le q_n(h)$, and let $(A, \mathcal{V})$ be a $(p,h,k)$-uniform selection pair. Then for any test function $f$ defined on the value space of $\big((u^{h}_i)_{i=1,\ldots,k}, (\nu^s)_{1\le s\le h-1}\big)$, we have
\[
\mathbb{P}_n\Big[f\big(\mathcal{V}, (\nu^s)_{1\le s\le h-1}\big) \,\big|\, A\Big]
= \mathbb{P}_n\Big[f\big((u^{h}_i)_{1\le i\le k}, (\nu^s)_{1\le s\le h-1}\big)\Big].
\]
\end{lemma}

The proof follows the same reasoning as that in the proof of Theorem~\ref{thm:Markov}, and is therefore omitted.

\begin{proof}[Proof of Proposition \ref{lem:Ur independence}]
Recall the definition of $E(n,\delta,m)$ in \eqref{eq:pre}. For each $0 \le r \le K_\theta$, we split the event $E(n,\delta,m)$ according to the behavior of the process above and below $h_r$:
\[
    E(n,\delta,m) = A_r \cap B_r,
\]
where
\[
    A_r := \big\{X_{\bar{h}+\lfloor c_0 n\rfloor} = X_{h_r} = m\big\}, 
    \qquad 
    B_r := \big\{X_{\bar{h}} = X_{h_r-1}=m\big\}.
\]
Then for any $1 \le r,r' \le K_\theta$ and subsets $a_1,\ldots,D_r$ of $\mathbb{N}$, we have
\begin{align}\label{eq:condi-new}
    \mathbb{Q}^m_n\Bigg[\bigcap_{l=1}^{r'}\{\mathcal{N}_{i,l}\in D_l\}\Bigg]
    =\frac{\mathbb{P}_n\Big[A_r\cap B_r\cap \bigcap_{l=1}^{r'}\{\mathcal{N}_{i,l}\in D_l\}\Big]}{\mathbb{P}_n(E(n,\delta,m))}.
\end{align}

\begin{lemma}\label{lem:AX}
For any $0\le r\le K_\theta$, both $\big(A_{r}, \cX_{h_r}\big)$ and $\big(A_{r}\cap \bigcap_{l=1}^{r}\{\mathcal{N}_{i,l}\in D_l\}, \cX_{h_r}\big)$ are $(h_r,m)$-uniform selection pairs.
\end{lemma}
\begin{proof}
    Clearly, both $A_{r}$ and $\bigcap_{l=1}^{r}\{\mathcal{N}_{i,l}\in D_l\}$ are measurable with respect to $\mathcal{G}_{h_r}\vee \sigma(\Lambda_{h_r})$. Furthermore, as in the proof of Proposition~\ref{prop:uniform}, the uniformity condition follows directly from the exchangeability of $\big(\mathscr{T}^n(v): |v|=h_r\big)$. This concludes the lemma.
\end{proof}

Returning to \eqref{eq:condi-new}, the next step is to apply Lemma~\ref{lem:selection_pair} to the numerator on the RHS. Recall the definitions of $\widetilde{B}_r$ and $\widetilde{\mathcal{N}}_{i,r}$ introduced before Proposition~\ref{lem:Ur independence}.
By Lemmas~\ref{lem:selection_pair} and \ref{lem:AX}, for $1 \le r \le K_\theta$,
\begin{align*}
    \mathbb{P}_n\Big[A_r\cap B_r\cap \bigcap_{l=1}^{r+1}\{\mathcal{N}_{i,l}\in D_l\}\Big]
    &=\mathbb{P}_n\Big[A_{r}\cap \bigcap_{l=1}^{r}\{\mathcal{N}_{i,l}\in D_l\}\Big]\,
      \mathbb{P}_n\big(\widetilde{B}_r\cap \{\widetilde{\mathcal{N}}_{i,r+1}\in D_r\}\big),\\
    \mathbb{P}_n\Big[A_r\cap B_r\cap \bigcap_{l=1}^{r}\{\mathcal{N}_{i,l}\in D_l\}\Big]
    &=\mathbb{P}_n\Big[A_{r}\cap \bigcap_{l=1}^{r}\{\mathcal{N}_{i,l}\in D_l\}\Big]\,
      \mathbb{P}_n(\widetilde{B}_r).
\end{align*}
Consequently,
\[
    \mathbb{P}_n\Big[A_r\cap B_r\cap \bigcap_{l=1}^{r+1}\{\mathcal{N}_{i,l}\in D_l\}\Big]
    = \mathbb{P}_n\Big[A_r\cap B_r \cap \bigcap_{l=1}^{r}\{\mathcal{N}_{i,l}\in D_l\}\Big]\,
      \mathbb{P}_n\big(\widetilde{\mathcal{N}}_{i,r+1}\in D_r\,\big|\,\widetilde{B}_r\big).
\]
Substituting into \eqref{eq:condi-new}, we obtain
\[
    \mathbb{Q}^m_n\Bigg[\bigcap_{l=1}^{r}\{\mathcal{N}_{i,l}\in D_l\}\Bigg]
    = \mathbb{Q}^m_n\Bigg[\bigcap_{l=1}^{r-1}\{\mathcal{N}_{i,l}\in D_l\}\Bigg]\,
      \mathbb{P}_n\big(\widetilde{\mathcal{N}}_{i,r}\in D_r\,\big|\,\widetilde{B}_r\big).
\]
Iterating this identity over $1 \le r \le K_\theta$ gives
\[
    \mathbb{Q}^m_n\Bigg[\bigcap_{l=1}^{K_\theta}\{\mathcal{N}_{i,l}\in D_l\}\Bigg]
    = \prod_{r=1}^{K_\theta}\mathbb{P}_n\big(\widetilde{\mathcal{N}}_{i,r}\in D_r\,\big|\,\widetilde{B}_r\big),
\]
which proves the desired results.
\end{proof}

Next, we turn to the proof of Proposition~\ref{prop:Ur}. 
By \eqref{eq:condi_law}, it suffices to show that
\[
\mathbb{P}_n\big(\widetilde{\mathcal{N}}_{i,r}\ge \gamma n^2\,\big|\,\widetilde{B}_{r}\big)\ge C_0.
\]
This will be established using the second moment method. 
To this end, 
we shall estimate, for each $v^j_l$, the probability
\[
    \mathbb{P}_n\Big[v^j_l \in \text{Desc}\big(\text{Anc}(u^{h_{r-1}}_i,h_{r})\big)\,\Big|\,\widetilde{B}_{r}\Big],
\]
and for distinct $v^j_l, v^{j'}_{l'}$,
\[
    \mathbb{P}_n\Big[v^j_l, v^{j'}_{l'} \in \text{Desc}\big(\text{Anc}(u^{h_{r-1}}_i,h_{r})\big)\,\Big|\,\widetilde{B}_{r}\Big].
\]

We begin with some notation.  
For any $1\le i\le m\le M_0$, $1\le r\le K_\theta$, and vertex $v^j_l$ with $h_r< j\le h_{r-1}$ and $1\le l\le q_n(j)-m$, define
\[
    \tau_i(v^j_l):=\inf\Big\{1\le k\le j-h_r : \text{Anc}(v^j_l,j-k)\in U_{i,r}\Big\},
\]
with the convention $\inf\emptyset=\infty$.  
In other words, $\tau_i(v^j_l)$ is the ``first time'' $k$ at which the ancestral lineage of $v^j_l$ merges into the trunk path $U_{i,r}$. Clearly,
\[
    \big\{v^j_l \in \text{Desc}\big(\text{Anc}(u^{h_{r-1}}_i,h_{r})\big)\big\}
    = \big\{\tau_i(v^j_l)<\infty\big\}.
\]

\begin{proposition}\label{p:tau_j}
Assume \eqref{H1}--\eqref{H2}. Then:
\begin{enumerate}[label=(\arabic*), ref=\arabic*]
    \item\label{item:tau1} Uniformly for all $1\le i\le m\le M_0$, $1\le r\le K_\theta$, and vertices $v^j_l$ with $h_r< j\le h_{r-1}$ and $1\le l\le q_n(j)-m$, 
    \begin{align}
        \p_n\big(\tau_i(v^j_l)<\infty\,\big|\,\widetilde{B}_{r}\big)
        =\frac{1}{m}\left\{1- \prod_{h=0}^{j-h_r-1}\Bigg[1-\frac{m\sigma^2_n(j-h)}{q_n(j-h)}\big(1+o_n(1)\big)\Bigg]\right\}.
        \label{eq:no merge}  
    \end{align}
    \item\label{item:tau2} Uniformly for all $1\le i\le m\le M_0$ and vertices $v^j_l,\,v^{j'}_{l'}$ with $h_r< j\le j' \le h_{r-1}$, $1\le l\le q_n(j)-m$, and $1\le l'\le q_n(j')-m$,  
    \begin{equation}\label{eq:no merge2} 
    \begin{aligned} 
        &\p_n\big( \tau_i(v^j_l) < \infty \,\big|\, \tau_i(v^{j'}_{l'})<\infty,\,\widetilde{B}_{r} \big) \\
        \le \,& \frac{1}{m}\left\{1-\big(1+o_n(1)\big)\prod_{h=0}^{j-h_r-1}\Bigg[1-\frac{(m+1)\sigma^2_n(j-h)}{q_n(j-h)}\big(1+o_n(1)\big)\Bigg]\right\}. 
    \end{aligned} 
    \end{equation}
\end{enumerate}
\end{proposition}

\begin{proof}[Proof of Proposition \ref{prop:Ur} assuming Proposition \ref{p:tau_j}]
Throughout the proof, we use $c_1, c_2, \dots$ to denote positive constants, independent of $\theta$ (though they may depend on other fixed parameters $\delta, s, M_0, c_0$).

We estimate the first and second moments of $\widetilde{\mathcal{N}}_{i,r}$. For the first moment, 
\begin{align}\label{eq:first_moment}
    \mathbb{E}_n \big(\widetilde{\mathcal{N}}_{i,r}\,\big|\,\widetilde{B}_{r}\big) 
    &= \sum_{j=h_r+1}^{h_{r-1}} \sum_{l =1}^{q_n(j)-m} \p_n\big(\tau_i(v^j_l)<\infty\,\big|\,\widetilde{B}_{r}\big).
\end{align}
We notice that $j-h_r\le \theta n$ for $h_r< j\le h_{r-1}$. By \eqref{eq:no merge}, \eqref{H1} and \eqref{H2},  we deduce that for all $n$ sufficiently large,
\begin{align*}
	\p_n\big(\tau_i(v^j_l)<\infty\,\big|\,\widetilde{B}_{r}\big)
	&\ge \frac{1}{m}\left\{1- \bigg[1-\frac{c_1m}{n}\bigg]^{j-h_r}\right\}
	\ge \frac{c_1(j-h_r)}{2n},
\end{align*}
provided that $\theta\le \frac{1}{2c_1 M_0}\log 2$ (in which case $\frac{c_1m}{n}\le  \frac{\log 2}{2\theta n}\le 1-(\frac12)^{\frac{1}{\theta n}}$ for all $n$ sufficiently large); the last equality then follows since $(1-x)^k\le 1-\frac k2 x$ for all $0\le x\le 1-(\frac12)^{\frac{1}{\theta n}}$ and $1\le k\le \theta n$.
Substituting this into \eqref{eq:first_moment} gives
\begin{align}\label{eq:first}
	\mathbb{E}_n \big(\widetilde{\mathcal{N}}_{i,r}\,\big|\,\widetilde{B}_{r}\big) 
	\ge \sum_{j=h_r+1}^{h_{r-1}} c_2 n\cdot \frac{c_1(j-h_r)}{2n}
	= \frac{c_1c_2}{4}\lfloor\theta n\rfloor(\lfloor\theta n\rfloor-1)
	\ge c_3 \theta^2n^2.
\end{align}

We next derive an upper bound for $\mathbb{E}_n \big((\widetilde{\mathcal{N}}_{i,r})^2\,\big|\,\widetilde{B}_{r}\big)$. Note that
\begin{align}\label{eq:2nd moment}
\begin{aligned}
	\mathbb{E}_n \big((\widetilde{\mathcal{N}}_{i,r})^2\,\big|\,\widetilde{B}_{r}\big) 
	&= \mathbb{E}_n \big(\widetilde{\mathcal{N}}_{i,r}\,\big|\,\widetilde{B}_{r}\big)
	+ \sum_{(j,l)\ne (j',l')} \mathbb{P}_n\big(\tau_i(y^j_l),\,\tau_i(y^{j'}_{l'})<\infty\,\big|\,\widetilde{B}_r\big)\\
	&\le \mathbb{E}_n \big(\widetilde{\mathcal{N}}_{i,r}\,\big|\,\widetilde{B}_{r}\big) +2\sum_{(j,l)\ne (j',l'):j\le j'}\mathbb{P}_n\big(\tau_i(y^j_l),\,\tau_i(y^{j'}_{l'})<\infty\,\big|\,\widetilde{B}_r\big),
\end{aligned}
\end{align}
where the sum on the first line runs over all $h_r<j,j'\le h_{r-1}$, $1\le l\le q_n(j)-m$, $1\le l'\le q_n(j')-m$ with $(j,l)\neq (j',l')$, while the sum on the last line further restricts to $j\le j'$.

By \eqref{eq:no merge2}, \eqref{H1} and \eqref{H2}, for $(j,l)\ne (j',l')$ with $j\le j'$,
\begin{align*}
    \p_n\big( \tau_i(v^j_l) < \infty \,\big|\, \tau_i(v^{j'}_{l'})<\infty,\,\widetilde{B}_{r} \big) 
    &\le \frac{1}{m}\bigg\{1-\big(1+o_n(1)\big)\bigg[1-\frac{c_4m}{n}\bigg]^{\theta n}\bigg\}\le c_5\theta,
\end{align*}
for all $1\le m\le M_0$ and sufficiently large $n$.  
Consequently,
\begin{align*}
    &\sum_{(j,l)\ne (j',l'):j\le j'}\mathbb{P}_n\big(\tau_i(y^j_l),\,\tau_i(y^{j'}_{l'})<\infty\,\big|\,\widetilde{B}_r\big)\\
    &= \sum_{(j,l)\ne (j',l'):j\le j'}\mathbb{P}_n\big(\tau_i(y^j_l)<\infty\,\big|\,\tau_i(y^{j'}_{l'})<\infty,\,\widetilde{B}_r\big)\,\mathbb{P}_n\big(\tau_i(y^{j'}_{l'})<\infty\,\big|\,\widetilde{B}_r\big)\\
    &\le c_5\theta\cdot \sum_{(j,l)\ne (j',l'):j\le j'} \mathbb{P}_n\big(\tau_i(y^{j'}_{l'})<\infty\,\big|\,\widetilde{B}_r\big)\le c_6 \theta^2n^2\cdot \mathbb{E}_n \big(\widetilde{\mathcal{N}}_{i,r}\,\big|\,\widetilde{B}_{r}\big).
\end{align*}
Substituting this estimate into \eqref{eq:2nd moment} yields that for all large $n$,
\begin{align*}
    \mathbb{E}_n \big((\widetilde{\mathcal{N}}_{i,r})^2\,\big|\,\widetilde{B}_{r}\big)
    \le (1+2c_6 \theta^2n^2)  \mathbb{E}_n \big(\widetilde{\mathcal{N}}_{i,r}\,\big|\,\widetilde{B}_{r}\big)
    &\overset{\eqref{eq:first}}{\le} \frac{1+2c_6 \theta^2n^2}{c_3\theta^2n^2} \Big[\mathbb{E}_n \big(\widetilde{\mathcal{N}}_{i,r}\,\big|\,\widetilde{B}_{r}\big)\Big]^2\le c_7\Big[\mathbb{E}_n \big(\widetilde{\mathcal{N}}_{i,r}\,\big|\,\widetilde{B}_{r}\big)\Big]^2.
\end{align*}
Finally, taking $\gamma=\tfrac{1}{2}c_3\theta^2$, we obtain
\[
\mathbb{P}_n(\mathcal{N}_{i,r}\ge \gamma n^2\,\big|\,\widetilde{B}_r)
\;\ge\; \mathbb{P}_n\!\left(\mathcal{N}_{i,r}\ge \tfrac{1}{2}\mathbb{E}_n(\mathcal{N}_{i,r})\,\Big|\,\widetilde{B}_r\right)
\;\ge\; \frac{\Big[\mathbb{E}_n \big(\widetilde{\mathcal{N}}_{i,r}\,\big|\,\widetilde{B}_{r}\big)\Big]^2}{4\,\mathbb{E}_n \big((\widetilde{\mathcal{N}}_{i,r})^2\,\big|\,\widetilde{B}_{r}\big)}
\;\ge\; \frac{1}{4c_7}.
\]
This completes the proof.
\end{proof}
Finally, we proceed to the proof of Proposition~\ref{p:tau_j}.  

\begin{proof}[Proof of Proposition~\ref{p:tau_j} \eqref{item:tau1}]
For any vertex $v^j_l$, set
\[
    \tau(v^j_l):=\inf\big\{\tau_i(v^j_l):1\le i\le m\big\}
    =\inf\Big\{1\le k\le j-h_r : \text{Anc}(v^j_l,j-k)\in \bigcup_{i=1}^m U_{i,r}\Big\}.
\]

We first record the following estimate, whose proof will be given afterwards.

\begin{lemma}\label{lem:tauvil}
Uniformly for all $1\le m\le M_0$, $1\le r\le K_\theta$, and vertices $v^j_l$ satisfying $h_r<j\le h_{r-1}$ and $1\le l\le q_n(j)-m$, we have
\begin{align}\label{eq:tau}
    \mathbb{P}_n\big(\tau(v^j_l)=\infty \,\big|\, \widetilde{B}_r\big)
    =\prod_{h=0}^{j-h_r-1}\Bigg[1-\frac{m\sigma^2_n(j-h)}{q_n(j-h)}\big(1+o_n(1)\big)\Bigg].
\end{align}
\end{lemma}

Assuming Lemma~\ref{lem:tauvil}, the desired estimate~\eqref{eq:no merge} follows directly from the exchangeability of the vector 
$\big(\mathscr{T}^n(v): |v|=h_r\big)$.
\end{proof}

\begin{proof}[Proof of Lemma \ref{lem:tauvil}]
We first note that
\[
    \{\tau(v^j_l)=\infty\}=\{\tau(v^j_l)>j-h_r\}.
\]
Thus it suffices to prove that, uniformly for the quantities appearing in the lemma and $0\le h\le j-h_r-1$,
\begin{align}
    \mathbb{P}_n\big(\tau(v^j_l)> h+1\,\big|\, \tau(v^j_l)>h,\,\widetilde{B}_r\big)
    =1-\frac{m\sigma^2_n(j-h)}{q_n(j-h)}\big(1+o_n(1)\big).
\end{align}
The argument parallels that of Proposition~\ref{lem:Ur independence}.  
For $h_r\le j_2< j_1\le h_{r-1}$, set
\begin{align*}
    \widetilde{B}_r^{[j_1,j_2]}
    &:=\Big\{\#\big\{\text{Anc}(u^{h_{r-1}}_i,j_2):1\le i\le m\big\}
       =\#\big\{\text{Anc}(u^{h_{r-1}}_i,j_1):1\le i\le m\big\}=m\Big\}\\
    &=\Big\{\#\text{\,ancestors of }\{u^{h_{r-1}}_i:1\le i\le m\}\text{ remains $m$ between heights }[j_2,j_1]\Big\}.
\end{align*}
Then we have the decomposition
\(\widetilde{B}_r=\widetilde{B}^{[j-h,h_{r-1}]}_r\cap \widetilde{B}^{[h_r,j-h]}_r.\)
By Proposition~\ref{prop:uniform} and the uniform ordering of vertices $v^j_l$, the vector
\[
    \Big(
    \text{Anc}\big(u^{h_{r-1}}_1,j-h\big),\ldots,
    \text{Anc}\big(u^{h_{r-1}}_m,j-h\big),
    \text{Anc}\big(v^j_l,j-h\big)
    \Big),
\]
together with the event $\{\tau(v^j_l)> h\}\cap \widetilde{B}^{[j-h,h_{r-1}]}_r$, forms a $(m,j-h,m+1)$-uniform selection pair, with $h_{r-1}$ serving as the base height $h^*$ in Definition~\ref{def:uniform}.
Hence, by Lemma~\ref{lem:AX},
\begin{align}\label{eq:vil}
\begin{aligned}
    \mathbb{P}_n\big(\tau(v^j_l)>h+1,\, \widetilde{B}_r\big)
    &=\mathbb{P}_n\big(\tau(v^j_l)>h,\,\widetilde{B}^{[j-h,h_{r-1}]}_r \big)\\
    &\quad\cdot\mathbb{P}_n\Big[\mathfrak{p}(u^{j-h}_1),\ldots,\mathfrak{p}(u^{j-h}_{m+1})\text{ are distinct, }\#\big\{\text{Anc}(u^{j-h}_i,h_r):1\le i\le m\big\}=m\Big].
\end{aligned}
\end{align}
 For brevity, denote
\[
    p_n:=\mathbb{P}_n\Big[\mathfrak{p}(u^{j-h}_1),\ldots,\mathfrak{p}(u^{j-h}_{m+1})\text{ are distinct, }
    \#\big\{\text{Anc}(u^{j-h}_j,h_r):1\le i\le m\big\}=m\Big].
\]
Note that
\[
\Big(\{\mathfrak{p}(u^{j-h}_1),\ldots,\mathfrak{p}(u^{j-h}_{m+1})\text{ are distinct}\},\ \big\{\mathfrak{p}(u^{j-h}_j):1\le i\le m\big\}\Big)
\]
is a $(m+1,1,m)$-uniform selection pair with $j-h$ as the base height.  
Applying Lemma~\ref{lem:AX} again yields
\begin{align*}
    p_n=\,&\mathbb{P}_n\big(\mathfrak{p}(u^{j-h}_1),\ldots,\mathfrak{p}(u^{j-h}_{m+1})\text{ are distinct}\big)\cdot \mathbb{P}_n\big(\#\big\{\text{Anc}(u^{j-h-1}_j,h_r):1\le i\le m\big\}=m\big).    
\end{align*}
By \eqref{eq:distinct2} in Corollary~\ref{distinctprob},
\begin{align*}
    &\quad\,\mathbb{P}_n\Big(\mathfrak{p}(u^{j-h}_1),\ldots,\mathfrak{p}(u^{j-h}_{m+1})\text{ are distinct}\,\Big|\,\mathfrak{p}(u^{j-h}_1),\ldots,\mathfrak{p}(u^{j-h}_{m})\text{ are distinct} \Big)= 1-\dfrac{m\sigma_n^2(j-h)}{q_n(j-h)}\big(1+o_n(1)\big).
\end{align*}
Substituting into \eqref{eq:vil}, we obtain
\begin{align*}
    \mathbb{P}_n\big(\tau(v^j_l)>h+1,\, \widetilde{B}_r\big)
    &=\left[1-\dfrac{m\sigma_n^2(j-h)}{q_n(j-h)}\big(1+o_n(1)\big)\right]\mathbb{P}_n\big(\tau(v^j_l)>h,\, \widetilde{B}^{[j-h,h_{r-1}]}_r \big)\\
    &\quad\cdot \mathbb{P}_n\big(\mathfrak{p}(u^{j-h}_1),\ldots,\mathfrak{p}(u^{j-h}_{m+1})\text{ are distinct}\big)\cdot \mathbb{P}_n\big(\#\big\{\text{Anc}(u^{j-h-1}_j,h_r):1\le i\le m\big\}=m\big).
\end{align*} 
Applying Lemma~\ref{lem:AX} once more, we deduce
\begin{align*}
    \mathbb{P}_n\big(\tau(v^j_l)>h,\, \widetilde{B}_r\big)
    &=\mathbb{P}_n\big(\tau(v^j_l)>h,\,\widetilde{B}^{[j-h,h_{r-1}]}_r \big)\\
    &\quad\cdot \mathbb{P}_n\big(\mathfrak{p}(u^{j-h}_1),\ldots,\mathfrak{p}(u^{j-h}_{m+1})\text{ are distinct}\big)\cdot \mathbb{P}_n\big(\#\big\{\text{Anc}(u^{j-h-1}_j,h_r):1\le i\le m\big\}=m\big).
\end{align*}
Therefore,
\[
    \mathbb{P}_n\big(\tau(v^j_l)>h+1,\, \widetilde{B}_r\big)
    =\left[1-\dfrac{m\sigma_n^2(j-h)}{q_n(j-h)}\big(1+o_n(1)\big)\right]\mathbb{P}_n\big(\tau(v^j_l)>h,\, \widetilde{B}_r\big).
\]
This proves the lemma.
\end{proof}
\begin{proof}[Proof of Proposition~\ref{p:tau_j} \eqref{item:tau2}]
The argument for part~\eqref{item:tau2} is essentially the same as that for part~\eqref{item:tau1}, so we only give a sketch of the proof here.

Fix vertices \( v^j_l,\,v^{j'}_{l'} \) such that \( h_r < j \le j' \le h_{r-1} \), \( 1 \le l \le q_n(j)-m \), and \( 1 \le l' \le q_n(j')-m \).  
It suffices to show that, uniformly in \( 1 \le k' \le j' - h_r \),
\begin{equation}\label{eq:no_merge2}
\begin{aligned}
&\p_n\big( \tau(v^j_l) = \infty ~\big|~ \tau(v^{j'}_{l'}) = k',\, v^{j'}_{l'} \notin \text{Desc}(v^j_l),\, \widetilde{B}_r \big) \\
=~& (1 + o_n(1)) 
\prod_{h=0}^{j-(j'-k')-2}
\Bigg[ 1 - \frac{(m+1)\sigma^2}{q_n(j-h)} (1 + o_n(1)) \Bigg] \cdot
\prod_{h={j-(j'-k')}}^{j-{h_r-1}}
\Bigg[ 1 - \frac{m\sigma^2}{q_n(j-h)} (1 + o_n(1)) \Bigg].
\end{aligned}
\end{equation}

Define
\[
q_n(h) := \p_n\big( \tau(v^j_l) > h+1 ~\big|~ \tau(v^j_l) > h,\, \tau(v^{j'}_{l'}) = k',\, v^{j'}_{l'} \notin \text{Desc}(v^j_l),\, \widetilde{B}_r \big).
\]
Under this conditioning, \( j' - k' \) represents the height at which \( v^{j'}_{l'} \) merges into \( U_{i,r} \).
Taking \( h_{r-1} \) as the base height, we obtain by exchangeability that:

\begin{enumerate}[label=(\roman*)]
\item If \( j-h > j'-k' \), then the vector
\[
\Big(
\text{Anc}(u^{h_{r-1}}_1,j-h), \ldots,
\text{Anc}(u^{h_{r-1}}_m,j-h),
\text{Anc}(v^j_l,j-h), \text{Anc}(v^{j'}_{l'},j-h)
\Big),
\]
together with the event
\[
\Big\{
\tau(v^j_l) > h,\,
\tau(v^{j'}_{l'}) = k',\,
v^{j'}_{l'} \notin \text{Desc}(v^j_l),\,
\widetilde{B}_r^{[j-h, h_{r-1}]}
\Big\},
\]
forms an \((m,j-h, m+2)\)-uniform selection pair.

\item If \( j-h \le j'-k' \), then the vector
\[
\Big(
\text{Anc}(u^{h_{r-1}}_1,j-h), \ldots,
\text{Anc}(u^{h_{r-1}}_m,j-h),
\text{Anc}(v^j_l,j-h)
\Big),
\]
together with the event
\[
\Big\{
\tau(v^j_l) > h,\,
\tau(v^{j'}_{l'}) = k',\,
v^{j'}_{l'} \notin \text{Desc}(v^j_l),\,
\widetilde{B}_r^{[j-h, h_{r-1}]}
\Big\},
\]
forms an \((m,j-h, m+1)\)-uniform selection pair.
\end{enumerate}

Following the same reasoning as in the proof of Proposition~\ref{p:tau_j}~(1), we obtain:

\begin{enumerate}[label=(\roman*)]
\item Uniformly over \( 0 \le h \le j - h_r - 1 \) with \( j - h > j' - k' + 1 \),
\begin{align*}
q_n(h)
&= \mathbb{P}_n\Big(
\mathfrak{p}(u^{j-h}_1), \ldots, \mathfrak{p}(u^{j-h}_{m+2}) \text{ are distinct}
~\Big|~
\mathfrak{p}(u^{j-h}_1), \ldots, \mathfrak{p}(u^{j-h}_{m+1}) \text{ are distinct}
\Big) \\
&\overset{\eqref{eq:precise}}{=}
1 - \frac{(m+1)\sigma_n^2(j-h)}{q_n(j-h)} (1 + o_n(1)).
\end{align*}

\item Uniformly over \( 0 \le h \le j - h_r - 1 \) with \( j - h \le j' - k' \),
\begin{align*}
q_n(h)
&= \mathbb{P}_n\Big(
\mathfrak{p}(u^{j-h}_1), \ldots, \mathfrak{p}(u^{j-h}_{m+1}) \text{ are distinct}
~\Big|~
\mathfrak{p}(u^{j-h}_1), \ldots, \mathfrak{p}(u^{j-h}_{m}) \text{ are distinct}
\Big) \\
&\overset{\eqref{eq:precise}}{=}
1 - \frac{m\sigma_n^2(j-h)}{q_n(j-h)} (1 + o_n(1)).
\end{align*} 

\item For \( h \) such that \( j - h = j' - k' + 1 \),
\begin{align*}
q_n(h)
&= \mathbb{P}_n\Big( \mathfrak{p}(u_2^{j-h}), \ldots, \mathfrak{p}(u_{m+2}^{j-h}) \text{ are distinct}
~\Big|~ \mathfrak{p}(u_2^{j-h}), \ldots, \mathfrak{p}(u_{m+1}^{j-h}) \text{ are distinct},\, \mathfrak{p}(u_1^{j-h}) = \mathfrak{p}(u_2^{j-h})
 \Big)\\
 &\overset{\eqref{eq:distinct2}}{=} 1 + o_n(1).
\end{align*}
\end{enumerate}
Combining these estimates, we conclude that
\[
\p_n\big(
\tau(v^j_l) = \infty
~\big|~
\tau(v^{j'}_{l'}) = k',~
v^{j'}_{l'} \notin \text{Desc}(v^j_l),~
\widetilde{B}_r
\big)
= \prod_{h=0}^{j-{h_r-1}} p_n(h)
= \text{RHS of } \eqref{eq:no_merge2}.
\]
This completes the proof.
\end{proof}

\section{Proof of Theorem \ref{main}}\label{sec:proof}

In the previous section we established the tightness of $\widetilde{\mathscr{T}}^n$. 
Here we first prove the finite-dimensional convergence of $\widetilde{\mathscr{T}}^n$ (Section~\ref{sec:fd}); combined with tightness, this yields convergence of the contour functions. 
We then show in Section~\ref{sec:joint} that the discrepancy between the height and contour functions vanishes in the limit, so both processes converge to the same limiting process, which completes the proof of Theorem~\ref{main}.

\subsection{Finite-dimensional convergence of $\widetilde{\mathscr{T}}^{n}$}\label{sec:fd}
Recall that $\mathcal{T}^{\ell,\sigma}$ is the rooted ordered real tree encoded by
$W^{\ell^\sigma}_{\alpha^{-1}(I(\ell)t)}$, and that $\mathcal{T}^{\ell,\sigma}_k$ denotes its uniform
$k$-point subtree, whose law is described in Theorem~\ref{thm:tree_construction}. Let $L_1,\ldots, L_k$ be the $k$ leaves of $\mathcal{T}^{\ell,\sigma}_k$. \footnote{Recall that if $p=p_{W^{\ell^\sigma}_{\alpha^{-1}(I(\ell)t)}}:[0,1]\rightarrow \mathcal{T}^{\ell,\sigma}$ is the canonical projection, then \vspace{-0.2cm}
\[(L_1,\ldots, L_k)\overset{d}{=}(p(U_1),\ldots,p(U_k)),\]
with $U_i$ are i.i.d. Uniform$[0,1]$.}
Also recall that \( \widetilde{\mathscr{T}}^n \) is the geometric tree obtained from \( \mathscr{T}^n \) by assigning length \(1/n\) to each edge. Let \((\widetilde{V}^n_1,\ldots,\widetilde{V}^n_{\# V(\widetilde{\mathscr{T}}^n)})\) be a uniform random ordering of \(V(\widetilde{\mathscr{T}}^n)\), and let \(\widetilde{\mathscr{T}}^n_k\) be the rooted ordered subtree of \( \widetilde{\mathscr{T}}^n \) spanned by the root and \(\{\widetilde{V}^n_1,\ldots,\widetilde{V}^n_k\}\).
In this subsection we prove the following proposition on the finite-dimensional convergence of $\widetilde{\mathscr{T}}^{n}$.

\begin{proposition}\label{sconvergence}
  Assume \eqref{H1} and \eqref{eq:3rd}. Then for every $k\ge 1$,
  \begin{align}\label{eq:fd_convergence}
      \big(\widetilde{\mathscr{T}}^n_k,\{\widetilde{V}^n_i\}_{1\le i\le k}\big)\overset{d}{\longrightarrow} \big(\mathcal{T}_k^{\ell,\sigma},\{L_i\}_{1\le i\le k}\big)\text{ under the distance \eqref{fidis}}.
  \end{align}
\end{proposition}
\begin{proof}
To begin with, we reduce the proof by controlling the heights of the sampled leaves.  
Fix $k\ge 1$. Let $\mathbf{a}_1 \ge \mathbf{a}_2 \ge \cdots \ge \mathbf{a}_k$ denote the heights of the leaves of the tree $\mathcal{T}^{\ell,\sigma}_k$ and   
$\mathbf{a}^n_1 \ge \mathbf{a}^n_2 \ge \cdots \ge \mathbf{a}^n_k$ be the order statistics of the set 
\(\big\{\lvert \widetilde{V}^n_1\rvert,\ldots,\lvert \widetilde{V}^n_k\rvert\big\}\). For any $\delta >0$, define
\begin{align*}
        A^n_\delta&:=\big\{\mathbf{a}^n_i\in (\delta , h(\ell)-\delta )\ \forall\,i=1,\ldots,k\big\}\cap \big\{\mathbf{a}^n_i\neq \mathbf{a}^n_j\ \forall\, 1\le i\neq j\le k\big\},\\
        A_\delta&:=\big\{\mathbf{a}_i\in (\delta , h(\ell)-\delta )\ \forall\,i=1,\ldots,k\big\}\cap \big\{\mathbf{a}_i\neq \mathbf{a}_j\ \forall\, 1\le i\neq j\le k\big\}.
    \end{align*}
Then we have the following:\phantom\qedhere
\begin{lemma}
    Assume \eqref{H1} and \eqref{eq:3rd}. 
    Then for any $k\ge 1$ and $\eta>0$, there exists $\delta=\delta(k,\eta)>0$ such that
    \begin{align}
        \liminf_{n\rightarrow\infty}\mathbb{P}_n\left(A^n_\delta\right)>1-\eta,\quad
        \mathbb{P}(A_\delta)>1-\eta.\label{eq:height_tight}
    \end{align}
\end{lemma}
\begin{proof}
The first bound follows directly from \eqref{H1} \ref{assump1}--\ref{assump2}.
The second bound follows from the piecewise Kingman representation in Proposition~\ref{prop:piecewise_Kingman}, together with the definition of continuous profiles in Section~\ref{sec:setup}.
\end{proof}
Thanks to the lemma, it remains to prove for any $k\ge 1$ and $\eta>0$, there exists a coupling of $\Big(\big(\widetilde{\mathscr{T}}^n_k, \{\widetilde{V}^n_i\}_{1\le i\le k}\big)\,\Big|\,A^n_\delta\Big)$, $n\ge 1$ and $\Big(\big(\mathcal{T}^{\ell,\sigma}_k,\{L_i\}_{1\le i\le k}\big)\,\Big|\,A_\delta\Big)$ such that
    \begin{align}\label{eq:fd_convergence_suffice}
        \mathbb{P}\left\{d\Big(\Big(\big(\widetilde{\mathscr{T}}^n_k, \{\widetilde{V}^n_i\}_{1\le i\le k}\big)\,\Big|\,A^n_\delta\Big),\Big(\big(\widetilde{\mathscr{T}}^n_k, \{\widetilde{V}^n_i\}_{1\le i\le k}\big)\,\Big|\,A^n_\delta\Big)\Big)\le \eta\right\}\ge 1-2\eta,
    \end{align}
    where $d$ denotes the distance in \eqref{fidis}.
    Once \eqref{eq:fd_convergence_suffice} is verified, the desired convergence
\eqref{eq:fd_convergence} follows from the following proposition: 
    \begin{proposition}
Let $X$ and $(X_n)_{n\ge 1}$ be random elements on a metric space $(S,\rho)$.
	If there exists a sequence $\Big\{\big(X^k,(X^k_n)_{n\ge 1}\big):k\ge 1\Big\}$ of couplings of $X$ and $(X_n)_{n\ge 1}$ satisfying that for each $\varepsilon>0$,
	\begin{align}\label{rho}		\lim_{k\rightarrow\infty}\limsup_{n\rightarrow\infty}\p\Big(\rho(X^k,X^k_n)>\varepsilon\Big)=0,
	\end{align}
	then $X_n\overset{d}{\longrightarrow}X$.
\end{proposition}
\begin{proof}
	For each closed subset $F$ of $(S,\rho)$ and $\varepsilon>0$,
	\begin{align*}
		\p(X_n\in F)=\p(X^k_n\in F)\le \p\big(\rho(X^k,X^k_n)>\varepsilon\big)+\p(X\in F_\varepsilon),
	\end{align*}
	where $F_\varepsilon=\{x\in S:\rho(x,F)\le \varepsilon\}$. Consequently, 
	\begin{align*}
		\limsup_{n\rightarrow\infty}\p(X_n\in F)\le \lim_{k\rightarrow\infty}\limsup_{n\rightarrow\infty}\p\Big(\rho(X^k,X^k_n)>\varepsilon\Big)+\p(X\in F_\varepsilon)=\p(X\in F_\varepsilon).
	\end{align*}
	Letting $\varepsilon\downarrow 0$ yields
	\[\limsup_{n\rightarrow\infty}\p(X_n\in F)\le\p(X\in F),\]
    which is equivalent to $X_n\overset{d}{\longrightarrow}X$.
\end{proof}

For the proof of \eqref{eq:fd_convergence_suffice}, we introduce the following notation. For $s\in [\delta,h(\ell)-\delta]$, let
\begin{align*}
    Y_s=Y_s^{(k)}&:=\#\,\text{ancestors of $L_1,\ldots, L_k$ in $\mathcal{T}^{\ell,\sigma}_k$ at height $s$},\\
    Y^n_s=Y_s^{n,(k)}&:=\#\,\text{ancestors of $\widetilde{V}^n_1,\ldots,\widetilde{V}^n_k$ in $\widetilde{\mathscr{T}}^{n}_k$ at height $\lceil ns \rceil /n$}.
\end{align*}
The proof is accomplished by a chain of lemmas.
\begin{lemma}\label{lem:height_convergence}
    For any $k\ge 1$ and $\delta>0$,
    \[\big((\mathbf{a}^n_1,\ldots, \mathbf{a}^n_k)\,\big|\, A^n_\delta\big)\overset{d}{\longrightarrow}\big((\mathbf{a}_1,\ldots, \mathbf{a}_k)\,\big|\, A_\delta\big).\]
\end{lemma}
\begin{proof}
This follows directly from \eqref{H1}, \eqref{eq:total_popuplation} and the first step of the equivalence construction of $\mathcal{T}^\ell_k$ presented after Proposition~\ref{prop:piecewise_Kingman}.
\end{proof}
\begin{lemma}\label{lem:Z_convergence}
Fix $k\ge 1$ and $\delta>0$. 
Let $(a_1^n,\ldots,a_k^n)\in\big((\delta,h(\ell)-\delta)\cap\frac{1}{n}\mathbb{Z}\big)^k$ for $n\ge1$, and 
$(a_1,\ldots,a_k)\in(\delta,h(\ell)-\delta)^k$ satisfy
\[  
(a_1^n,\ldots,a_k^n)\longrightarrow(a_1,\ldots,a_k)\quad\text{and}\quad a_1>\cdots>a_k.
\]
Then
\begin{align}\label{eq:Z_convergence}
\begin{aligned} \Big(\big(Y^n(s):s\in[\delta, h(\ell)-\delta]\big)\,\big|\,(\mathbf{a}^n_1,\ldots, \mathbf{a}^n_k)=(a^n_1,\ldots,a^n_k)\Big)\\\overset{d}{\longrightarrow} \Big(\big(Y(s):s\in[\delta, h(\ell)-\delta]\big)\,\big|\,(\mathbf{a}_1,\ldots, \mathbf{a}_k)=(a_1,\ldots,a_k)\Big). \end{aligned}
\end{align}
\end{lemma}
\begin{lemma}\label{lem:Z_to_tree}
    Let $\delta>0$, and let 
    \[
        y^n=\big(y^n(s):s\in[\delta,h(\ell)-\delta]\big), \qquad
        y=\big(y(s):s\in[\delta,h(\ell)-\delta]\big)
    \]
    be the deterministic paths in the value spaces of 
    $(Y^n_s\,|\, A^n_\delta)$ and $(Y_s\,|\, A_\delta)$, respectively, such that
    \[
        y^n \;\longrightarrow\; y \quad \text{in } D[\delta,h(\ell)-\delta]
        \quad\text{and}\quad \text{$y$ has jumps of size at most $1$.}
    \]
    Then there exists a coupling of $\big(\widetilde{\mathscr{T}}^n_k \,\big|\, A^n_\delta,\; Y^n=y^n\big)$, $n\ge 1$, and $\big(\mathcal{T}^{\ell,\sigma}_k \,\big|\, A_\delta,\; Y=y\big)$ such that a.s.,
    \[
        \limsup_{n\rightarrow\infty}d\Big\{\Big(\big(\widetilde{\mathscr{T}}^n_k,\{\widetilde{V}^n_i\}_{1\le i\le k}) \,\big|\, A^n_\delta,\; Y^n=y^n\big)\Big),\Big(\big(\mathcal{T}^{\ell,\sigma}_k,\{L_i\}_{1\le i\le k}) \,\big|\, A_\delta,\; Y=y\big)\Big)\Big\}\le 2k\delta.
    \]
\end{lemma}

\begin{proof}[Proof of \eqref{eq:fd_convergence_suffice} assuming Lemmas~\ref{lem:Z_convergence} and \ref{lem:Z_to_tree}]
It follows from Lemmas~\ref{lem:height_convergence} and \ref{lem:Z_convergence} that
\[
    (Y^n \mid A^n_\delta)\;\overset{d}{\longrightarrow}\; (Y \mid A_\delta)
    \qquad\text{in }D[\delta,h(\ell)-\delta].
\]
    Moreover, by the piecewise Kingman representation of $Y$ in Proposition~\ref{prop:piecewise_Kingman}, we have a.s., $(Y\,|\,A_\delta)$ has jumps of size at most $1$.
    Thus, 
    further combining Lemma \ref{lem:Z_to_tree} with sufficiently small $\delta>0$ yields \eqref{eq:fd_convergence_suffice}. 
\end{proof}
We now turn to the proof of Lemmas~\ref{lem:Z_convergence} and \ref{lem:Z_to_tree}. 

\begin{proof}[Proof of Lemma~\ref{lem:Z_convergence}]
We begin by introducing a preparatory lemma, for which we set up some notation.

Let $(b_1^n,\ldots,b_m^n)\in\big((\delta,h(\ell)-\delta)\cap\frac{1}{n}\mathbb{Z}\big)^m$ for $n\ge1$, and 
$(b_1,\ldots,b_m)\in(\delta,h(\ell)-\delta)^m$ satisfy
\begin{align*}
    (b_1^n,\ldots,b_m^n)\longrightarrow(b_1,\ldots,b_m), 
\qquad  
b^n_1>\cdots>b^n_m, 
\qquad  
b_1>\cdots>b_m,\\
\{b^n_1,\ldots,b^n_m\}\cap\{a^n_1,\ldots,a^n_k\}=\emptyset,\quad \{b_1,\ldots,b_m\}\cap\{a_1,\ldots,a_k\}=\emptyset.
\end{align*}
Define the deterministic paths  
\(y^n=\big(y^n(s):s\in[\delta,h(\ell)-\delta]\big)\)  
(resp. \(y=\big(y(s):s\in[\delta,h(\ell)-\delta]\big)\))  
as follows: viewed backwards from \(h(\ell)-\delta\) down to \(\delta\), the path starts at  
\(y^n(h(\ell)-\delta)=0\) (resp. \(y(h(\ell)-\delta)=0\)), has upward jumps of size \(+1\) at times  
\(a_1^n,\ldots,a_k^n\) (resp. \(a_1,\ldots,a_k\)), and downward jumps of size \(-1\) at times  
\(b_1^n,\ldots,b_m^n\) (resp. \(b_1,\ldots,b_m\)). (If a time \(a_i^n\) coincides with some \(b_j^n\) (resp. \(a_i\) with \(b_j\)), then the two jumps cancel and the path has no jump at that time.)

For any $s\in (\delta,h(\ell)-\delta)\cap\frac{1}{n}\mathbb{Z}$, define the event
\begin{align}\label{eq:Bs}
    B_s=B_s(n,\delta):=\left\{Y^n|_{[s,h(\ell)-\delta)\cap\frac{1}{n}\mathbb{Z}}=y^n|_{[s,h(\ell)-\delta)\cap\frac{1}{n}\mathbb{Z}}\right\}\cap A^n_\delta.
\end{align}
    We now state the lemma (its proof is postponed):
\begin{lemma}\label{lem:local_limit}
For any $\delta>0$, $k\ge1$, and sequences $(a_i^n)$, $(a_i)$, $(b_i^n)$, $(b_i)$ satisfying the above assumptions, we have uniformly for all $s\in (\delta,h(\ell)-\delta)\cap\frac{1}{n}\mathbb{Z}$,
\begin{align*}
        \mathbb{P}_n\left(Y^n(s-\tfrac{1}{n})=y^n(s-\tfrac{1}{n})\,|\,B_s\right)
        =\,\begin{cases}
            1+o_n(1),&\quad\text{if }s-\tfrac{1}{n}\in \{a^n_1,\ldots,a^n_k\},\\
            \dfrac{\binom{y^n(s)}{2}\sigma_n(s)^2}{q_n(s)}\big(1+o_n(1)\big),&\quad\text{if }s-\tfrac{1}{n}\in \{b^n_1,\ldots,b^n_m\},\\
            1-\dfrac{\binom{y^n(s)}{2}\sigma_n(s)^2}{q_n(s)} \big(1 + o_n(1)\big),&\quad\text{otherwise}.
        \end{cases}
    \end{align*}
\end{lemma}
Let us use the lemma to prove Lemma~\ref{lem:Z_convergence}. Let
\[
p_n := \mathbb{P}_n(Y^n = y^n\,|\,A^n_\delta),
\]
and let \(p\) denote the density of the probability that the conditioned process 
\((Y\,|\,A_\delta)\) follows the same jump pattern as \(y\), with upward jump times 
perturbed by \(\mathrm{d}a_1,\ldots,\mathrm{d}a_k\).

It follows from Proposition~\ref{prop:piecewise_Kingman} that a.s.,
\[\{b_1,\ldots,b_m\}\cap\{a_1,\ldots,a_k\}=\emptyset.\]
Thus by \cite[Theorem~3.3]{billingsley2013convergence}, it is enough to show the following local limit estimate:
for any $\delta>0$, $k\ge1$, and sequences $(a_i^n)$, $(a_i)$, $(b_i^n)$, $(b_i)$ satisfying the above assumptions, 
\begin{align}\label{eq:local_limit}
    n^k\, p_n \;\longrightarrow\; p.
\end{align}

    To this end, we first note that
\[
\{Y^n = y^n\}
  = \Big\{\,Y^n\big|_{(\delta,h(\ell)-\delta)\cap \tfrac1n\mathbb{Z}}
      = y^n\big|_{(\delta,h(\ell)-\delta)\cap \tfrac1n\mathbb{Z}}\,\Big\},
\]
by the definition of $Y^n$.
Thus, it follows from Lemma \ref{lem:local_limit} and \eqref{H1} that
\begin{align}\label{eq:density_convergence}
    \begin{aligned}
        n^k p_n&=n^k\prod_{s\in (\delta+\frac{1}{n},h(\ell)-\delta)\cap \frac{1}{n}\mathbb{Z}}\mathbb{P}_n\left(Y^n(s-\tfrac{1}{n})=y^n(s-\tfrac{1}{n})\,|\,B_s\right)\\
        &\underset{n\rightarrow\infty}{\longrightarrow}\,\exp\left[-\int_{\delta}^{h(\ell)-\delta}\binom{y(s)}{2}\frac{\sigma(s)^2}{\ell(s)}\d s\right] \prod_{i=1}^k\binom{y(b_i)}{2}\dfrac{\sigma(b_i)^2}{\ell(b_i)}.       
    \end{aligned}
    \end{align}

    On the other hand, by Proposition~\ref{prop:piecewise_Kingman}, \((Y\,|\,(\mathbf{a}_1,\ldots,\mathbf{a}_k)=(a_1,\ldots,a_k))\) is distributed as the piecewise Kingman's coalescent \(Z_s^{(k)}(a_1,\ldots,a_k)\). A direct inspection of its transition mechanism shows that the RHS of \eqref{eq:density_convergence} is exactly the path density \(p\). This proves \eqref{eq:local_limit}.
\end{proof}
\begin{proof}[Proof of Lemma~\ref{lem:Z_to_tree}]
The proof also requires a preliminary lemma.

Assume the same setting as in Lemma \ref{lem:local_limit}. Let $\mathbf{W}^n_i$ (resp. $\mathbf{W}_i$), $1\le i\le m$, be the leaf of $\big(\widetilde{\mathscr{T}}^n_k \,\big|\, A^n_\delta,\; Y^n=y^n\big)$ (resp. $\big(\mathcal{T}^{\ell,\sigma}_k \,\big|\, A_\delta,\; Y=y\big)$) at height $a^n_i$ (resp. $a_i$), and $\mathbf{B}^n_i$ (resp. $\mathbf{B}_i$), $1\le i\le k$, be the branching point of $\big(\widetilde{\mathscr{T}}^n_k \,\big|\, A^n_\delta,\; Y^n=y^n\big)$ (resp. $\big(\mathcal{T}^{\ell,\sigma}_k \,\big|\, A_\delta,\; Y=y\big)$) at height $b^n_i$ (resp. $b_i$).
\begin{lemma}\label{lem:uniform_order}
   Assume the same setting as in Lemma \ref{lem:local_limit}. Then the tree order of the leaves $(\mathbf{W}^n_i)$ and branching points $(\mathbf{B}_i^n)$ of $\big(\widetilde{\mathscr{T}}^n_k \,\big|\, A^n_\delta,\; Y^n=y^n\big)$ are uniform over all possible orders. The analogous result also holds for $\big(\mathcal{T}^{\ell,\sigma}_k \,\big|\, A_\delta,\; Y=y\big)$.
\end{lemma}
We postpone the proof of the lemma. 
We now work under the assumptions of Lemma~\ref{lem:Z_to_tree}.
Observe that the convergence
\(y^n \overset{d}{\longrightarrow} y\)
in the statement of Lemma~\ref{lem:Z_to_tree} is equivalent to the convergence of the corresponding upward and downward jump times. Thus Lemma \ref{lem:uniform_order} implies that for sufficiently large $n$, there exists a coupling of $\big(\widetilde{\mathscr{T}}^n_k \,\big|\, A^n_\delta,\; Y^n=y^n\big)$ and $\big(\mathcal{T}^{\ell,\sigma}_k \,\big|\, A_\delta,\; Y=y\big)$ such that the tree order of their leaves and branching points agree. That is,
\begin{align*}
    &\mathbf{W}^n_i\prec \mathbf{W}_j\text{ iff }\mathbf{W}_i\prec \mathbf{W}_j,\qquad \mathbf{B}_i^n\prec \mathbf{B}_j^n\text{ iff }\mathbf{B}_i\prec \mathbf{B}_j,\qquad\mathbf{W}^n_i\prec \mathbf{B}_j^n\text{ iff }\mathbf{W}_i\prec \mathbf{B}_j.
\end{align*}

Consequently, let $\mathbf{W}^n_{(1)},\ldots, \mathbf{W}^n_{(m)}$ (resp. $\mathbf{W}_{(1)},\ldots, \mathbf{W}_{(m)}$) denote the order statistic of $\mathbf{W}^n_1,\ldots, \mathbf{W}^n_m$ (resp. $\mathbf{W}_1,\ldots, \mathbf{W}_m$). Then for any $1\le i\le m-1$,
\begin{enumerate}
    \item the branching point of $\mathbf{W}^n_{(i)}$ and $\mathbf{W}^n_{(i+1)}$ belongs to $\{\mathbf{B}_j^n\}$ (i.e. the height of branching points is in $[\delta,h(\ell)-\delta]$) iff the branching point of $\mathbf{W}_{(i)}$ and $\mathbf{W}_{(i+1)}$ belongs to $\{\mathbf{B}_j\}$;
    \item $\mathbf{B}_j^n$ is the branching point of $\mathbf{W}^n_{(i)}$ and $\mathbf{W}^n_{(i+1)}$\quad iff\quad $\mathbf{B}_j$ is the branching point of $\mathbf{W}_{(i)}$ and $\mathbf{W}_{(i+1)}$.
\end{enumerate}
Now recall the distance defined in \eqref{fidis}. For $1\le i\le m-1$, writing $\mathbf{B}^n_{(i)}$ (resp. $\mathbf{B}_{(i)}$) for the branching point of $\mathbf{W}^n_{(i)}$ and $\mathbf{W}^n_{(i+1)}$ (resp. $\mathbf{W}_{(i)}$ and $\mathbf{W}_{(i+1)}$), we distinguish two cases:
\begin{enumerate}
    \item if both $\mathbf{B}^n_{(i)}$ and $\mathbf{B}_{(i)}$ have heights in $(\delta, h(\ell)-\delta)$, then for some $j$, $\mathbf{B}^n_{(i)}=\mathbf{B}_j^n$ and $\mathbf{B}_{(i)}=\mathbf{B}_j$. Using the convergence of the heights of the leaves and branching points, we obtain a.s.
    \begin{align*}
        d(\mathbf{W}^n_{(i)},\mathbf{B}^n_{(i)})\longrightarrow d(\mathbf{W}_{(i)},\mathbf{B}_{(i)}),\qquad d(\mathbf{W}^n_{(i+1)},\mathbf{B}^n_{(i)})\longrightarrow d(\mathbf{W}_{(i+1)},\mathbf{B}_{(i)}),
    \end{align*}
    where $d$ denotes the tree distance in the corresponding space  
(either on $\big(\widetilde{\mathscr{T}}^n_k \mid A^n_\delta,\, Y^n=y^n\big)$ or on  
$\big(\mathcal{T}^{\ell,\sigma}_k \mid A_\delta,\, Y=y\big)$, respectively);
    \item otherwise, both $\mathbf{B}^n_{(i)}$ and $\mathbf{B}_{(i)}$  lie in $[0,\delta]$, and hence
    \[\limsup_{n\rightarrow\infty} \abs{d(\mathbf{W}^n_{(i)},\mathbf{B}^n_{(i)})-d(\mathbf{W}_{(i)},\mathbf{B}_{(i)})}\le 2\delta,\qquad \limsup_{n\rightarrow\infty} \abs{d(\mathbf{W}^n_{(i+1)},\mathbf{B}^n_{(i)})-d(\mathbf{W}_{(i)},\mathbf{B}_{(i)})}\le 2\delta.\]
\end{enumerate}

   It is also clear that
   \[d(\mathbf{W}^n_{(1)},\text{root})\longrightarrow d(\mathbf{W}_{(1)},\text{root}),\quad d(\mathbf{W}^n_{(m)},\text{root})\longrightarrow d(\mathbf{W}_{(m)},\text{root}).\]
Putting everything together, we obtain Lemma~\ref{lem:Z_to_tree} under the above coupling.
\end{proof}
\end{proof}
Finally, we verify Lemmas~\ref{lem:local_limit} and \ref{lem:uniform_order}.  
Recall the definition of \(B_s\) in \eqref{eq:Bs}.  
By recursively applying Proposition~\ref{prop:uniform} downward from height \(h(\ell)-\delta\), we obtain that for every  
\(s\in (\delta,h(\ell)-\delta)\cap \tfrac{1}{n}\mathbb{Z}\),  
conditionally on \(B_s\), the \(y^n(s)\) ancestors of  
\(\widetilde{V}^n_1,\ldots,\widetilde{V}^n_k\) at height \(s\) are distributed as a uniformly chosen
\(y^n(s)\)-tuple of vertices at that height.

\begin{proof}[Proof of Lemma~\ref{lem:local_limit}]
Using Lemma~\ref{lem:selection_pair}, we have for all  
\(s\in (\delta,h(\ell)-\delta)\cap \tfrac{1}{n}\mathbb{Z}\),
\begin{align*}
&\mathbb{P}_n\!\left(Y^n(s-\tfrac{1}{n}) = y^n(s-\tfrac{1}{n}) \,\middle|\, B_s\right)\\
=\;&
\begin{cases}
\displaystyle 
\Big(1-\frac{y^n(s)}{q^n(ns)}\Big)
\mathbb{P}_n\!\left(\#\{\mathfrak{p}(u^{ns}_1),\ldots,\mathfrak{p}(u^{ns}_{y^n(s)})\}=y^n(s)\right),
& \text{if } s-\tfrac{1}{n}\in\{a^n_1,\ldots,a^n_k\},\\
\displaystyle 
\mathbb{P}_n\!\left(\#\{\mathfrak{p}(u^{ns}_1),\ldots,\mathfrak{p}(u^{ns}_{y^n(s)})\}=y^n(s)-1\right),
& \text{if } s-\tfrac{1}{n}\in\{b^n_1,\ldots,b^n_m\},\\
\displaystyle 
\mathbb{P}_n\!\left(\#\{\mathfrak{p}(u^{ns}_1),\ldots,\mathfrak{p}(u^{ns}_{y^n(s)})\}=y^n(s)\right),
& \text{otherwise}.
\end{cases}
\end{align*}
Here, the factor \(1 - \tfrac{y^n(s)}{q^n(ns)}\) is the probability that the newly created leaf at height \(s-\tfrac{1}{n}\) avoids the \(y^n(s)\) ancestral lineages above height \(s\).  
Combining the above with \eqref{eq:precise} and \eqref{eq:distinct2} yields the desired result.
\end{proof}
\begin{proof}[Proof of Lemma \ref{lem:uniform_order}]
    The statement for 
    \(\big(\mathcal{T}^{\ell,\sigma}_k \,\big|\, A_\delta,\; Y=y\big)\)
    is a consequence of the final step in the piecewise Kingman's construction described in Proposition~\ref{prop:piecewise_Kingman}.
    
    We now treat 
    \(\big(\widetilde{\mathscr{T}}^n_k \,\big|\, A^n_\delta,\; Y^n=y^n\big)\).
    List the \(y^n(s)\) vertices at height \(s\) in tree order as
    \[
        z^{ns}_{i_1} \prec z^{ns}_{i_2} \prec \cdots \prec z^{ns}_{i_{y^n(s)}}.
    \]
    In the application of Lemma~\ref{lem:selection_pair} that appeared in the 
proof of Lemma~\ref{lem:local_limit}, we additionally keep track of the
probability that:
    \begin{enumerate}
        \item when \(s \in \{a^n_1,\ldots,a^n_k\}\), the newly created leaf is exactly \(z^{ns}_{i_j}\);
        \item when \(s \in \tfrac{1}{n}+\{b^n_1,\ldots,b^n_m\}\), the vertices  
              \(z^{ns}_{i_j}\) and \(z^{ns}_{i_{j+1}}\) are the pair that coalesces at height \(s\).
    \end{enumerate}
    In each case, one checks that the index \(j\) is uniformly distributed over all its possible values 
(that is, over \(\{1,\ldots,y^n(s)\}\) in the first case and 
\(\{1,\ldots,y^n(s)-1\}\) in the second). 

    Moreover, a straightforward combinatorial check shows that the possible choices of $j$ at each jump time $s$ are in bijection with all possible tree orders of the leaves and branch points. Hence all such orders occur with equal probability.
\end{proof}

\begin{remark}\label{rmk:equivalent}
Under \eqref{H1}, the following three statements are in fact equivalent:
\begin{enumerate}[label=(\roman*),ref=\roman*]
    \item\label{item:i} the finite-dimensional convergence in Proposition~\ref{sconvergence};
    \item\label{item:ii} convergence of the ancestral process to Kingman’s coalescent;
    \item\label{item:iii} the third-moment condition~\eqref{eq:3rd}.
\end{enumerate}

Indeed, we have already shown that \eqref{eq:3rd} implies the finite-dimensional convergence, and hence \eqref{item:iii} implies both \eqref{item:i} and \eqref{item:ii}.  
Conversely, if \eqref{eq:3rd} fails, then along any subsequence $(n_k)$ for which the limit in \eqref{eq:3rd} is strictly positive, the probability of a triple–merger event remains bounded away from \(0\). This prevents the ancestral process from converging to Kingman’s coalescent and therefore rules out the finite-dimensional convergence.  
\end{remark}

\begin{corollary}
Assume \eqref{H1}--\eqref{H2}. Then
\begin{align}\label{eq:contour_convergence}
    \big(n^{-1}C_{\lfloor 2n^2 t\rfloor}(\mathscr{T}^{n})\big)_{t\ge 0}
    \ \overset{d}{\longrightarrow}\
    \big(W_{\alpha^{-1}(t)}^{\,\ell^\sigma}\big)_{t\ge 0}\quad\text{in \(D[0,\infty)\)}.
\end{align}
\end{corollary}

\begin{proof}
Theorems~\ref{thm:tight} and~\ref{sconvergence} establish tightness and finite-dimensional convergence of the sequence \((\mathscr{T}^n)\) respectively. Recall that \(\mathcal{T}^{\ell,\sigma}_k\) is the uniform \(k\)-point subtree induced by \(W^{\ell^\sigma}_{\alpha^{-1}(I(\ell)t)}\), and note that
\[
\big(n^{-1}C_k(\mathscr{T}^{n})\big)_{0\le k\le 2(\# V(\mathscr{T}^{n})-1)}
\]
is exactly the contour function of \(\widetilde{\mathscr{T}}^{n}\).
Applying Theorem~\ref{thm:convergence} yields
\[
\big(n^{-1}C_{\lfloor 2(\# V(\mathscr{T}^{n})-1)t \rfloor}(\mathscr{T}^{n})\big)_{0\le t\le 1}
\ \overset{d}{\longrightarrow}\
\big(W_{\alpha^{-1}(I(\ell)t)}^{\,\ell^\sigma}\big)_{0\le t\le 1}.
\]
Combining this with \eqref{eq:total_popuplation} gives \eqref{eq:contour_convergence}.
\end{proof}

\subsection{Joint convergence of height and contour functions}\label{sec:joint}
In this subsection we prove the joint convergence of the (rescaled) contour and height functions as stated in Theorem~\ref{main}.

\begin{proof}[Proof of Theorem~\ref{main}]
Let $\rho = v^n_0 < v^n_1 < \cdots < v^n_{\#V(\mathscr{T}^n)}$ be the enumeration of all vertices of $\mathscr{T}^n$ in lexicographic order. This is also the order in which the vertices are visited during the depth-first traversal of $\mathscr{T}^n$. For $i = 0, \ldots, \#V(\mathscr{T}^n)$, let $\tau^n_i$ denote the first hitting time of $v^n_i$ in this traversal.

Recall the definition of the contour function from Section~\ref{sec:geometric}. It is immediate that 
\[
H_i(\mathscr{T}^n) = C_{\tau^n_i}(\mathscr{T}^n).
\]
Moreover, by \cite[Lemma~2]{marckert2003depth}, we have
\[
\tau^n_i + H_i(\mathscr{T}^n) = 2i.
\]
Hence, for $0 \le t \le 1$,
\begin{align*}
\left|\frac{C_{\lfloor 2n^2 t\rfloor}(\mathscr{T}^n)}{n}
- \frac{H_{\lfloor n^2 t\rfloor}(\mathscr{T}^n)}{n}\right|
&= \left|\frac{C_{\lfloor 2n^2 t\rfloor}(\mathscr{T}^n)}{n}
- \frac{C_{\tau^n_{\lfloor n^2 t\rfloor}}(\mathscr{T}^n)}{n}\right| \\
&= \left|\frac{C_{\lfloor 2n^2 t\rfloor}(\mathscr{T}^n)}{n}
- \frac{C_{2\lfloor n^2 t\rfloor - H_{\lfloor n^2 t\rfloor}(\mathscr{T}^n)}(\mathscr{T}^n)}{n}\right| \\
&\overset{(*)}{\le} 
\sup_{\big(t-\frac{h(q_n)+1}{2n^2}\big)\vee 0 \le s \le t}
\left|\frac{C_{\lfloor 2n^2 t\rfloor}(\mathscr{T}^n)}{n}
- \frac{C_{\lfloor 2n^2 s\rfloor}(\mathscr{T}^n)}{n}\right|,
\end{align*}
where $(*)$ follows from the inequality
\[
\lfloor 2n^2 t\rfloor - 
\big(2\lfloor n^2 t\rfloor - H_{\lfloor n^2 t\rfloor}(\mathscr{T}^n)\big)
\le h(q_n) + 1.
\]
Then, by the tightness of the rescaled contour functions  
\( \big\{n^{-1} C_{\lfloor 2n^2 \cdot \rfloor}(\mathscr{T}^n) : n \ge 1\big\} \)
and the fact that $h(q_n)/n \to h(q)$ as $n \to \infty$, we obtain  
\[
\sup_{t \ge 0} \Big| n^{-1} C_{\lfloor 2n^2 t \rfloor}(\mathscr{T}^n)
- n^{-1} H_{\lfloor n^2 t \rfloor}(\mathscr{T}^n) \Big|\ 
\overset{d}{\longrightarrow}\  0.
\]
Combining this with the convergence of the contour functions, we conclude the joint convergence.     
\end{proof}


\appendix
\renewcommand{\thesection}{\Alph{section}}
\numberwithin{equation}{section}
\section{Near-sharpness of \eqref{H2}: a counterexample}\label{sec:example}
Fix $\alpha\in (0,1)$. For all $n\in\N$, denote by $r_n=\lfloor \frac{n}{(\log n)^\alpha}\rfloor$ and $p_n:=\frac{(\log n)^{2\alpha}}{n}$. Let $\nu=(\nu_1,\cdots,\nu_n)$ be a random exchangeable vector distributed as follows:
\begin{itemize}
    \item With probability $1-p_n$, set $\nu_i=1$ for all $1\le i\le n$;
    
    \item With probability $p_n$, choose an index $i_0$ uniformly from $\{1, \dots, n\}$ and set $\nu_{i_0}=r_n$. Then choose $r_n - 1$ distinct indices $i_1,\cdots,i_{r_n-1}$ uniformly from $\{1, \dots, n\} \setminus \{i_0\}$ and set $\nu_{i_1}=\cdots=\nu_{i_{r_n-1}}=0$. Set $\nu_j = 1$ for all remaining indices $j$.
\end{itemize}
In this construction, $\sum_{i=1}^n \nu_i =n$. As $n\to +\infty$, we have
\[
\e[\nu_1]=1,\quad \e[(\nu_1)_2] \to 1,\quad \e[(\nu_1)_3]\sim\frac{n}{(\log n)^\alpha}.
\]

Let $(q_n(s))_{0\le s\le h_n}$ be the generation-size profile defined by
\[h(q_n) = n+1,\qquad q_n(s) = n,\ \text{ for }1\le s\le n.\]
Let $(\nu^s:1\le s\le n)$ be i.i.d. copies of $\nu$. Consider the Cannings tree $\T^n$ associated with the generation-size profile $q_n$ and offspring vectors $(\nu^s)$. As before, let $\widetilde{\T}^n$ denote the geometric tree obtained by assigning length $1/n$ to each edge. Since $\e_n[(\nu_1)^3]/n\to 0$, the finite-dimensional convergence of the tree sequence still holds. However, the contour function fails to converge because tightness breaks down:
\begin{theorem}\label{thm:example}
Suppose $0<\alpha<1$. Then the sequence  $(\widetilde \T^n)$ is not tight. Consequently, the convergence stated in Theorem~\ref{main} fails.
\end{theorem}

The proof relies on the failure of the CDFI property. Recall the coalescent process $X_j$ defined in Definition \ref{def:coal_process} with $h^* = n+1$. In particular, $X_1$ represents the number of ancestors at height 1 of the entire population at height $n+1$.

\begin{lemma}\label{lem:CDFI}
 If the sequence $(\widetilde{\T}^n)$ is tight, then
\[
\lim_{M\to\infty} \limsup_{n\to\infty} \p_n(X_1>M) =0
\]
\end{lemma}
\begin{proof}
Recall the definition of $\Delta(n,k)$ in \eqref{def:Delta}. Tightness implies that for any $\delta > 0$,
\begin{align}\label{eq:tightness2}
    \lim_{k\to\infty} \limsup_{n\to \infty} \p_n\left(\max_{v\in V(\T^n)} d(v,\{V^n_1,\cdots,V^n_k\})>\delta n\right) =0,
\end{align}
where $\{V^n_{i}:1\le i\le k\}$ is a uniformly chosen 
$k$-tuple of distinct vertices.

On the event 
\[\left\{\max_{v\in V(T^n)} d(v,\{V^n_1,\cdots,V^n_k\})\le \frac12 n\right\},\] 
every vertex at height $n$ must share a common ancestor at height $\lfloor\frac12 n\rfloor$ with some $V^n_{l}$. Consequently, the number of ancestors at generation $1$ is bounded by $k$. Hence,
\[\mathbb{P}_n\left(\max_{v\in V(T^n)} d(v,\{V^n_1,\cdots,V^n_k\})\le \frac12 n\right)\le \mathbb{P}_n\left(X_1\le k\right).\]
The desired assertion follows from \eqref{eq:tightness2}.
\end{proof}
 
We now turn to the proof of Theorem~\ref{thm:example}.
In view of Lemma \ref{lem:CDFI}, it suffices to show that
there exists $\varepsilon_0 > 0$ such that for all $k \ge 1$,
\begin{align}\label{eq:eps0}
    \limsup_{n \to \infty} \mathbb{P}_n(X_1 > k) \ge \varepsilon_0.
\end{align}

The main idea is the following. In our construction, most generations are ``trivial’’ (every individual has exactly one offspring), so the tree behaves like a tall comb. Only in a small number of randomly chosen generations does one individual produce a large family of size $r_n$. We trace the first $k$ individuals at height $n+1$, and show that with high probability, they avoid merging into any of these large families. Thus they have $k$ distinct ancestors at height~1, which yields \eqref{eq:eps0}.

\begin{proof}[Proof of Theorem~\ref{thm:example}]
Let 
\[S := \#\big\{1 \le s \le n : \nu^s \neq (\overbrace{1,1,\ldots,1}^n)\big\}\] 
be the number of generations with non-trivial reproduction. Let $h_1 > \dots > h_S$ denote these ``active'' generations. For each $j \in \{1, \dots, S\}$, let $m_j$ be the unique index such that $\nu_{m_j}^{h_j} = r_n$. Conditionally on $S$, $\{h_j\}$ is a uniformly chosen $S$-tuple of $\{1, \dots, n\}$, and $\{m_j\}$ are i.i.d. uniform on $\{1, \dots, n\}$. Fix $k_0 = k$. For $1 \le i \le S$, define recursively:
$$k_i := \inf \left\{ j < m_i : \#\{1 \le l \le j : \nu_l^{h_i} = 1 \} = k_{i-1} \right\},$$
with $\inf \emptyset = \infty$. 
See Figure~\ref{fig:example} for a description.

\begin{figure}
    \centering
    \includegraphics[width=0.5\linewidth]{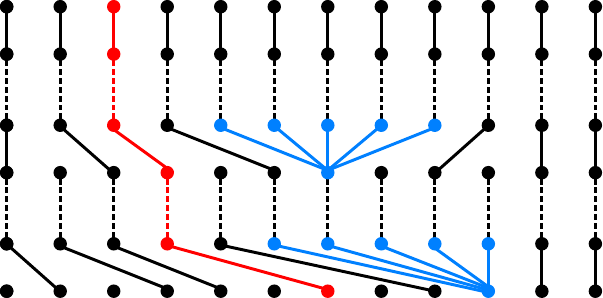}
    \caption{An example of the construction for $n=12$, $\alpha=0.9$ (so $r_n=5$), and $k=3$.
Here there are $S=2$ active generations, occurring at indices $m_1=7$ and $m_2=10$.
Starting from $k_0=3$, the recursive procedure yields $k_1=4$ and $k_2=7$.}
    \label{fig:example}
\end{figure}

Let us start with a simple observation.\phantom\qedhere
\begin{lemma}\label{lem:kS}
If $k_S < \infty$, then $X_1 \ge k$.
\end{lemma}
\begin{proof}
Suppose $k_S < \infty$, which implies that $k_i<m_i$ for all $1\le i\le S$. Then for every active generation $h_i$, the first $k$ individuals at height $n+1$ all trace their lineages to one of the $k_{i-1}$ size-$1$ parents rather than to the unique individual with offspring number $r_n$. Consequently, their ancestral lineages remain distinct throughout the sequence of active generations $\{h_i\}_{i=1}^S$, and hence they have $k$ distinct ancestors at generation $1$. In particular,
\(X_1 \ge k\).
\end{proof}
Recall \eqref{eq:eps0}. By Lemma \ref{lem:kS}, it suffices to prove that there exists $\varepsilon_0>0$ such that
\begin{align}\label{eq:k_S}
\limsup_{n\to\infty} \p_n(k_S<\infty) \ge \varepsilon_0.
\end{align}
Denote
\(m_* := \min \{m_i : 1 \le i \le S\}\).
We start with a preliminary bound on $S$ and $m_*$.
\begin{lemma}[Bounds on $S$ and $m_*$]\label{lem:S_m*_bound}
    For any $\varepsilon\in(0,1)$, we have for $n$ large,
    \begin{align}\label{eq:S_m*_bound}
        \p_n\bigg(m_*> \frac{\varepsilon n}{8(\log n)^{2\alpha}},\, S<3(\log n)^{2\alpha}\bigg)\ge 1-\varepsilon.
    \end{align}
\end{lemma}

\begin{proof}
    Note that $S \sim \text{Binomial}(n, p_n)$. By the Chernoff bound, there exists $c > 0$ such that
\begin{equation}\label{eq:s_estimate}
    \mathbb{P}_n(S \ge 3(\log n)^{2\alpha}) \le \exp(-c(\log n)^{2\alpha}).
    \end{equation}
On the other hand, it is seen that conditionally on $S$, $(m_1,\ldots, m_S)$ are i.i.d. uniform random variables on $\{1,\cdots,n\}$. Therefore
    \begin{align}\label{eq:m* estimate} 
    \begin{aligned}
    &\p_n\bigg(m_*\le \frac{\varepsilon n}{8(\log n)^{2\alpha}}\,\bigg|\, S<3(\log n)^{2\alpha}\bigg) 
    \le  1 - \bigg(1- \frac{\varepsilon}{8(\log n)^{2\alpha}}\bigg)^{\lfloor 3(\log n)^{2\alpha}\rfloor} \le \frac{\varepsilon}{2}.
    \end{aligned}
\end{align}
Combining \eqref{eq:s_estimate} and \eqref{eq:m* estimate} yields \eqref{eq:S_m*_bound}.
\end{proof}

Henceforth, we fix a small $\varepsilon\in (0,1)$ and 
define the conditional probability measure
\[\widetilde{\mathbb{P}}_n(\cdot):=\mathbb{P}_n\Big(\,\cdot\,\Big|\,m_* > \frac{\varepsilon n}{8(\log n)^{2\alpha}}, \, S < 3(\log n)^{2\alpha}\Big).\]
The proof of \eqref{eq:k_S} is accomplished by the following two propositions.
\begin{proposition}[Growth of $k_i$]\label{prop:grwoth_ki}
    There exists $c=c(\varepsilon)>0$ such that for all $n,i\ge 1$,
\begin{equation}\label{eq:k_i estimate}
\widetilde{\p}_n\bigg( k_{i+1}> k_i\Big( 1+ \frac{2}{(\log n)^\alpha}\Big) \,\bigg|\, k_i\bigg) \mathbf{1}\left\{i<S,\,k_i < \frac{\varepsilon n}{16(\log n)^{2\alpha}}\right\}
\le \exp\Big(-\frac{c k_i}{(\log n)^{2\alpha}}\Big).
\end{equation}
\end{proposition}
\begin{proposition}[Size of $k_i$ at the $\sqrt{n}$ threshold]\label{prop:threshold_size}
    Let $i_* := \inf \{1 \le i \le S : k_i > \sqrt{n}\}$. Then there exists $c=c(\varepsilon)>0$ such that for all $n\ge 1$,
    \begin{align}\label{eq:i* estimate}
\widetilde{\p}_n\bigg( k_{i_*}> \sqrt{n}\Big( 1+ \frac{2}{(\log n)^\alpha}\Big) \,\bigg|\, i_*\le S\bigg) \le \exp\Big(-\frac{c \sqrt{n}}{(\log n)^{2\alpha}}\Big).
\end{align}
\end{proposition}
\begin{proof}[Proof of \eqref{eq:k_S} assuming Propositions~\ref{prop:grwoth_ki} and \ref{prop:threshold_size}]
We first note that if \(i_*=\infty\), then necessarily \(k_S\le \sqrt{n}<m_*\), and hence \(k_S<\infty\).
It therefore suffices to treat the case \(i_*\le S\).

Define the event
\[
A=A_n:=\bigg\{\, i_*\le S,\quad k_{i_*}\le \sqrt{n}\Big(1+\frac{2}{(\log n)^\alpha}\Big),\quad 
k_{j+1}\le k_j\Big(1+\frac{2}{(\log n)^\alpha}\Big)\ \ \forall\, i_*\le j<S \bigg\}.
\]
On the event \(A\), we have \(\widetilde{\mathbb{P}}_n\)-a.s. that
\[
k_S\le \sqrt{n}\Big(1+\frac{2}{(\log n)^\alpha}\Big)^{S+1-i_*}
 \le \sqrt{n}\Big(1+\frac{2}{(\log n)^\alpha}\Big)^{3(\log n)^{2\alpha}}
 < m_*,
\]
for all sufficiently large \(n\), since \(\alpha<1\). Consequently,
\begin{equation}\label{eq:k_S_lower}
\widetilde{\mathbb{P}}_n(k_S<m_*,\, i_*\le S)\ge \widetilde{\mathbb{P}}_n(A).
\end{equation}

Combining \eqref{eq:S_m*_bound}, \eqref{eq:k_i estimate} and \eqref{eq:i* estimate}, we obtain that for any \(\varepsilon\in(0,1)\) there exists a constant \(c=c(\varepsilon)>0\) such that, for all sufficiently large \(n\),
\[
\widetilde{\mathbb{P}}_n(A\,|\,i_*\le S)\ge (1-\varepsilon)\Bigg(1-3(\log n)^{2\alpha}\exp\Big(-\frac{c \sqrt{n}}{(\log n)^{2\alpha}}\Big)
- \exp\Big(-\frac{c \sqrt{n}}{(\log n)^{2\alpha}}\Big)\Bigg).
\]
Substituting this bound into \eqref{eq:k_S_lower} and accounting for the case \(i_*=\infty\) yields \eqref{eq:k_S}.
\end{proof}

Finally, let us prove Propositions \ref{prop:grwoth_ki} and \ref{prop:threshold_size}.
\begin{proof}[Proof of Proposition \ref{prop:grwoth_ki}]
Note that conditionally on $S$ and $(h_i,m_i:1\le i\le S)$, for each $1\le i\le S$ and $1\le j\le m_i$, the count
    \[\#\{1 \le l \le j : \nu_l^{h_i} = 1\}\]
has the same law as the Hypergeometric$(n,n-r_n, j)$ random variable; that is, the number of successes in $j$ draws, without replacement, from a finite population of size $n$ that contains exactly $n-r_n$ objects with that feature. Therefore,
\begin{align*}
    \text{LHS of \eqref{eq:k_i estimate}} \;&\le\; \mathbb{P}\!\left( \mathrm{Hypergeometric}\!\left(n,n-r_n,\big\lfloor k_i(1+2(\log n)^{-\alpha})\big\rfloor\right)\ge k_i \right).
\end{align*}
\eqref{eq:k_i estimate} then follows by standard hypergeometric tail bounds (e.g., \cite[Corollary 1.1]{Serfling}).
\end{proof}
\begin{proof}[Proof of Proposition \ref{prop:threshold_size}]
    We first rewrite
\begin{equation*}
\begin{aligned}
\widetilde{\p}_n\bigg( k_{i_*}> \sqrt{n}\Big( 1+ \frac{2}{(\log n)^\alpha}\Big) \,\bigg|\, i_*\le S\bigg)
=\, \frac{\sum_{i=1}^{\infty}\widetilde{\p}_n(k_i> \sqrt{n}( 1+ \frac{2}{(\log n)^\alpha}),\, k_{i-1}\le \sqrt{n},\, i\le S)}{\sum_{i=1}^{\infty} \widetilde{\p}_n(k_i>\sqrt{n},\, k_{i-1}\le \sqrt{n},\, i\le S)}.
\end{aligned}
\end{equation*}
It suffices to prove that for any $i\ge 1$ and sufficiently large $n$,
\begin{align}\label{eq:i* total}
    \frac{\widetilde{\p}_n\bigg(k_i\ge \sqrt{n}\Big( 1+ \frac{2}{(\log n)^\alpha}\Big) \,\Big|\,k_{i-1}\le \sqrt{n},\, i\le S\bigg)}{\widetilde{\p}_n\Big(k_i\ge \sqrt{n} \,\Big|\,k_{i-1}\le \sqrt{n},\, i\le S\Big)}\le \left(\frac45\right)^{\frac{\sqrt{n}}{2(\log n)^\alpha}}.
\end{align}

Observe that for any $i\ge 1$, on the event
\[\left\{i\le S,\, k_{i-1}\le \sqrt{n},\, m_*>\frac{\varepsilon n}{8(\log n)^{2\alpha}}\right\},\]
for any $\sqrt{n}\le l\le \sqrt{n}( 1+ \frac{2}{(\log n)^\alpha})$, the event $\{k_i=\ell\}$ occurs if and only if $\nu^{h_i}_l=1$ and exactly \(k_{i-1}-1\) entries among
\((\nu^{h_i}_j : 1\le j\le \ell-1)\) equal \(1\). Thus, for any $1\le l'\le \sqrt{n}$ and $i\ge 1$,
\[
\widetilde{\p}_n(k_i= l \,|\, k_{i-1}=l',\, i\le S) = \frac{\binom{l-1}{l'-1}\binom{n-l}{n-l'-r_n}}{\binom{n}{n-r_n}}=:b_{l,l'}.
\]
 A direct computation gives
\[
\frac{b_{\ell+1,\ell'}}{b_{\ell,\ell'}}
=
\frac{\ell}{\ell-\ell'+1}\cdot 
\frac{\ell'+r_n-\ell}{\,n-\ell\,},
\qquad
\ell'\le \ell\le ( \ell'+r_n )\wedge (n-1),
\]
and this ratio is decreasing in \(\ell\) and increasing in \(\ell'\).
In particular, uniformly for 
\(\ell\ge \sqrt{n}(1+\tfrac{3}{2}(\log n)^{-\alpha})\) and 
\(\ell'\le\sqrt{n}\),
\[
\frac{b_{\ell+1,\ell'}}{b_{\ell,\ell'}}\le \frac23 + o_n(1).
\]
This immediately leads to \eqref{eq:i* total}.
\end{proof}
\end{proof}

\begin{funding}
    XL is supported by NSFC, China (No.\ 12301184 and No.\ 12371144). YZ is supported by China Postdoctoral Science Foundation (No.\ 2023M743721 and No.\ 2025T180850).
\end{funding}
\begin{acks}[Acknowledgments]
    The authors are grateful to Elie A\"id\'ekon for suggesting the research problem that motivated this work.
\end{acks}

\bibliographystyle{imsart-number}
\bibliography{Cannings}

\end{document}